\documentclass[12pt]{amsart}

\usepackage{amsmath}
\usepackage{amsbsy}
\usepackage{amssymb}
\usepackage{amscd}
\usepackage{eq}
\usepackage[dvips]{graphicx, color}

\newcommand{\bbc}{{\mathbb C}}

\newcommand{\bbq}{{\mathbb Q}}
\newcommand{\bbr}{{\mathbb R}}

\newcommand{\bbz}{{\mathbb Z}}

\newcommand{\be}{{\beta}}

\newcommand{\lam}{{\lambda}}
\newcommand{\Lam}{{\Lambda}}

\newcommand{\sig}{{\sigma}}

\newcommand{\gB}{{\mathfrak B}}

\newcommand{\gt}{{\mathfrak t}}

\newcommand{\ch}{{\operatorname{ch}}\,}

\newcommand{\aff}{{\operatorname {Aff}}}

\newcommand{\m}{{\operatorname {M}}}

\newcommand{\gl}{{\operatorname{GL}}}

\newcommand{\spl}{{\operatorname{SL}}}

\newcommand{\sst}{{\operatorname{ss}}}

\newcommand{\rep}{representation}

\newcommand{\pv}{prehomogeneous vector space}

\newcommand{\Z}{\bbz}
\newcommand{\Q}{\bbq}
\newcommand{\R}{\bbr}
\newcommand{\C}{\bbc}

\newcommand{\mk}{k^{\times}}

\newcommand{\kableadd}%
{Department of Mathematics\\ Cornell University\\
Ithaca NY 14853}

\newcommand{\sub}{\subset}

\newcommand{\ccd}{,\ldots,}

\newcommand{\lan}{\langle}
\newcommand{\ran}{\rangle}

\makeatletter
\def\varddots{\mathinner{
\mkern1mu%
 \raise\p@\hbox{.}\mkern2mu%
 \raise4\p@\hbox{.}\mkern2mu%
 \raise7\p@\vbox{\kern7\p@\hbox{.}}%
\mkern1mu}}
\makeatother

\theoremstyle{plain}
\newtheorem{thm}{Theorem}[section]
\newtheorem{lem}[thm]{Lemma}

\newtheorem{prop}[thm]{Proposition}

\theoremstyle{definition}

\theoremstyle{remark}

\usepackage{mathrsfs,bbm}

\newcommand{\coorde}{{\mathbbm e}}
\newcommand{\coordf}{{\mathbbm f}}
\newcommand{\diag}{{\mathrm{diag}}}
\newcommand{\size}{\tiny}
\newcommand{\Ex}{\mathrm{Ex}}

\begin{document}

\address[K. Tajima]
{National Institute of Technology, Sendai College, Natori Campus, 
48 Nodayama, Medeshima-Shiote, Natori-shi, Miyagi, 981-1239, Japan}
\email{kndt2147@yahoo.co.jp}

\address[A. Yukie]
{Department of Mathematics, Graduate School of Science, 
Kyoto University, Kyoto 606-8502, Japan}
\email{yukie@math.kyoto-u.ac.jp}
\thanks{The second author was partially supported by 
Grant-in-Aid (C) (17K05169)\\}


\title{On the GIT stratification of prehomogeneous vector spaces II}
\author{Kazuaki Tajima}
\author{Akihiko Yukie}
\maketitle

\begin{abstract}
We determine all orbits of two \pv s rationally over an arbitrary 
perfect field in this paper.  
\end{abstract}

\section{Introduction}
\label{sec:introduction}

This is part two of a series of four papers. 
In Part I, we determined the set $\gB$ of vectors 
which parametrizes the GIT stratification \cite{tajima-yukie} 
of the four \pv s (1)--(4) in \cite{tajima-yukie-GIT1}. 
Even though the set $\gB$ was determined, the corresponding stratum 
$S_{\be}$ may be the empty set. 

In this part, we consider the following two \pv s:

(1) $G=\gl_3\times \gl_3\times \gl_2$, 
$V=\aff^3\otimes \aff^3\otimes \aff^2$,

(2) $G=\gl_6\times \gl_2$, 
$V=\wedge^2 \aff^6\otimes \aff^2$.

For a general introduction to this series of papers, 
see the introduction of \cite{tajima-yukie-GIT1}. 

Throughout this paper, $k$ is a fixed perfect field. 
Let $\Ex_2(k)$ be the set of isomorphism classes of 
separable extensions of $k$ of degree up to two. 
The following theorems are our main theorems in this part.

\begin{thm}
\label{thm:main1}
For the \pv{} (1), there are $16$ non-empty strata $S_{\be}$. 
Moreover, except for one stratum $S_{\be_0}$, the set 
$G_k\backslash S_{\be\,k}$
consists of a single point whereas the set $G_k\backslash S_{\be_0\,k}$ 
is in bijective correspondence with $\Ex_2(k)$. 
\end{thm}
\begin{thm}
\label{thm:main2}
For the \pv{} (2), there are $13$ non-empty strata $S_{\be}$. 
Moreover, except for one stratum $S_{\be_0}$, the set 
$G_k\backslash S_{\be\,k}$
consists of a single point whereas the set $G_k\backslash S_{\be_0\,k}$ 
is in bijective correspondence with $\Ex_2(k)$. 
\end{thm}

We need preparations to make the above statements more precise. 
We shall state more precise theorems after some 
preparations in later sections. 
For the case (1) (resp. (2)), see Theorem \ref{thm:main1-detail} 
(resp. Theorem \ref{thm:main2-detail}). 

As we pointed out in Part I, the orbit decomposition of 
these \pv s is known over $\C$ (see \cite[pp.385--387]{kimu}, 
Proposition 2.2 \cite[pp.456,457]{kimura-orbits}). 
Our approach answers to rationality questions and provide 
the inductive structure of orbits rationally over $k$. 


\section{Notation}
\label{sec:notation}

We discuss notations used in this part.   

We denote the characteristic of $k$ by $\ch(k)$. 
We often have to refer to a set consisting of a single point. 
We write SP for such a set. If $V$ is a vector space then 
we denote the dual space of $V$ by $V^*$. 

Let $\gl_n$ (resp. $\spl_n$) be the general linear group 
(resp. special linear group) of matrices of size $n$, 
$\m_{n,m}$ the space of $n\times m$ matrices and $\m_n=\m_{n,n}$. 
We express the set of $k$-rational points by 
$\gl_n(k)$, etc. We sometimes use 
the notation $[x_1\ccd x_n]$ to express column vectors 
to save space.  We denote the unit matrix of dimension $n$ 
by $I_n$.  We use the notation $\diag(g_1\ccd g_m)$ 
for the block diagonal matrix whose diagonal blocks are 
$g_1\ccd g_m$.  For $u=(u_{ij})\in \aff^{n(n-1)/2}$ ($1\leq j<i\leq n$), 
let $n_n(u)$ be the lower triangular matrix whose diagonal entries are $1$ 
and the $(i,j)$-entry is $u_{ij}$ for $i>j$. 
If $v_1\ccd v_m$ are elements of a vector space $V$ then 
$\lan v_1\ccd v_m\ran$ is the subspace spanned by 
$v_1\ccd v_m$. 

For the rest of this paper, 
tensor products are always over $k$.

If $n=2$, we define a map $\theta:\m_2\to \m_2$ by 
\begin{equation}
\label{eq:theta-defn}
\m_2\ni 
A = \begin{pmatrix}
a & b \\
c & d 
\end{pmatrix}
\mapsto \theta(A) = 
\begin{pmatrix}
d & -c \\
-b & a 
\end{pmatrix} 
\in \m_2. 
\end{equation}
Since $\theta(g)= (\det g){}^t g^{-1}$ for $g\in\gl_2$, 
$\theta$ induces an automorphism of $\gl_2$.  
Since $\gl_2\sub\m_2$ is Zariski dense, 
$\theta(AB)=\theta(A)\theta(B)$ for $A,B\in\m_2$. 
Since $\theta$ is clearly additive and bijective, 
it is an automorphism of the $k$-algebra $\m_2$. 

For $x=[x_1,x_2],y=[y_1,y_2]\in\aff^2$ (these are column vectors), 
we define 
\begin{equation}
\label{eq:invariant-bilinear-aff2}
[x,y]_{\aff^2} = {}^tx y = x_1y_1+x_2y_2. 
\end{equation}
Then it is easy to see that 
\begin{equation}
\label{eq:bilinear-gl2-property}
[\theta(g)x,gy]_{\aff^2} =[gx,\theta(g)y]_{\aff^2} = (\det g)[x,y]_{\aff^2}
\end{equation}
for $g\in\gl_2$.  By the above pairing we can identify $\aff^2$ with 
its dual space. Also the standard basis $\{\coorde_1,\coorde_2\}$ 
can be identified with the dual basis.  The relation 
(\ref{eq:bilinear-gl2-property}) shows that $\det g$ times the 
contragredient \rep{} can be identified with the action of $\theta(g)$. 
Sometimes the involution   
\begin{math}
g\mapsto 
\left(
\begin{smallmatrix}
0 & 1 \\ 
1 & 0 
\end{smallmatrix}
\right) {}^t g^{-1}
\left(
\begin{smallmatrix}
0 & 1 \\ 
1 & 0 
\end{smallmatrix}
\right)
\end{math}
is used so that lower triangular matrices map to 
lower triangular matrices, but we do not need that 
property in this part. 

If $\chi_1,\chi_2$ are characters of an algebraic group, 
they are said to be {\it proportional} if there exist 
positive integers $a,b$ such that $\chi_1^a=\chi_2^b$.  
If $\rho:G\to \gl(V)$ is a \rep{} of an algebraic group 
($\dim V<\infty$) then 
$\rho^*:G\to \gl(V^*)$ is the contragredient \rep{} of $\rho$
($V^*$ is the dual space).  Note that if $V_1,V_2$ 
are (finite dimensional) \rep s of $G$ then 
$(V_1\otimes V_2)^*\cong V_1^*\otimes V_2^*$.  
By (\ref{eq:bilinear-gl2-property}), the standard \rep{} 
of $\spl_2$ is equivalent to its contragredient \rep.

Suppose that $G$ is in the form 
$\gl_{n_1}\times \cdots \times \gl_{n_a}$. 
For the \pv s (1), (2), $a=3,2$.  
We use parabolic subgroups which consist 
of lower triangular  blocks. 
Let $i=1\ccd a$ and $j_0=0<j_1<\cdots<j_{N_i}=n_i$. 
We use the notation $P_{i,[j_{i1}\ccd j_{i\,N_i-1}]}$ 
(resp. $M_{i,[j_{i1}\ccd j_{i\,N_i-1}]}$)
for the parabolic subgroup (resp. reductive subgroup) of $\gl_{n_i}$ 
in the form 
\begin{equation*}
\begin{pmatrix} 
P_{11} & 0 \cdots 0 & 0 \\
\vdots & \vdots & 
\begin{matrix}
0 \\ \vdots \\ 0
\end{matrix}  \\ 
P_{N_i\,1} & \cdots & P_{N_i\,N_i} 
\end{pmatrix}, \quad
\begin{pmatrix} 
M_{11} & 0 \cdots 0 & 0 \\
\begin{matrix}
0 \\ \vdots \\ 0
\end{matrix}
& \ddots & 
\begin{matrix}
0 \\ \vdots \\ 0
\end{matrix}  \\ 
0 & 0 \cdots 0 & M_{N_i\,N_i} 
\end{pmatrix}
\end{equation*}
where the size of $P_{kl},M_{kl}$ is 
$(j_{ik}-j_{i\,k-1})\times (j_{il}-j_{i\,l-1})$. 
If $N_i=1$ then we use the notation $P_{i,\emptyset},M_{i,\emptyset}$.

We put 
\begin{equation}
\label{eq:parabolic-defn}
\begin{aligned} 
P_{[j_{11}\ccd j_{1\,N_1-1}]\ccd [j_{a1}\ccd j_{a\,N_a-1}]}
& = P_{1,[j_{11}\ccd j_{1\,N_1-1}]}
\times \cdots \times P_{a,[j_{a1}\ccd j_{a\,N_a-1}]}, \\
M_{[j_{11}\ccd j_{1\,N_1-1}]\ccd [j_{a1}\ccd j_{a\,N_a-1}]}
& = M_{1,[j_{11}\ccd j_{1\,N_1-1}]}
\times \cdots \times M_{a,[j_{a1}\ccd j_{a\,N_a-1}]}.
\end{aligned}
\end{equation}
If $N_i=1$ then we replace $[j_{i1}\ccd j_{i\,N_i-1}]$ by 
$\emptyset$. 
Let 
\begin{equation}
\label{eq:M1-defn}
\begin{aligned}
& M^1_{[j_{11}\ccd j_{1\,N_1-1}]\ccd [j_{a1}\ccd j_{a\,N_a-1}]} \\
& = (\spl_{n_1}\times \cdots \times \spl_{n_a})\cap
M_{1,[j_{11}\ccd j_{1\,N_1-1}]}
\times \cdots \times M_{a,[j_{a1}\ccd j_{a\,N_a-1}]}
\end{aligned}
\end{equation}
and $M^s_{[j_{11}\ccd j_{1\,N_1-1}]\ccd [j_{a1}\ccd j_{a\,N_a-1}]}$ 
be the semi-simple part of 
$M_{[j_{11}\ccd j_{1\,N_1-1}]\ccd [j_{a1}\ccd j_{a\,N_a-1}]}$.

We consider many \rep s of groups of the 
form $M_{[j_{11}\ccd j_{1\,N_1-1}]\ccd [j_{a1}\ccd j_{a\,N_a-1}]}$
in later sections. 
We use notations such as 
\begin{equation}
\label{eq:Lam-defn}
\Lam^{m,i}_{j,[c,d]}.
\end{equation}
The meaning of this notation is that this is 
$\wedge^i \aff^m$ as a vector space where 
$\aff^m$ is the standard \rep{} of $\gl_m$ 
and the indices $j,[c,d]$ mean that 
the block from the $(c,c)$-entry to 
the $(d,d)$-entry of the $j$-th factor $\gl_{n_j}$ 
of $M_{[j_{11}\ccd j_{1\,N_1-1}]\ccd [j_{a1}\ccd j_{a\,N_a-1}]}$.
is acting on this vector space. 
For example, if $a=2$, $n_1=3,n_2=3$, $N_1=1,N_2=2$ then 
$M_{[1],[2]}$ consists of elements of the form 
$(\diag(t_1,g_1),\diag(g_2,t_2))$ where $t_1,t_2\in\gl_1,g_1,g_2\in\gl_2$. 
Then $\Lam^{2,1}_{2,[1,2]}$ is the standard \rep{} of $g_2\in \gl_2$
identified with the element $(I_3,\diag(g_2,1))$.   
The trivial \rep{} of $M_{[j_{11}\ccd j_{1\,N_1-1}]\ccd [j_{a1}\ccd j_{a\,N_a-1}]}$
is denoted by $1$. 

Since $V$ of \pv s (1), (2) has a ``scalar direction'', 
it is natural to remove scalar directions from $G$ to measure
stability in the sense of geometric invariant theory. 
In \cite{tajima-yukie}, we considered a certain subgroup 
$G_1\sub G$ for that purpose. However, since our group here
is in the form $G=G_1\times G_2\times G_3$ or 
$G=G_1\times G_2$, we use the notation $G_{\text{st}}$ 
(``st'' stands for ``stability'') for the group $G_1$ in 
\cite{tajima-yukie}. 

For the \pv{} (1) (resp. (2)), we choose 
\begin{equation*}
T_0=\{(t_1I_3,t_2I_3,t_3I_2)\mid t_1,t_2,t_3\in\gl_1\} 
\; (\text{resp.} \; T_0=\{(t_1I_6,t_2I_2)\mid t_1,t_2\in\gl_1\})
\end{equation*}
(the center of $G$) and $G_{\text{st}}=\spl_3\times \spl_3\times \spl_2$
(resp. $G_{\text{st}}=\spl_6\times \spl_2$). 

We use other notations such as 
\begin{equation*}
\gt^*,\gt^*_+,\gt^*_{\Q},(\;,\;)_*,M_{\be}, 
P_{\be},U_{\be},Z_{\be},W_{\be},Y_{\be},S_{\be},\lam_{\be},\chi_{\be}
\end{equation*}
of \cite{tajima-yukie} and Section 2 of \cite{tajima-yukie-GIT1}. 
There is a slight ambiguity on the domain of definition of $\chi_{\be}$
in \cite{tajima-yukie}. In this paper, $\chi_{\be}$ is an indivisible character 
on $M_{\be}$, proportion to $\be$. If $M_{\be}$ is in the form 
(\ref{eq:parabolic-defn}), $M^1_{\be}$ is defined to be the 
group (\ref{eq:M1-defn}). Note that $M^1_{\beta}=M_{\be}\cap G_{\text{st}}$. 
Let $M^s_{\be}$ be the semi-simple part of $M_{\be}$. 
Let $G_{\text{st},\beta}=\{g\in M^1_{\beta}\mid \chi_{\be}(g)=1\}^{\circ}$
(the identity component). This $G_{\text{st},\beta}$ is $G^1_{\be}$ of 
\cite{tajima-yukie}. 

The space $\gt^*$ is 
\begin{align*}
& \left\{(a_{11},a_{12},a_{13},a_{21},a_{22},a_{23},a_{31},a_{32})\in\R^8\,
\,\vrule\, \, \sum_{i=1}^3 a_{1i}=
\sum_{i=1}^3 a_{2i}=\sum_{i=1}^2 a_{3i}=0\right\}, \\[5pt]
& \left\{(a_{11},a_{12},a_{13},a_{14},a_{15},a_{16},a_{21},a_{22})\in\R^8\,
\,\vrule\, \, \sum_{i=1}^6 a_{1i}=\sum_{i=1}^2 a_{2i}=0\right\}
\end{align*}
for the cases (1), (2) respectively.

%

Let $\coorde_i$ be the coordinate vector of 
$\aff^3$ with respect to the $i$-th coordinate
and $\coordf_i$ the coordinate vector of 
$\aff^2$ with respect to the $i$-th coordinate.  
When we have to distinguish the first two factors, 
we may write $\coorde_{1,i},\coorde_{2,i}$. 
We put $e_{i_1i_2i_3}=\coorde_{i_1}\otimes \coorde_{i_2}\otimes \coordf_{i_3}$
for $i_1,i_2=1,2,3,i_3=1,2$. The numbering used in \cite{tajima-yukie-GIT1} 
for (1) is as follows. 

\vskip 10pt

\tiny

\begin{center}
 
\begin{tabular}{|c|c|c|c|c|c|c|c|c|}
\hline
1 & 2 & 3 & 4 & 5 & 6 & 7 & 8 & 9 \\
\hline
$e_{111}$ & $e_{121}$ & $e_{131}$ & $e_{211}$ & $e_{221}$ & $e_{231}$ 
& $e_{311}$ & $e_{321}$ & $e_{331}$ \\
\hline
10 & 11 & 12 & 13 & 14 & 15 & 16 & 17 & 18 \\
\hline
$e_{112}$ & $e_{122}$ & $e_{132}$ & $e_{212}$ & $e_{222}$ & $e_{232}$ 
& $e_{312}$ & $e_{322}$ & $e_{332}$ \\
\hline 
\end{tabular}

\end{center}

\normalsize

\vskip 10pt

Let $\coorde_i$ be the coordinate vector of 
$\aff^6$ with respect to the $i$-th coordinate
and $\coordf_i$ the coordinate vector of 
$\aff^2$ with respect to the $i$-th coordinate. 
We put $e_{i_1i_2,i_3}=(\coorde_{i_1}\wedge \coorde_{i_2})\otimes \coordf_{i_3}$
for $i_1,i_2=1\ccd 6,i_3=1,2$. The numbering used in \cite{tajima-yukie-GIT1} 
for (2) is as follows. 

\vskip 10pt

\tiny

\begin{center}
 
\begin{tabular}{|c|c|c|c|c|c|c|c|c|c|c|c|c|c|c|}
\hline
1 & 2 & 3 & 4 & 5 & 6 & 7 & 8 & 9 & 10 & 11 & 12 & 13 & 14 & 15 \\
\hline
$e_{121}$ & $e_{131}$ & $e_{141}$ & $e_{151}$ & $e_{161}$ 
& $e_{231}$ & $e_{241}$ & $e_{251}$ & $e_{261}$ & $e_{341}$ & $e_{351}$ 
& $e_{361}$ & $e_{451}$ & $e_{461}$ & $e_{561}$ \\
\hline
16 & 17 & 18 & 19 & 20 & 21 & 22 & 23 & 24 & 25 & 26 & 27 & 28 & 29 & 30  \\
\hline
$e_{122}$ & $e_{132}$ & $e_{142}$ & $e_{152}$ & $e_{162}$ & $e_{232}$ 
& $e_{242}$ & $e_{252}$ & $e_{262}$ & $e_{342}$ & $e_{352}$ & $e_{362}$ 
& $e_{452}$ & $e_{462}$ & $e_{562}$ \\ 
\hline 
\end{tabular}

\end{center}

\normalsize

\vskip 10pt

\section{Non-empty strata for the case (1)}
\label{sec:non-empty-strata1}

In this section and the next, we consider the 
case (1). We put $G_1=G_2=\gl_3,G_3=\gl_2$
and $G=G_1\times G_2\times G_3$. Let $G_{\text{st}},M_{\be}$, etc., 
be as in Section \ref{sec:notation}. The set $\gB$ consists 
of $49$ $\be_i$'s. We use the table in 
Section 7 \cite{tajima-yukie-GIT1}. 
We shall prove that $S_{\be_i}\not=\emptyset$ for 
\begin{equation}
\label{eq:list-non-empty}
i = 4,5,6,11,22,28,29,38,39,40,41,42,45,47,48,49
\end{equation}
for the \pv{} (1) in this section. We shall 
prove that $S_{\be_i}=\emptyset$ for other $\be_i$'s 
for the \pv{} (1) in the next section.

For the case (1), the following observation is useful. 
Suppose that 
\begin{equation*}
\be=(b_{11},b_{12},b_{13};b_{21},b_{22},b_{23};b_{31},b_{32})\in\gB.
\end{equation*}
Then 
\begin{equation*}
\sig(\be)\stackrel{\rm def}=
(b_{21},b_{22},b_{23};b_{11},b_{12},b_{13};b_{31},b_{32})\in\gB
\end{equation*}
also. It is easy to see that $S_{\be}\not=\emptyset$ 
if and only if $S_{\sig(\be)}\not=\emptyset$. Also 
$G_k\backslash Y^{\sst}_{\be\, k}$ is in bijective correspondence with 
$G_k\backslash Y^{\sst}_{\sig(\be)\, k}$. 
Therefore, we handle $\be$ and $\sig(\be)$ simultaneously. 
We list $\{\be,\sig(\be)\}$ such that $\sig(\be)\not=\be$ in 
the following. 
\begin{align*}
& \{\be_1,\be_7\},\; 
\{\be_2,\be_3\},\; 
\{\be_4,\be_5\},\;
\{\be_9,\be_{10}\},\;
\{\be_{12},\be_{13}\},\;
\{\be_{14},\be_{15}\}, \\
& \{\be_{17},\be_{23}\},\;
\{\be_{18},\be_{24}\},\;
\{\be_{19},\be_{25}\},\;
\{\be_{20},\be_{26}\},\;
\{\be_{21},\be_{27}\},\;
\{\be_{30},\be_{34}\}, \\
& \{\be_{31},\be_{35}\},\;
\{\be_{32},\be_{36}\},\;
\{\be_{33},\be_{37}\},\; 
\{\be_{40},\be_{41}\},\;
\{\be_{43},\be_{44}\},\;  
\{\be_{47},\be_{48}\}.          
\end{align*}
For other $\be_i$'s, $\sig(\be_i)=\be_i$.

What we are going to do in this section is to find a non-constant invariant 
polynomial on $Z_{\be}$ with respect to the action of $G_{\text{st},\be}$. 
Moreover, we shall describe the set of rational orbits 
$G_k\backslash S_{\be}$. Since 
$S_{\be}\cong G_k\times_{P_{\be\, k}} Y^{\sst}_{\be\, k}$, 
it is enough to describe the set  
$P_{\be\, k} \backslash Y_{\be\, k}^{\sst}$. 
It turns out that $Y_{\be\, k}^{\sst}=U_{\be\, k}Z_{\be\, k}^{\sst}$ 
and so it is enough to describe the set  
$M_{\be\, k} \backslash Z_{\be\, k}^{\sst}$. 
Note that we measure the stability with respect to 
$G_{\text{st},\be}$, but we consider the group $M_{\be}$ when 
we consider rational orbits.

Note that $\lam_{\be}$ acts on $Z_{\be}$ by scalar multiplication. 
Also stability does not change by replacing $k$ by $\overline k$. 
Therefore, it is enough to find a non-constant polynomial 
$P(x)$ and a character $\chi$ of $M^1_{\be}$ 
proportional to $\chi_{\be}$ 
such that $P(gx)=\chi(g)P(x)$ for $M^1_{\be}$, $x\in Z_{\be}$. 
Then $P(gx)=P(x)$ for $g\in G_{\text{st},\be}$.  

The following table describes $M_{\be}$, $Z_{\be}$ as a 
\rep{} of $M_{\be}$, the coordinates of $Z_{\be},W_{\be}$ 
and $G_k\backslash S_{\be\, k}\cong P_{\be\,k}\backslash Y^{\sst}_{\be\,k}$. 

\vskip 10pt

\begin{center}
 
\begin{tabular}{|c|l|l|}
\hline 
\size $i$ & $\;$ \hskip 1in \size $M_{\be_i}$ \hskip 1in $\;$
& $\;$ \hskip 0.9in \size $Z_{\be_i}$ \hskip 0.9in $\;$ \\
\hline
\size $G_k\backslash S_{\be_i\, k}$  & 
\hskip 0.5in \size
coordinates of $Z_{\be_i}$
& \size \hskip 0.4in
coordinates of $W_{\be_i}$  \\
\hline \hline 
\size 4 & \size $M_{[1],\emptyset,\emptyset}\cong \gl_3\times \gl_2^2\times \gl_1$ 
& \size $\Lam^{3,1}_{2,[1,3]}\otimes \Lam^{2,1}_{1,[2,3]}\otimes \Lam^{2,1}_{3,[1,2]}$  \\
\hline
\size SP & \size
\begin{math}
\begin{matrix}
x_{211},x_{221},x_{231},x_{311},x_{321},x_{331}, \hfill \\ 
x_{212},x_{222},x_{232},x_{312},x_{322},x_{332} \hfill 
\end{matrix}
\end{math} 
& \size \hskip 0.8in -   \\
\hline \hline 
\size 5 & \size $M_{\emptyset,[1],\emptyset}\cong \gl_3\times \gl_2^2\times \gl_1$ 
& \size $\Lam^{3,1}_{1,[1,3]}\otimes \Lam^{2,1}_{2,[2,3]}\otimes \Lam^{2,1}_{3,[1,2]}$ \\
\hline
\size SP & \size
\begin{math}
\begin{matrix}
x_{121},x_{131},x_{221},x_{231},x_{321},x_{331}, \hfill \\ 
x_{122},x_{132},x_{222},x_{232},x_{322},x_{332} \hfill 
\end{matrix}
\end{math} 
& \size \hskip 0.8in -   \\
\hline \hline 
\size 6 & \size $M_{[1],[1],[1]}\cong \gl_2^2\times \gl_1^4$ 
& \size $\Lam^{2,1}_{1,[2,3]}\otimes \Lam^{2,1}_{2,[2,3]}
\oplus \Lam^{2,1}_{2,[2,3]} \oplus \Lam^{2,1}_{1,[2,3]}$ \\
\hline
\size SP & \size
\begin{math}
\begin{matrix}
x_{221},x_{231},x_{321},x_{331}, \hfill 
x_{122},x_{132},x_{212},x_{312}  \hfill 
\end{matrix}
\end{math} 
& \size $x_{222},x_{232},x_{322},x_{332}$   \\
\hline \hline 
\size 11 & \size $M_{[1,2],[1,2],[1]}\cong \gl_1^8$ 
& \size $1^{5\oplus}$ \\
\hline
\size SP & \size
\begin{math}
\begin{matrix}
x_{231},x_{321},x_{132},x_{222},x_{312}
\end{matrix}
\end{math} 
& \size $x_{331},x_{232},x_{322},x_{332}$   \\
\hline \hline
\size 22 & \size $M_{[2],[2],\emptyset}\cong \gl_2^3\times \gl_1^2$
& \size $\Lam^{2,1}_{1,[1,2]}\otimes \Lam^{2,1}_{3,[1,2]}\oplus 
\Lam^{2,1}_{2,[1,2]}\otimes \Lam^{2,1}_{3,[1,2]}$ \\
\hline
\size SP & \size
\begin{math}
x_{131},x_{231},x_{311},x_{321},x_{132},x_{232},x_{312},x_{322} 
\end{math} 
& \size $x_{331},x_{332}$ \\
\hline \hline 
\size 28 & \size $M_{[1,2],[1,2],[1]}\cong \gl_1^8$
& \size $1^{4\oplus}$ \\
\hline
\size SP & \size $x_{331},x_{132},x_{222},x_{312}$
& \size $x_{232},x_{322},x_{332}$ \\
\hline \hline
\size 29 & \size $M_{[1],[1],\emptyset}\cong \gl_2^3\times \gl_1^2$ 
& \size $\Lam^{2,1}_{1,[2,3]}\otimes \Lam^{2,1}_{2,[2,3]}\otimes \Lam^{2,1}_{3,[1,2]}$ \\
\hline
\size $\Ex_2(k)$ & \size $x_{221},x_{231},x_{321},x_{331},x_{222},x_{232},x_{322},x_{332}$ 
&  \size \hskip 0.8in - \\
\hline \hline 
\size 38 & \size $M_{\emptyset,\emptyset,[1]}\cong \gl_3^2\times \gl_1^2$
& \size $\Lam^{3,1}_{1,[1,3]}\otimes \Lam^{3,1}_{1,[1,3]}$ \\
\hline
\size SP & \size
\begin{math}
\begin{matrix}
x_{112},x_{122},x_{132},x_{212},x_{222}, \hfill \\
x_{232},x_{312},x_{322},x_{332} \hfill 
\end{matrix}
\end{math}
& \size \hskip 0.8in - \\ 
\hline \hline 
\size 39 & \size $M_{[1,2],[1,2],[1]}\cong \gl_1^8$
& \size $1^{3\oplus}$ \\
\hline
\size SP & \size $x_{331},x_{232},x_{322}$
& \size $x_{332}$ \\
\hline \hline 
\size 40 & \size $M_{[1,2],[1],[1]}\cong \gl_2\times \gl_1^6$ 
& \size $(\Lam^{2,1}_{2,[2,3]})^{2\oplus} \oplus 1$  \\
\hline
\size SP & \size $x_{321},x_{331},x_{222},x_{232},x_{312}$ 
& \size $x_{322},x_{332}$ \\
\hline
\end{tabular}

\end{center}

\begin{center}
 
\begin{tabular}{|c|l|l|}
\hline 
\size $i$ & $\;$ \hskip 1in \size $M_{\be_i}$ \hskip 1in $\;$
& $\;$ \hskip 0.9in \size $Z_{\be_i}$ \hskip 0.9in $\;$ \\
\hline
\size $G_k\backslash S_{\be_i\, k}$  & 
\hskip 0.5in \size
coordinates of $Z_{\be_i}$
& \size \hskip 0.4in
coordinates of $W_{\be_i}$  \\
\hline \hline 
\size 41 & \size $M_{[1],[1,2],[1]}\cong \gl_2\times \gl_1^6$ 
& \size $\Lam^{2,1}_{1,[2,3]}\oplus 1\oplus \Lam^{2,1}_{1,[2,3]}$  \\
\hline
\size SP & \size $x_{231},x_{331},x_{132},x_{222},x_{322}$
& \size $x_{232},x_{332}$ \\ 
\hline \hline 
\size 42 & \size $M_{[2],[2],[1]}\cong \gl_2^2\times \gl_1^4$ 
& \size $1\oplus \Lam^{2,1}_{1,[1,2]}\otimes \Lam^{2,1}_{2,[1,2]}$ \\
\hline
\size SP & \size $x_{331},x_{112},x_{122},x_{212},x_{222}$
& \size $x_{132},x_{232},x_{312},x_{322},x_{332}$ \\
\hline \hline
\size 45 & \size $M_{[1],[1],[1]}\cong \gl_2^2\times \gl_1^4$
& \size $\Lam^{2,1}_{1,[2,3]}\otimes \Lam^{2,1}_{2,[2,3]}$ \\
\hline
\size SP & \size $x_{222},x_{232},x_{322},x_{332}$
& \size \hskip 0.8in - \\
\hline \hline 
\size 47 & \size $M_{[2],[1],\emptyset}\cong \gl_2^3\times \gl_1^2$
& \size $\Lam^{2,1}_{2,[2,3]}\otimes \Lam^{2,1}_{3,[1,2]}$ \\
\hline
\size SP & \size $x_{321},x_{331},x_{322},x_{332}$
& \size \hskip 0.8in - \\
\hline 
\size 48 & \size $M_{[1],[2],\emptyset}\cong \gl_2^3\times \gl_1^2$
& \size $\Lam^{2,1}_{1,[2,3]}\otimes \Lam^{2,1}_{3,[1,2]}$ \\
\hline
\size SP & \size $x_{231},x_{331},x_{232},x_{332}$
& \size \hskip 0.8in - \\
\hline \hline
\size 49 & \size $M_{[2],[2],[1]}\cong \gl_2^2\times \gl_1^4$
& \size $1$ \\
\hline
\size SP & \size $x_{332}$
& \size \hskip 0.8in - \\
\hline 

\end{tabular}

\end{center}

\vskip 5pt

For example, 

\begin{center}
 
\begin{tabular}{|c|l|l|}
\hline 
\size $\;$ \hskip 0.15in 6 \hskip 0.15in $\;$ & 
\size $M_{[1],[1],[1]}\cong \gl_2^2\times \gl_1^4$ 
\hskip 1.05in $\;$ 
& \size 
$\Lam^{2,1}_{1,[2,3]}\otimes \Lam^{2,1}_{2,[2,3]}
\oplus \Lam^{2,1}_{2,[2,3]} \oplus \Lam^{2,1}_{1,[2,3]}$ 
\hskip 0.2in $\;$ \\
\hline
\size SP & \size
\begin{math}
\begin{matrix}
x_{221},x_{231},x_{321},x_{331}, \hfill 
x_{122},x_{132},x_{212},x_{312}  \hfill 
\end{matrix}
\end{math} 
& \size $x_{222},x_{232},x_{322},x_{332}$   \\
\hline
\end{tabular}

\end{center}

\vskip 5pt
\noindent
means the following. 

\begin{itemize}
\item[(1)]
The stratum in question is $S_{\be_6}$.
\item[(2)]
$M_{\be_6}= M_{[1],[1],[1]}$. 
\item[(3)]
$Z_{\be_6}$ is equivalent to 
$\Lam^{2,1}_{1,[2,3]}\otimes \Lam^{2,1}_{2,[2,3]}
\oplus \Lam^{2,1}_{2,[2,3]} \oplus \Lam^{2,1}_{1,[2,3]}$
as \rep s of $M_{\be_6}$ ignoring $\gl_1$'s in $M_{\be_6}$. 
\item[(4)]
$G_k\backslash S_{\be_6\, k}\cong P_{\be_6\,k}\backslash Y^{\sst}_{\be_6\,k}$
consists of a single point. 
\item[(5)]
$Z_{\be_6}$ is spanned by  
\begin{math}
e_{221},e_{231},e_{321},e_{331}, 
e_{122},e_{132},e_{212},e_{312}. 
\end{math}
\item[(6)] 
$W_{\be_6}$ is spanned by 
\begin{math}
e_{222},e_{232},e_{322},e_{332}. 
\end{math}
\end{itemize}

\vskip 5pt

Note that if $M_{\be}\cong M_{[1],[1],[1]}$ for example, 
then $P_{\be}=P_{[1],[1],[1]}$. So we did not list $P_{\be}$'s. 

Before considering individual cases, we prove a lemma
which will be needed in some cases. 
Let $G=\gl_3\times \gl_2^2,\; E=\aff^3,\; H=\m_2$ and $Z=E\otimes H$. 
We consider the action of $G$ on $Z$ defined  by 
$(g_1,g_2,g_3)(e\otimes A)=(g_1e)\otimes (g_2A{}^t g_3)$. 
Let $\{\coorde_1,\coorde_2,\coorde_3\}$ be the standard basis of $E$. 
We define a map $\Phi:Z\to \wedge^3 H$ by 
\begin{equation*}
\Phi: Z\ni \sum_{i=1}^3 \coorde_i\otimes h_i \mapsto 
h_1\wedge h_2\wedge h_3\in \wedge^3 H. 
\end{equation*}
This is the ``Castling transform'' and is a special case 
of the Pl\"ucker coordinate. 
Since $\dim H=4$, $\wedge^3 H\cong H^*$ (the dual space). 
Note that $H^*\cong (\aff^2)^*\otimes (\aff^2)^*\cong \m_2$. 
As a \rep{} of $\gl_2^2$, $H^*$ is $(\det g_2)^2(\det g_3)^2$ 
times the tensor product of contragredient \rep s 
of the standard \rep s of two $\gl_2$'s.  

Let $E_{ij}\in \m_2$ be the matrix whose $(i,j)$-entry 
is $1$ and other entries are zero. 
Let 
\begin{equation}
\label{eq:R-defn}
R= \coorde_1\otimes (-E_{11}+E_{22})
+ \coorde_2\otimes E_{12}
+ \coorde_3\otimes E_{21}.
\end{equation}
Let $\{E_{ij}^*\}$ be the dual basis of 
$\{E_{11},E_{12},E_{21},E_{22}\}$. 
Regarding $E_{11}^*=E_{12}\wedge E_{21}\wedge E_{22}$, 
\begin{equation}
\label{eq:RtoI}
\Phi(R)=I_2.
\end{equation}

By the comment after (\ref{eq:bilinear-gl2-property}), 
we obtain the following lemma.
\begin{lem}
\label{lem:Phi-equivariant}
In the above situation, 
\begin{equation*}
\Phi(gx) = (\det g_1)(\det g_2)(\det g_3)\theta(g_2)\Phi(x){}^t\theta(g_3).
\end{equation*}
\end{lem}

We now verify that $S_{\be_i}\not=\emptyset$ and 
determine $G_k\backslash S_{\be_i\, k}$
for the above $i$'s.

\vskip 10pt

(1) $\be_4 = \tfrac {1} {6} (-2,1,1,0,0,0,0,0)$, 
$\be_5 = \tfrac {1} {6} (0,0,0,-2,1,1,0,0)$.  
We only consider $\be_4$. 

We identify the element 
$(\diag(t_1,g_1),g_2,g_3)\in M_{[1],\emptyset,\emptyset}$
with $g=(g_2,g_1,g_3,t_1)\in \gl_3\times \gl_2^2\times \gl_1$.
It is easy to see that 
$\chi_{\be_4}(g)=t_1^{-2}(\det g_1)$ for $g\in M_{\be_4}$. 
If $g\in M^1_{\be_4}$ then  $\chi_{\be_4}(g)=(\det g_1)^3$.

Let $\{\coorde_{i,1},\coorde_{i,2},\coorde_{i,3}\}$ 
be the standard basis of the $i$-th $\aff^3$ 
and $\{\coordf_1,\coordf_2\}$  
the standard basis of $\aff^2$.  
So $\{\coorde_{2,1},\coorde_{2,2},\coorde_{2,3}\}$ 
is a basis of $\Lam^{3,1}_{2,[1,3]}$. We put 
$H=\Lam^{2,1}_{1,[2,3]}\otimes \Lam^{2,1}_{3,[1,2]}$. 
We identify $H$ with $\m_2$ so that elements of 
$\Lam^{2,1}_{1,[2,3]}$ (resp. $\Lam^{2,1}_{3,[1,2]}$) 
correspond to columns (resp. rows). 

Let $\Phi:Z_{\be_4}\to \wedge^3 H \cong H^* \cong \m_2$ be the 
following map: 
\begin{equation*}
\Phi: Z_{\be_4} = \Lam^{3,1}_{2,[1,3]}\otimes H\ni 
\sum_{i=1}^3 \coorde_{2,i}\otimes h_i \mapsto 
h_1\wedge h_2\wedge h_3\in \wedge^3 H \cong H^*. 
\end{equation*}
Let $P(x)=\det \Phi(x)$. Lemma \ref{lem:Phi-equivariant}
implies that $P(gx)=(\det g_1)^3 (\det g_2)^2(\det g_3)^3P(x)$. 
Note that $g_2$ corresponds to $g_1$ in Lemma \ref{lem:Phi-equivariant}.
If $g\in M^1_{\be_4}$ then $t_1\det g_1=1,\det g_2=\det g_3=1$.  
So the character 
$(\det g_1)^3 (\det g_2)^2(\det g_3)^3=(\det g_1)^3$ is 
$\chi_{\be_4}$ and $P(x)$ is invariant by the action of $G_{\text{st},\be_4}$. 
Therefore, 
\begin{equation*}
Z^{\sst}_{\be_4\, k} = \{x\in Z_{\be_4\, k}\mid \det \Phi(x)\not=0\}. 
\end{equation*}

It is known that 
if $x,y\in Z_{\be_4\,k}$ and $\Phi(x),\Phi(y)\not=0$ then 
there exists $g_1\in\gl_3(k)$ such that $x=g_1y$ if and only if 
$\Phi(x)$ is a scalar multiple of $\Phi(y)$. 
Since $\{A\in\m_2(k)\mid \det A\not=0\}$ is a single 
$\gl_2(k)$-orbit, $M_{\be_4\, k}\backslash Z^{\sst}_{\be_4}$ 
consists of a single point. 

Let 
\begin{equation*}
R(4) = \coorde_{2,1} \otimes (-E_{11} + E_{22})   
+ \coorde_{2,2} \otimes E_{12} 
+ \coorde_{2,3} \otimes E_{21}. 
\end{equation*}
By (\ref{eq:RtoI}), $\Phi(R(4))=I_2$. Since $W_{\be_4}=\{0\}$, $Z^{\sst}_{\be_4} = M_{\be_4\, k}R(4)$. 

\vskip 5pt




(2) $\be_6 = \tfrac {1} {66} (-4,2,2,-4,2,2,-3,3)$.  

We identify the element 
$(\diag(t_1,g_1),\diag(t_2,g_2),t_{31},t_{32})\in M_{[1],[1],[1]}$
with $g=(g_1,g_2,t_1,t_2,t_{31},t_{32})\in \gl_2^2\times \gl_1^4$. 
On $M^1_{\be_6}$, 
%
\begin{equation*}
\chi_{\be_6}(g) = (t_1t_2)^{-4}(\det g_1)^2(\det g_2)^2 
t_{31}^{-3}t_{32}^{3} 
= ((\det g_1)(\det g_2)t_{32})^6. 
\end{equation*}

Let 
\begin{equation*}
\begin{aligned}
A(x) & = 
\begin{pmatrix}
x_{221} & x_{231} \\
x_{321} & x_{331} 
\end{pmatrix}, \quad
v_1(x) = 
\begin{pmatrix}
x_{212} \\ x_{312}
\end{pmatrix}, \quad 
v_2(x) = 
\begin{pmatrix}
x_{122} \\ x_{132}
\end{pmatrix}, \\
P_1(x) & = \det A(x), \quad 
P_2(x) = {}^t v_1(x) \theta(A(x)) v_2(x).
\end{aligned}
\end{equation*}
Then 
\begin{equation*}
A(gx) = t_{31} g_1 A(x) {}^t g_2, \quad
v_1(gx) = t_2t_{32}g_1 v_1(x), \quad
v_2(gx) = t_1t_{32}g_2 v_2(x).  
\end{equation*}
On $M^1_{\be_6}$, 
\begin{equation*}
P_1(gx) = (\det g_1)(\det g_2)t_{31}^2P_1(x)
=(\det g_1)(\det g_2)t_{32}^{-2}P_1(x). 
\end{equation*}
Also on $M^1_{\be_6}$, 
\begin{align*}
P_2(gx) & =  t_1t_2t_{31}t_{32}^2{}^t (g_1v_1(x)) \theta(g_1A(x){}^tg_2) g_2 v_2(x) \\
& = t_1t_2t_{32}{}^t v_1(x) {}^t g_1 \theta(g_1)\theta(A(x)) \theta({}^tg_2) g_2 v_2(x) \\
& = t_1t_2t_{32}(\det g_1)(\det g_2){}^t v_1(x) {}^tg_1 {}^tg_1^{-1} 
\theta(A(x)) g_2^{-1}g_2 v_2(x) \\
& = t_{32} P_2(x). 
\end{align*}
Therefore, if we put $P(x) = P_1(x)P_2(x)^3$ then  
$P(gx)=(\det g_1)(\det g_2)t_{32}P(x)$ and 
the character $(\det g_1)(\det g_2)t_{32}$ is 
proportional to $\chi_{\be_6}$. Hence, 
$P(x)$ is invariant under the action of $G_{\text{st},\be_6}$. 
This implies that 
\begin{equation*}
Z^{\sst}_{\be_6\,k} = \{x\in Z_{\be_6\,k}
\mid \det A(x)\not=0,{}^tv_1(x)A(x)v_2(x)\not=0\}. 
\end{equation*}
\begin{prop}
\label{prop:orbit-beta6}
Let $R(6)\in Z_{\be_6\,k}$ be the element such that $A(x)=I_2$, $v_1(x)=v_2(x)=[1,0]$.
Then $Y^{\sst}_{\be_6\,k}=P_{\be_6\, k}R(6)$. 
\end{prop}
\begin{proof}
Let $x\in Z^{\sst}_{\be_6\, k}$. 
By assumption, $\det A(x)\not=0,{}^tv_1(x)A(x)v_2(x)\not=0$. 
By the action of $g_1\in \gl_2(k)\sub G_{1\, k}$, 
we may assume that $A(x)=I_2$. By assumption, $v_2(x)\not=0$. 
By the action of an element of the form 
$(\diag(1,g_1),\diag(1,{}^tg_1^{-1}))\in G_{1\,k}\times G_{2\,k}$, 
we may assume that $v_2(x)=[1,0]$. By assumption, 
$x_{212}\not=0$. Applying an element of the form  
$g=(I_2,I_2,1,t_2,1,1)$, we may assume $x_{212}=1$. 
Applying the element 
$(n_2(-x_{312}),{}^tn_2(x_{312}),1,1,1,1)$, 
we can make $x_{312}=0$ (see the beginning of Section \ref{sec:notation}). 
Therefore, we may assume that 
$v_1(x)=[1,0]$. So $x=R(6)$.

If $x\in Y^{\sst}_{\be_6\,k}$ then 
by the above consideration, we may assume that 
$x$ is in the form $x=(R(6),w)$
where $w=(x_{222},x_{232},x_{322},x_{332})$.

Let  $u_1=(u_{1ij})_{3\geq i>j\geq 1}$, 
$u_2=(u_{2ij})_{3\geq i>j\geq 1}$, 
$u_3=u_{321}$ and put  
\begin{equation}
\label{eq:n(u)-defn}
n(u) = (n_3(u_1),n_3(u_2),n_2(u_3)). 
\end{equation}
This element $n(u)$ belongs to $U_{\be_6}$ if and only if 
$u_{132}=u_{232}=0$.

Let $\coorde_{i,1},\coordf_i$, etc., be as before.  
By the action of $n(u)$, 
\begin{equation*}
\coorde_{i,1}\mapsto \coorde_{i,1} + u_{i21}\coorde_{i,2}+ u_{i31}\coorde_{i,3},\;
\coorde_{i,2}\mapsto \coorde_{i,2},\;
\coorde_{i,3}\mapsto \coorde_{i,3}
\end{equation*}
for $i=1,2$ and 
\begin{equation*}
\coordf_1\mapsto \coordf_1 + u_{321}\coordf_2, 
\coordf_2\mapsto \coordf_2. 
\end{equation*}

We consider $u$ such that $u_{121}=0$ 
($u_{132}=u_{232}=0$ is still assumed). 
Let $n(u)(R(6),w)=(R(6),w')$. Then $w'=(x_{222}'\ccd x_{332}')$
where 
\begin{align*}
x_{222}' & = x_{222}+u_{321} + u_{221}, \quad
x_{232}' = x_{232}+u_{231}, \\
x_{322}' & = x_{322}+u_{131}, \quad
x_{332}' = x_{332}+u_{321}.
\end{align*}
So we can choose $u$ so that $w'=(0,0,0,0)$. 
\end{proof} 

\vskip 5pt

(3) $\be_{11} = \tfrac {1} {18} (-2,0,2,-2,0,2,-1,1)$.  



We identify the element 
\begin{equation}
\label{eq:elementsofT}
t = (\diag(t_{11},t_{12},t_{13}),\diag(t_{21},t_{22},t_{23}),\diag(t_{31},t_{32}))
\in M_{[1,2],[1,2],[1]}
\end{equation}
with $t=(t_{11},t_{12},t_{13},t_{21},t_{22},t_{23},t_{31},t_{32})
\in \gl_1^8$. 
On $M^1_{\be_{11}}$, 
%
\begin{equation*}
\chi_{\be_{11}}(t) = t_{11}^{-2}t_{13}^2t_{21}^{-2}t_{23}^2t_{31}^{-1}t_{32}
= t_{12}^2t_{13}^4t_{22}^2t_{23}^4t_{32}^2. 
\end{equation*}
Let $t e_{ijk}=\chi_{ijk}(t) e_{ijk}$ for $i,j=1,2,3,k=1,2$
where the left hand side is the action and the right hand side 
is a scalar multiplication. Then $\chi_{ijk}(t)$ for 
the coordinates of $Z_{\be_{11}}$ on $M_{\be_{11}}$ or 
$M^1_{\be_{11}}$ are as follows. 

\vskip 10pt

\tiny

\begin{center}
 
\begin{tabular}{|c|c|c|c|c|c|}
\hline 
\text{character} & $\chi_{231}(t)$ & $\chi_{321}(t)$ & $\chi_{132}(t)$ 
& $\chi_{222}(t)$ & $\chi_{312}(t)$ \\
\hline
on $M_{\be_{11}}$ & $t_{12}t_{23}t_{31}$ & $t_{13}t_{22}t_{31}$ & $t_{11}t_{23}t_{32}$
& $t_{12}t_{22}t_{32}$ & $t_{13}t_{21}t_{32}$  \\
\hline
on $M^1_{\be_{11}}$ & $t_{12}t_{23}t_{32}^{-1}$ & $t_{13}t_{22}t_{32}^{-1}$ 
& $t_{12}^{-1}t_{13}^{-1}t_{23}t_{32}$ & $t_{12}t_{22}t_{32}$ 
& $t_{13}t_{22}^{-1}t_{23}^{-1}t_{32}$ \\
\hline
\end{tabular}

\end{center} 

\normalsize

\vskip 10pt

If we put
\begin{math}
P(x) = (x_{231}x_{321}x_{132}x_{312})^4 x_{222}^2
\end{math}
then by straightforward computations, 
$P(tx)=t_{12}^2t_{13}^4t_{22}^2t_{23}^4t_{32}^2P(x)$ 
for $t\in M^1_{\be_{11}}$. 
So $P(x)$ is invariant under the action of $t\in G_{\text{st},\be_{11}}$. 
Therefore, 
\begin{equation*}
Z^{\sst}_{\be_{11}\,k}=\{x\in Z_{\be_{11}\,k}\mid 
x_{231},x_{321},x_{132},x_{222},x_{312}\in \mk\}. 
\end{equation*}

We consider $P_{\be_{11}\,k}\backslash Y^{\sst}_{\be_{11}\,k}$. 
Given $q_1,q_2,q_3,q_4,q_5\in\mk$, if we put 
$t_{22}=t_{23}=t_{31}=1$, 
$t_{12}=q_1,t_{13}=q_2,t_{32}=q_4q_1^{-1},t_{11}=q_1q_3q_4^{-1},
t_{21}=q_1q_2^{-1}q_4^{-1}q_5$ ($t\in M_{\be_{11}\, k}$) then 
$\chi_{231}(t)=q_1\ccd \chi_{312}(t)=q_5$. 
Let $R(11)=(1,1,1,1,1)\in Z_{\be_{11}\,k}$.
Then by the above consideration, 
$Z^{\sst}_{\be_{11}\,k}=M_{\be_{11}\, k}R(11)$. 

\begin{prop}
\label{prop:orbit-beta11}
$Y^{\sst}_{\be_{11}\,k}=P_{\be_{11}\, k}R(11)$. 
\end{prop}
\begin{proof}
Let $x\in Y^{\sst}_{\be_{11}}$. 
By the above consideration, we may assume that 
$x=(R(11),w)$ where $w=(x_{331},x_{232},x_{322},x_{332})$. 

We consider $n(u)$ in (\ref{eq:n(u)-defn}) such that
$u_{132}=u_{231}=u_{232}=u_{321}=0$. 
Let $n(u)x=(R(11),w')$. Then $w'=(x_{331}'\ccd x_{332}')$. 
where 
\begin{align*}
x_{331}' & = x_{331} + u_{132}, \quad
x_{232}' = x_{232} + u_{121}, \\
x_{322}' & = x_{322} + u_{132}+u_{221}, \quad
x_{332}' = x_{332} + u_{131}.
\end{align*}
So we can choose $u$ so that $w'=(0,0,0,0)$. 
\end{proof}
%


\vskip 5pt

(4) $\be_{22} = \tfrac {1} {12} (-1,-1,2,-1,-1,2,0,0)$.  

%

We identify the element 
$(\diag(g_1,t_1),\diag(g_2,t_2),g_3)\in M_{[2],[2],\emptyset}$
with the element $g=(g_1,g_2,g_3,t_1,t_2)\in \gl_2^3\times \gl_1^2$. 
On $M^1_{\be_{22}}$, 
%
\begin{equation*}
\chi_{\be_{22}}(g) = (\det g_1)^{-1}t_1^2(\det g_2)^{-1}t_2^2
= (t_1t_2)^3. 
\end{equation*}

For $x\in Z_{\be_{22}}$, let 
\begin{equation*}
A(x) = 
\begin{pmatrix}
x_{131} & x_{132} \\
x_{231} & x_{232}
\end{pmatrix}, \quad
B(x) = 
\begin{pmatrix}  
x_{311} & x_{312} \\
x_{321} & x_{322}
\end{pmatrix}. 
\end{equation*}
We identify $x$ with the pair  
$(A(x),B(x))\in \m_2\oplus \m_2$. Then the
action of $g\in M_{\be_{22}\, k}$ 
is $(A(gx),B(gx)) = (t_2g_1A(x){}^tg_3,t_1g_2B(x){}^tg_3)$.

Let $P(x)=(\det A(x))(\det B(x))$. Then on $M^1_{\be_{22}}$, 
\begin{equation*}
P(gx) = (t_1t_2)^2 (\det g_1)(\det g_2)(\det g_3)^2 P(x)
= (t_1t_2) P(x). 
\end{equation*}
So $P(x)$ is invariant under the action of $G_{\text{st},\be_{22}}$. 
This implies that 
\begin{equation*}
Z_{\be_{22}\,k}^{\sst} = \{x\in Z_{\be_{22}\,k}\mid \det A(x),\det B(x)\not=0\}. 
\end{equation*}
Since $(g_1,g_2,I_2,1,1)\cdot (I_2,I_2)=(g_1,g_2)$,  
$Z^{\sst}_{\be\,k}$ is a single $M_{\be \, k}$-orbit. 

Let $R(22)\in Z^{\sst}_{\be\,k}$ be the element such that 
$A(R(22))=B(R(22))=I_2$. 
\begin{prop}
\label{prop:orbit-beta22}
$Y^{\sst}_{\be_{22}\,k}=P_{\be_{22}\, k}R(22)$. 
\end{prop}
\begin{proof}
Let $x\in Y^{\sst}_{\be_{22}}$.  By the above consideration, 
we may assume that $A(x)=B(x)=I_2$. 
We put $x=(R(22),w)$ where $w=(x_{331},x_{332})$. 

Elements of $U_{\be\,k}$ are in the form   
$n(u)$ in (\ref{eq:n(u)-defn}) such that
$u_{121}=u_{221}=u_{321}=0$. We assume 
further that $u_{132}=u_{231}=0$. 
Let $n(u)(R(22),w)=(R(22),w')$. Then 
$w'=(x_{331}',x_{332}')$ where 
\begin{align*}
x_{331}' & = x_{331} + u_{131}, \quad
x_{332}' = x_{332} + u_{232}. 
\end{align*}
So we can choose $u$ so that $w'=(0,0)$. 
\end{proof}

\vskip 5pt

(5) $\be_{28} = \tfrac {1} {6} (-1,0,1,-1,0,1,-1,1)$. 
%
%

We express elements of $M_{[1,2],[1,2],[1]}$
as (\ref{eq:elementsofT}). 
On $M^1_{\be_{28}}$, 
%
\begin{equation*}
\chi_{\be_{28}}(t) = t_{11}^{-1}t_{13}t_{21}^{-1}t_{23}t_{31}^{-1}t_{32}
= t_{12}t_{13}^2t_{22}t_{23}^2t_{32}^2. 
\end{equation*}

Let $t e_{ijk}=\chi_{ijk}(t) e_{ijk}$ for $i,j=1,2,3,k=1,2$ 
as before. Then $\chi_{ijk}(t)$ for 
the coordinates of $Z_{\be_{28}}$ on $M_{\be_{28}}$ or 
$M^1_{\be_{28}}$ are as follows.

\vskip 10pt

\tiny

\begin{center}
 
\begin{tabular}{|c|c|c|c|c|}
\hline 
\text{character} & $\chi_{331}(t)$ & $\chi_{132}(t)$ 
& $\chi_{222}(t)$ & $\chi_{312}(t)$ \\
\hline
on $M_{\be_{28}}$ & $t_{13}t_{23}t_{31}$ & $t_{11}t_{23}t_{32}$ 
& $t_{12}t_{22}t_{32}$ & $t_{13}t_{21}t_{32}$  \\
\hline
on $M^1_{\be_{28}}$ & $t_{13}t_{23}t_{32}^{-1}$ & $t_{12}^{-1}t_{13}^{-1}t_{23}t_{32}$ 
& $t_{12}t_{22}t_{32}$ & $t_{13}t_{22}^{-1}t_{23}^{-1}t_{32}$ \\
\hline
\end{tabular}

\end{center} 

\normalsize

\vskip 10pt

We put $P(x)=(x_{331}x_{222})^2x_{132}x_{312}$. 
Then on $M^1_{\be_{28}}$, 
$P(tx)=t_{12}t_{13}^2t_{22}t_{23}^2t_{32}^2P(x)$. 
So $P(x)$ is invariant under the action of $G_{\text{st},\be_{28}}$. 
This implies that
\begin{equation*}
Z_{\be_{28}\,k}^{\sst} = \{x\in Z_{\be_{28}\,k}\mid 
x_{331},x_{132},x_{222},x_{312}\not=0\}. 
\end{equation*}

Let $q_1,q_2,q_3,q_4\in\mk$. If we put 
$t_{31}=q_1,t_{11}=q_2,t_{12}=q_3,t_{21}=q_4$. 
$t_{13}=t_{22}=t_{23}=t_{32}=1$. 
Then $\chi_{331}(t)=q_1,\chi_{132}(t)=q_2,\chi_{222}(t)=q_3,\chi_{312}(t)=q_4$. 
So $Z^{\sst}_{\be_{28}\,k}$ is a single $M_{\be_{28} \, k}$-orbit.

Let $R(28) =(1,1,1,1)\in Z^{\sst}_{\be_{28}\,k}$. 
\begin{prop}
\label{prop:orbit-beta28}
$Y^{\sst}_{\be_{28}\,k}=P_{\be_{28}\, k}R(28)$. 
\end{prop}
\begin{proof}
Let $x\in Y^{\sst}_{\be_{28}}$.  By the above consideration, 
we may assume that $x=(R(28),w)$  
where $w=(x_{232},x_{322},x_{332})$. 

Elements of $U_{\be_{28}\,k}$ are in the form   
$n(u)$ in (\ref{eq:n(u)-defn}). 
We assume that $u_{132}=u_{231}=u_{232}=u_{321}=0$. 
Let $n(u)(R(28),w)=(R(28),w')$. Then
$w'=(x_{232}',x_{322}',x_{332}')$ where
\begin{align*}
x_{232}' & = x_{232} + u_{121}, \quad
x_{322}' = x_{322} + u_{221}, \quad
x_{332}' = x_{332} + u_{131}.
\end{align*}
So we can choose $u$ so that $w'=(0,0,0)$. 
\end{proof}

\vskip 5pt

(6) $\be_{29} = \tfrac {1} {6} (-2,1,1,-2,1,1,0,0)$.  



We identify the element 
$(\diag(t_1,g_1),\diag(t_2,g_2),g_3)\in M_{[1],[1],\emptyset}$
with the element $g=(g_1,g_2,g_3,t_1,t_2)\in \gl_2^3\times \gl_1^2$. 
On $M^1_{\be_{29}}$, 
%
\begin{equation*}
\chi_{\be_{29}}(g) = (t_1t_2)^{-2}(\det g_1)(\det g_2)
= (\det g_1)^3(\det g_2)^3. 
\end{equation*}

The action of $M_{\be_{29}}\cong \gl_2^3\times \gl_1^2\ni (g_1,g_2,g_3,t_1,t_2)$
on $Z_{\be_{29}}\cong \aff^2\otimes \aff^2\otimes \aff^2$ is 
$v_1\otimes v_2\otimes v_3\mapsto g_1v_1\otimes g_2v_2\otimes g_3v_3$. 
Even though there are extra $\gl_1$-factors, 
this \rep{} is essentially the same as the ``$D_4$-case'' 
in \cite{wryu}. 
It is known that this is a  \pv{} and 
$M_{\be_{29}\, k}\backslash Z^{\sst}_{\be_{29}\,k}$
is in bijective correspondence with $\Ex_2(k)$. 
Note that this \rep{} is ``regular'' in the sense of 
Definition 2.1 \cite[p.310]{kato-yukie-jordan} 
by simple Lie algebra computations without any
assumption on $\ch(k)$. 
It allows us to use the usual cohomological 
considerations. 

There is a relative invariant polynomial $P(x)$ 
of degree $4$. 
If $g\in M^1_{\be_{29}}$ then 
\begin{equation*}
P(gx)=(\det g_1\det g_2 \det g_3)^2 P(x)
= (\det g_1\det g_2)^2 P(x). 
\end{equation*}
So $P(x)$ is invariant under the action of $G_{\text{st},\be_{29}}$. 
Since $W_{\be_{29}}=\{0\}$, 
\begin{math}
P_{\be_{29}\, k}\backslash Y^{\sst}_{\be_{29}\,k} 
\cong M_{\be_{29}\, k}\backslash Z^{\sst}_{\be_{29}\,k}
\cong \Ex_2(k). 
\end{math}
We do not provide the details here. 

\vskip 5pt

(7) $\be_{38} = \tfrac {1} {2} (0,0,0,0,0,0,-1,1)$.  



We identify the element 
$(g_1,g_2,\diag(t_{31},t_{32}))\in M_{\emptyset,\emptyset,[1]}$
with the element $g=(g_1,g_2,t_{31},t_{32})\in \gl_3^2\times \gl_1^2$. 
On $M^1_{\be_{38}}$, 
%
\begin{equation*}
\chi_{\be_{38}}(g) = t_{31}^{-1}t_{32} = t_{32}^2. 
\end{equation*}

For $x\in Z_{\be_{38}}$, let 
\begin{equation*}
A(x) = 
\begin{pmatrix}
x_{112} & x_{122} & x_{132} \\
x_{212} & x_{222} & x_{232} \\
x_{312} & x_{322} & x_{332}
\end{pmatrix}. 
\end{equation*}
We identify $Z_{\be_{38}}$ 
with $\m_3$ by the map $x\mapsto A(x)$. 
Then the action of $g$ on 
$Z_{\be_{38}}\cong \m_3$ is 
$\m_3\ni A\mapsto t_{32}g_1A{}^tg_2$. 
Let $P(x)=\det A(x)$. Then on $M^1_{\be_{38}}$, 
$P(gx)= t_{32}^3(\det g_1)(\det g_2)P(x)=t_{32}^3P(x)$. 
Therefore, $P(x)$ is invariant under the action of 
$G_{\text{st},\be_{38}}$. This implies that 
\begin{equation*}
Z_{\be_{38}\,k}^{\sst} = \{x\in Z_{\be_{38}\,k}\mid 
\det A(x)\not=0\}. 
\end{equation*}
It is easy to see that $Z^{\sst}_{\be_{38}\,k}$ is a single 
$M_{\be_{38} \, k}$-orbit. 

Let $R(38) \in Z^{\sst}_{\be_{38}\,k}$ be the element such that 
$A(R(38))=I_3$.  Since $W_{\be_{38}}=\{0\}$, 
$Y^{\sst}_{\be_{38}\,k}=Z^{\sst}_{\be_{38}\,k}=M_{\be_{38}\, k}R(38)=P_{\be_{38}\, k}R(38)$.

\vskip 5pt

(8) $\be_{39} = \tfrac {1} {6} (-2,0,2,-2,0,2,-1,1)$.  



We express elements of $M_{[1,2],[1,2],[1]}$
as (\ref{eq:elementsofT}). 
On $M^1_{\be_{39}}$, 
%
\begin{equation*}
\chi_{\be_{39}}(t) = t_{11}^{-2}t_{13}^2t_{21}^{-2}t_{23}^2t_{31}^{-1}t_{32}
= t_{12}^2t_{13}^4t_{22}^2t_{23}^4t_{32}^2.
\end{equation*}

Let $t e_{ijk}=\chi_{ijk}(t) e_{ijk}$ for $i,j=1,2,3,k=1,2$ 
as before. Then $\chi_{ijk}(t)$ for 
the coordinates of $Z_{\be_{39}}$ on $M_{\be_{39}}$ or 
$M^1_{\be_{39}}$ are as follows.

\vskip 10pt

\tiny

\begin{center}
 
\begin{tabular}{|c|c|c|c|}
\hline 
\text{character} & $\chi_{331}(t)$ & $\chi_{232}(t)$ 
& $\chi_{322}(t)$  \\
\hline
on $M_{\be_{39}}$ & $t_{13}t_{23}t_{31}$ & $t_{12}t_{23}t_{32}$ 
& $t_{13}t_{22}t_{32}$  \\
\hline
on $M^1_{\be_{39}}$ & $t_{13}t_{23}t_{32}^{-1}$ & $t_{12}t_{23}t_{32}$ 
& $t_{13}t_{22}t_{32}$ \\
\hline
\end{tabular}

\end{center} 

\normalsize

\vskip 10pt

We put $P(x)=(x_{331}x_{232}x_{322})^2$. Then 
$P(tx)=t_{12}^2t_{13}^4t_{22}^2t_{23}^4t_{32}^2P(x)$ for $t\in M^1_{\be_{39}}$. 
So $P(x)$ is invariant under the action of $G_{\text{st},\be_{39}}$. 
This implies that 
\begin{equation*}
Z_{\be_{39}\,k}^{\sst} = \{x\in Z_{\be_{39}\,k}\mid 
x_{331},x_{232},x_{322}\not=0\}. 
\end{equation*}

Let $q_1,q_2,q_3\in\mk$. If we put 
$t_{31}=q_1,t_{12}=q_2,t_{22}=q_3$, 
$t_{11}=t_{13}=t_{21}=t_{23}=t_{32}=1$. 
Then $\chi_{331}(t)=q_1,\chi_{232}(t)=q_2,\chi_{322}(t)=q_3$. 
So $Z^{\sst}_{\be_{39}\,k}$ is a single $M_{\be \, k}$-orbit.

Let $R(39)=(1,1,1)\in Z^{\sst}_{\be_{39}\,k}$. 
\begin{prop}
\label{prop:orbit-beta39}
$Y^{\sst}_{\be_{39}\,k}=P_{\be_{39}\, k}R(39)$. 
\end{prop}
\begin{proof}
Let $x\in Y^{\sst}_{\be_{39}}$.  By the above consideration, 
we may assume that $x$ is in the form 
$(R(39),w)$ where $w=(x_{332})$. 

Elements of $U_{\be_{39}\,k}$ are in the form   
$n(u)$ in (\ref{eq:n(u)-defn}). 
We assume that $u_{ijk}=0$ unless $(i,j,k)=(1,3,2)$.  
Let $n(u)(R(39),x_{332})=(R(39),x_{332}')$. Then
\begin{align*}
x_{332}' & = x_{332} + u_{132} 
\end{align*}
So we can choose $u$ so that $x_{332}'=0$. 
\end{proof}

\vskip 5pt

(9) $\be_{40} = \tfrac {1} {30} (-10,2,8,-4,2,2,-3,3)$, 
$\be_{41} = \tfrac {1} {30} (-4,2,2,-10,2,8,-3,3)$. 
We only consider $\be_{40}$.  



We identify the element 
$(\diag(t_{11},t_{12},t_{13}),\diag(t_2,g_2),\diag(t_{31},t_{32}))
\in M_{[1,2],[1],[1]}$ with the element 
$g=(g_2,t_{11},t_{12},t_{13},t_2,t_{31},t_{32})\in \gl_2\times \gl_1^6$. 
On $M^1_{\be_{40}}$, 
%
\begin{equation*}
\chi_{\be_{40}}(g) 
= t_{11}^{-10}t_{12}^2t_{13}^8t_2^{-4}(\det g_2)^2 t_{31}^{-3}t_{32}^3
= t_{12}^{12}t_{13}^{18}(\det g_2)^6 t_{32}^6.
\end{equation*}

For $x\in Z_{\be_{40}}$, we put  
\begin{equation*}
A(x) = \begin{pmatrix}
x_{321} & x_{222} \\ 
x_{331} & x_{232} 
\end{pmatrix}. 
\end{equation*}
Then $(\Lam^{2,1}_{2,[2,3]})^{2\oplus}$ can be identified with 
$\m_2$ by the map $(x_{321},x_{331},x_{222},x_{232})\mapsto A(x)$. 
We express $x$ as $x=(A(x),x_{312})$.
Let $a(t) = \diag(t_{13}t_{31},t_{12}t_{32})$. 
The action of $g$ as above on $Z_{\be_{40}}$ is 
$(A(x),x_{312})\mapsto (g_2A(x)a(t),t_{13}t_2t_{32} x_{312})$. 

Let $P_1(x)=\det A(x)$. Then on $M^1_{\be_{40}}$, 
\begin{equation*}
P_1(gx)=(\det g_2)t_{12}t_{13}t_{31}t_{32}P_1(x)=t_{12}t_{13}(\det g_2)P_1(x).
\end{equation*}
We put $P(x)= P_1(x)^2x_{312}$. Then 
\begin{equation*}
P(gx)=t_{12}^2t_{13}^2(\det g_2)^2t_{13}t_2t_{32}P(x)
= t_{12}^2t_{13}^3(\det g_2)t_{32}P(x).
\end{equation*}
So $P(x)$ is invariant under the action of $G_{\text{st},\be_{40}}$. 
This implies that 
\begin{equation*}
Z^{\sst}_{\be_{40}\, k}=\{(A(x),x_{312})\mid \det A(x),x_{312}\not=0\}. 
\end{equation*}

By the action of an element of the form  $g = (g_2,1,1,1,t_2,1,1)$, 
$(A(x),x_{312})$ maps to $(g_1A(x),t_2x_{312})$. 
This implies that $Z^{\sst}_{\be_{40}\, k}$ is a single 
$M_{\be_{40}\, k}$-orbit. 

Let $R(40) =(I_2,1)\in Z^{\sst}_{\be_{40}\,k}$. 
\begin{prop}
\label{prop:orbit-beta40}
$Y^{\sst}_{\be_{40}\,k}=P_{\be_{40}\, k}R(40)$. 
\end{prop}
\begin{proof}
Let $x\in Y^{\sst}_{\be_{40}}$.  By the above consideration, 
we may assume that $x=(R(40),w)$  
where $w=(x_{322},x_{332})$. 

Elements of $U_{\be_{40}\,k}$ are in the form   
$n(u)$ in (\ref{eq:n(u)-defn}) where $u_{232}=0$. 
 We assume further that $u_{ijk}=0$ unless 
$(i,j,k)=(2,2,1),(2,3,1)$.  
Let $n(u)(R(40),w)=(R(40),w')$. Then
$w'=(x_{322}',x_{332}')$ where 
\begin{align*}
x_{322}' & = x_{322} + u_{221}, \quad
x_{332}' = x_{332} + u_{231}.  
\end{align*}
So we can choose $u$ so that $w'=(0,0)$. 
\end{proof}

\vskip 5pt





\vskip 5pt

(10) $\be_{42} = \tfrac {1} {30} (-1,-1,2,-1,-1,2,-3,3)$.  



We identify the element 
$(\diag(g_1,t_1),\diag(g_2,t_2),\diag(t_{31},t_{32}))
\in M_{[2],[2],[1]}$ with the element 
$g=(g_1,g_2,t_1,t_2,t_{31},t_{32})\in \gl_2^2\times \gl_1^4$. 
On $M^1_{\be_{42}}$, 
%
\begin{equation*}
\chi_{\be_{42}}(g) = ((\det g_1)(\det g_2))^{-1}(t_1t_2)^2 t_{31}^{-3}t_{32}^3
= (t_1t_2)^3t_{32}^6. 
\end{equation*}

For $x\in Z_{\be_{42}}$, we put  
\begin{equation*}
A(x) = \begin{pmatrix}
x_{112} & x_{122} \\ 
x_{212} & x_{222} 
\end{pmatrix}. 
\end{equation*}
We express elements of $Z_{\be_{42}}$ as $x=(x_{331},A(x))$.  
Then $gx = (t_1t_2t_{31}x_{331},t_{32}g_1 A(x){}^tg_2)$. Let
$P_1(x)=\det A(x)$ and $P(x)=x_{331}^4 P_1(x)^3$. 
Then 
\begin{equation*}
P_1(gx) = t_{32}^2(\det g_1)(\det g_2)P_1(x)=t_{32}^2(t_1t_2)^{-1}P_1(x)
\end{equation*}
and $P(gx) = t_1t_2t_{32}^2P(x)$ for $g\in M^1_{\be_{42}}$. 
So $P(x)$ is invariant under the action of $G_{\text{st},\be_{40}}$. 
This implies that 
\begin{equation*}
Z^{\sst}_{\be_{42}\, k}=\{(x_{331},A(x))\mid x_{331},\det A(x)\not=0\}. 
\end{equation*}

It is easy to see that $Z^{\sst}_{\be_{42}\, k}$ 
is a single $M_{\be_{42}\, k}$-orbit. 
Let $R(42)\in Z^{\sst}_{\be_{42}\,k}$ be the element 
such that $x_{331}=1,A(R(42))=I_2$. 

\begin{prop}
\label{prop:orbit-beta42}
$Y^{\sst}_{\be_{42}\,k}=P_{\be_{42}\, k}R(42)$. 
\end{prop}
\begin{proof}
Let $x\in Y^{\sst}_{\be_{42}}$. 
By the above consideration, we may assume that 
$x=(R(42),w)$ where $w=(x_{132},x_{232},x_{312},x_{322},x_{332})$. 

Elements of $U_{\be_{42}}$ are in the form 
$n(u)$ in (\ref{eq:n(u)-defn}) where $u_{121}=u_{221}=0$. 
Let $n(u)x=(R(42),w')$. Then $w'=(x_{132}'\ccd x_{332}')$
where 
\begin{align*}
x_{132}' & = x_{132} + u_{231}, \quad
x_{232}' = x_{232} + u_{232}, \\
x_{312}' & = x_{312} + u_{131}, \quad
x_{322}' = x_{322} + u_{132}.
\end{align*}
By $u_{131},u_{132},u_{231},u_{232}$, we can make
$x'_{132},x'_{232},x'_{312},x'_{322}=0$. Then assuming that
$x_{132},x_{232},x_{312},x_{322}=0$ and
$u_{131},u_{132},u_{231},u_{232}=0$,
$x'_{332}=x_{332}+u_{321}$. Therefore, we can make $w=0$.
\end{proof}
%


\vskip 5pt

(11) $\be_{45} = \tfrac {1} {6} (-2,1,1,-2,1,1,-3,3)$.  



We identify the element 
$(\diag(t_1,g_1),\diag(t_2,g_2),\diag(t_{31},t_{32}))
\in M_{[1],[1],[1]}$ with the element 
$g=(g_1,g_2,t_1,t_2,t_{31},t_{32})\in \gl_2^2\times \gl_1^4$. 
On $M^1_{\be_{45}}$, 
%
\begin{equation*}
\chi_{\be_{45}}(g) = (t_1t_2)^{-2}(\det g_1)(\det g_2)t_{31}^{-3}t_{32}^3
= ((\det g_1)(\det g_2))^3t_{32}^6. 
\end{equation*}

For $x\in Z_{\be_{45}}$, we put  
\begin{equation*}
A(x) = \begin{pmatrix}
x_{222} & x_{232} \\
x_{322} & x_{332}
\end{pmatrix}. 
\end{equation*}
Then $Z_{\be_{45}}$ can be identified with 
$\m_2$ by the map $x\mapsto A(x)$. 
Moreover, $A(gx) = t_{32}g_1A(x){}^tg_2$. 
So if we put $P(x)=\det A(x)$ then 
$P(gx)=(\det g_1)(\det g_2)t_{32}^2$. 
Therefore, $P(x)$ is invariant under the action of $G_{\text{st},\be_{45}}$. 
This implies that 
\begin{equation*}
Z^{\sst}_{\be_{45}\, k}=\{x\in Z_{\be_{45}\,k}\mid \det A(x)\not=0\}. 
\end{equation*}
It is easy to see that $Z^{\sst}_{\be_{45}\, k}$ 
is a single $M_{\be_{45}\, k}$-orbit. 
Let $R(45)\in Z^{\sst}_{\be_{45}\,k}$ be the element 
such that $A(x)=I_2$. 
Since $W_{\be_{45}}=\{0\}$, 
\begin{math}
Y^{\sst}_{\be_{45}\,k}=Z^{\sst}_{\be_{45}\,k}
=M_{\be_{45}\, k}R(45)=P_{\be_{45}\, k}R(45).    
\end{math}

\vskip 5pt

(12) $\be_{47} = \tfrac {1} {6} (-2,-2,4,-2,1,1,0,0)$, 
$\be_{48} = \tfrac {1} {6} (-2,1,1,-2,-2,4,0,0)$.

 


These cases are similar to the case (11). 






\vskip 5pt

(13) $\be_{49} = \tfrac {1} {6} (-2,-2,4,-2,-2,4,-3,3)$.  



We identify the element 
$(\diag(g_1,t_1),\diag(g_2,t_2),\diag(t_{31},t_{32}))
\in M_{[2],[2],[1]}$ with the element 
$g=(g_1,g_2,t_1,t_2,t_{31},t_{32})\in \gl_2^2\times \gl_1^4$. 
On $M^1_{\be_{49}}$, 
%
\begin{equation*}
\chi_{\be_{49}}(g) = ((\det g_1)(\det g_2))^{-2}(t_1t_2)^4t_{31}^{-3}t_{32}^3
= (t_1t_2)^6t_{32}^6. 
\end{equation*}

For $x\in Z_{\be_{49}}$, we put  
$P(x)=x_{332}$. Then 
$P(gx)=t_1t_2t_{32}P(x)$. 
So $P(x)$ is invariant under the action of $G_{\text{st},\be_{49}}$. 
This implies that $Z^{\sst}_{\be_{49}\,k}=\{x=(x_{332})\in k\mid x_{332}\not=0\}$.
This is clearly the orbit of $x$ such that $x_{332}=1$. 
Since $W_{\be_{49}}=\{0\}$, 
\begin{math}
Y^{\sst}_{\be_{49}\,k}
\end{math}
is also a single $P_{\be_{49}\, k}$-orbit. 

We can now state the main theorem for the case (1) 
more precisely as follows.

\begin{thm}
\label{thm:main1-detail}
\begin{itemize}
\item[(1)]
For the \pv{} (1), $S_{\be_i}\not=\empty$ 
if and only if $i$ is one of the numbers in (\ref{eq:list-non-empty}).
\item[(2)]
For $i$ in (\ref{eq:list-non-empty}), 
$S_{\be_i\,k}$ is a single $G_k$-orbit 
except for $i=29$.  
\item[(3)]
$G_k\backslash S_{\be_{29}\,k}$ 
is in bijective correspondence with $\Ex_2(k)$. 
\end{itemize}
\end{thm}

Considerations of this section proves the 
above theorem except for 
the ``only if'' part of (1). 
We shall prove the ``only if'' part of (1) 
in the next section.

\section{Empty strata for the case (1)}
\label{sec:empty-strata1}

In this section we prove that $S_{\be_i}=\emptyset$ 
for $i$ not in (\ref{eq:list-non-empty}) for 
the \pv{} (1).

The following lemmas are easy and we do not provide the proof. 

\begin{lem}
\label{lem:eliminate-standard}
Let $G=\spl_n$ and $V=\aff^n$ (the standard \rep). 
Then for any $x\in V_k$ there exists $g\in G_k$ such that 
$gx$ is in the form $[0\ccd 0,*]$. 
\end{lem}

\begin{lem}
\label{lem:eliminate-nm-matrix}
Let $n>m>0$, $G=\spl_n$ and $V=\m_{n,m}$.  
Then for any $A\in V_k$ there exists $g\in G_k$ such that 
if $B=(b_{ij})=gA$ then $b_{ij}=0$ for $i=1\ccd n-m,j=1\ccd m$. 
\end{lem}
\begin{lem}
\label{lem:eliminate-2m(2)}
Let $G=\spl_2\times \spl_2$ and $V=\aff^2\oplus \m_2$
where $(g_1,g_2)\in G$ acts on $V$ by $V\ni (v,A)\mapsto (g_1v,g_1A{}^tg_2)$.
Then for any element $(v,A)\in V_k$, there exists $g\in G_k$ such that 
if $g(v,A)=(w,B)$ then the first entry of $w$ and the $(1,1)$-entry of $B$ 
are $0$.  
\end{lem}
\begin{lem}
\label{lem:eliminate-2m(32)}
Let $G=\spl_3\times \spl_2$ and $V=\aff^2\oplus \m_{3,2}$
where $(g_1,g_2)\in G$ acts on $V$ by $V\ni (v,A)\mapsto (g_2v,g_1A{}^tg_2)$.
Then for any element $(v,A)\in V_k$, there exists $g\in G_k$ such that 
if $g(v,A)=(w,B)$ then the first entry of $w$ and the first row of $B$ 
are $0$.  
\end{lem}
\begin{lem}
\label{lem:eliminate-2m(3-32)}
Let $G=\spl_3\times \spl_2$ and $V=\aff^3\oplus \m_{3,2}$
where $(g_1,g_2)\in G$ acts on $V$ by $V\ni (v,A)\mapsto (g_1v,g_1A{}^tg_2)$.
Then for any element $(v,A)\in V$, there exists $g\in G$ such that 
if $g(v,A)=(w,B)$ then the first two entries of $w$ and the 
$(1,1)$-entry of $B$ are $0$.  
\end{lem}

The following lemma is Witt's theorem. 
\begin{lem}
\label{lem:alernating-matrix}
Let $n>0$ be an integer, $G=\spl_n$ and $V=\wedge^2 \aff^n$. 
We identify $V$ with the space of alternating matrices with diagonal 
entries $0$ (this assumption is necessary if $\ch(k)=2$).  
\begin{itemize}
\item[(1)]
If the rank of $A\in V_k$ is $m$ then $m$ is even.  
\item[(2)]
Suppose that $m=2l$. There exists $g\in G_k$ such that 
if $B=(b_{ij})=gA{}^tg$ then $b_{ij}=0$ unless 
$(i,j)=(n-2l+1,n-2l+2),(n-2l+2,n-2l+1)\ccd (n-1,n),(n,n-1)$. 
In particular, if $n$ is odd then the first row and the first column 
are zero.  
\end{itemize}
\end{lem}

We now explain our strategy of proving that $S_{\be_i}=\emptyset$. 
It is enough to prove that $Z^{\sst}_{\be_i}=\emptyset$. 

We remind the reader that $M^s_{\be_i}$ is the semi-simple part of $M_{\be_i}$. 
Since $M^s_{\be_i}$ is connected, it  has no non-trivial character and so  
$M^s_{\be_i}\sub G_{\text{st},\be_i}$. What we do is the 
following.

(1) For any $x\in Z_{\be_i}$, we find $g\in M^s_{\be_i}$ 
such that certain coordinates of $gx$ are $0$. 

(2) Assuming that certain coordinates of $x$ are $0$, 
we find a 1PS (one parameter subgroup) 
$\lam(t)$ of $G_{\text{st},\be_i}$ such that 
if $x_j$ is a non-zero coordinate of $x$ then 
the weight of $x_j$ with respect to $\lam(t)$ 
is positive, i.e., if $\coorde_j$ is the corresponding
coordinate vector then $\lam(t)\coorde_j=t^{a_j}\coorde_j$
with $a_j>0$.  

Note that if (1), (2) are carried out then $S_{\be_i}=\emptyset$. 

In the following table, for each $i$, we list 
which coordinates of $x$ we can eliminate and 
the 1PS with the property (2). For example, 
consider $\be_1$. 

\vskip 10pt

\begin{center}

\begin{tabular}{|l|l|l|}
\hline
\size $1$ & \size $M_{[2],\emptyset,[1]}$, \;  
$\Lam^{3,1}_{2,[1,3]}\oplus \Lam^{2,1}_{1,[1,2]}\otimes \Lam^{3,1}_{2,[1,3]}$ 
\hskip 0.05in $\;$
& \size $[2,2,-4,-5,-5,10,-4,4]$  \\
\cline{2-3} 
& \size $x_{311},x_{321}=0$ 
& \size $x_{331},x_{112},x_{122},x_{132},x_{212},x_{222},x_{232}$, \quad $[2,1,1,16,1,1,16]$ 
\hskip 0.05in $\;$ \\ 
\hline
\end{tabular}

\end{center}

\vskip 10pt
These entries mean the following.

\begin{itemize}
\item[(1)]
The stratum in question is $\be_1$. 
\item[(2)]
$M_{\be_1} = M_{[2],\emptyset,[1]}$.
\item[(3)]
$Z_{\be_1}\cong \Lam^{3,1}_{2,[1,3]}\oplus 
\Lam^{2,1}_{1,[1,2]}\otimes \Lam^{3,1}_{2,[1,3]}$ 
as \rep s of $M^s_{\be_1}$ (see (\ref{eq:Lam-defn})). 
\item[(4)]
The 1PS 
\begin{math}
\lam(t)=(\diag(t^2,t^2,t^{-4}),\diag(t^{-5},t^{-5},t^{10}),
\diag(t^{-4},t^4)) 
\end{math}
has the required property. 
\item[(5)]
We can make $x_{311},x_{321}=0$.  
\item[(6)]
The weights of $x_{331},x_{112},x_{122},x_{132},x_{212},x_{222},x_{232}$ 
are $2,1,1,16,1,1,16$ respectively. For example, 
$\lam(t)e_{232}=t^{16}e_{232}$. 
\end{itemize}

(5) will be proved after the table. (6) can be verified by hand easily, 
but we shall point out later that it can be verified by a maple program.  
We did not list $M^s_{\be_1}$ because we can easily 
determine $M^s_{\be_1}$ from $M_{\be_1}$. In this case 
the sizes of the blocks are $2,1$ in $G_1$, $3$ in $G_2$ 
and $1,1$ in $G_3$ and so 
$M^s_{\be_1}\cong \spl_3\times \spl_2$. 

There are requirements for the 1PS. Since $\lam(t)$ 
is a 1PS of $G_{\text{st},\be_1}\sub M^1_{\be_1}=M_{\be_1}\cap G_{\text{st}}$, 
the sum of exponents for each of $G_1,G_2,G_3$ is $0$. 
In this case $2+2-4=-5-5+10=-4+4=0$. 
If $\lam(t)$ is a 1PS of $M^1_{\be_1}$, 
it is contained in $G_{\text{st},\be_1}$ if and only if 
it is orthogonal to $\be_1$. In this case 
\begin{equation*}
(-2)\cdot 2+(-2)\cdot 2 + 4\cdot(-4)
+ 0\cdot (-5) + 0\cdot (-5) + 0\cdot 10
+ (-3)\cdot (-4)
+ 3\cdot 4=0.
\end{equation*}

If we find a 1PS with required properties 
then we are not obliged to show how we found it. 
However, we chose the case $\be_{10}$ to explain 
how we found the 1PS for the sake of the reader.  
For other $\be_i$'s, we only provide the required information. 

We consider the case $\be_{10}=\tfrac 1{114}(-2,-2,4,-6,0,6,-3,3)$. 

In this case $M^1_{\be_{10}}=M^1_{[2],[1,2],[1]}$. 
It is slightly easier to consider the semi-simple 
part and the torus part (the center) separately. The semi-simple part
is $\spl_2$ and the torus part is $\gl_1^4$. 
Elements of the center of $M^1_{[2],[1,2],[1]}$ 
can be expressed as 
\begin{equation*}
t = (\diag(t_1^{-1}I_2,t_1^2),\diag(t_2^{-1},t_3^{-1},t_2t_3),\diag(t_4^{-1},t_4)).
\end{equation*}
On the torus part, the character $\chi_{\be_{10}}$ is proportional to 
the character $\chi(t)=t_1^{12}t_2^{12}t_3^6t_4^6$. 
Since $G_{\text{st},\be_{10}}$ is the identity component of 
$\ker(\chi_{\be_{10}})$, we consider $t$ such that 
$t_1^2t_2^2t_3t_4=1$, i.e., $t_4=t_1^{-2}t_2^{-2}t_3^{-1}$.   
We also consider elements of the form 
$((t_5^{-1},t_5,1),I_3,I_2)\in M^s_{[2],[1,2],[1]}$.  

The coordinates of $Z_{\be_{10}}$ are 
$x_{131},x_{231},x_{321},x_{122},x_{222},x_{312}$.  
The action of the product of the above elements 
are scalar multiplications as in the following table. 

\vskip 10pt

\tiny

\begin{center}

\begin{tabular}{|c|c|c|c|c|c|}
\hline
$x_{131}$ & $x_{231}$ & $x_{321}$ & $x_{122}$ & $x_{222}$ & $x_{312}$ \\
\hline
$t_1t_2^3t_3^2t_5^{-1}$ & $t_1t_2^3t_3^2t_5$ 
& $t_1^4t_2^2$ 
& $t_1^{-3}t_2^{-2}t_3^{-2}t_5^{-1}$ & $t_1^{-3}t_2^{-2}t_3^{-2}t_5$
& $t_2^{-3}t_3^{-1}$ \\
\hline

\end{tabular}

\end{center}

\normalsize

\vskip 10pt

Both $\lan e_{131},e_{231}\ran$ and 
$\lan e_{122},e_{222}\ran$ 
are the standard \rep{} of $\spl_2$. 
So it is possible to make $x_{131}=0$ or $x_{122}=0$. 
We choose to make $x_{131}=0$.  In this case, to make $x_{122}=0$ 
does not work and one has to do trial and error in some cases.  
We would like to find a 1PS such that the weights of 
the coordinates except for $x_{131}$ are positive.

Let $c=[c_1,c_2,c_3,c_5]\in \Z^4$ and 
$\lam_c(t)$ be the following 1PS:
\begin{equation*}
(\diag(t^{-c_1-c_5},t^{-c_1+c_5},t^{2c_1}),
\diag(t^{-c_2},t^{-c_3},t^{c_2+c_3}),
\diag(t^{2c_1+2c_2+c_3},t^{-2c_1-2c_2-c_3})).
\end{equation*}
We put 
\begin{equation*}
v_1 = 
\begin{pmatrix}
1 \\ 3 \\ 2 \\ 1
\end{pmatrix}, \;
v_2 = 
\begin{pmatrix}
4 \\ 2 \\ 0 \\ 0
\end{pmatrix}, \;
v_{31} = 
\begin{pmatrix}
-3 \\ -2 \\ -2 \\ -1
\end{pmatrix}, \;
v_{32} = 
\begin{pmatrix}
-3 \\ -2 \\ -2 \\ 1
\end{pmatrix}, \;
v_4 = 
\begin{pmatrix}
0 \\ -3 \\ -1 \\ 0
\end{pmatrix},\; 
c = 
\begin{pmatrix}
c_1 \\ c_2 \\ c_3 \\ c_5
\end{pmatrix}
\end{equation*}
and $A= (v_1 \; v_2 \; v_{31} \; v_{32} \; v_4)$. 
Then the coordinates 
$x_{231},x_{321},x_{122},x_{222},x_{312}$ have positive
weights with respect to $\lam(t)$ 
if and only if all entries of ${}^tc A$ are positive. 

By the following sequence of MAPLE commands:

\vskip 10pt
\small
\begin{verbatim}
 > with(linalg):
 > v1:= matrix(4,1,[1,3,2,1]):
 > v2:= matrix(4,1,[4,2,0,0]):
 > v31:= matrix(4,1,[-3,-2,-2,-1]):
 > v32:= matrix(4,1,[-3,-2,-2,1]):
 > v4:= matrix(4,1,[0,-3,-1,0]):
 > A:= augment(v1,v2,v31,v32,v4):
 > rref(A);
\end{verbatim}
\normalsize

\vskip 10pt
\noindent
we find that $\{v_1,v_2,v_{31},v_{32}\}$ is linearly independent
and $v_4 = -(11/8)v_1-(5/16)v_2-(9/8)v_{31}+(1/4)v_4$. 

We can choose $c$ so that 
$a_1={}^tc v_1\ccd a_4={}^tc v_{32}$ are arbitrary positive 
numbers since $\{v_1,v_2,v_{31},v_{32}\}$ is linearly independent. 
Then 
\begin{equation*}
{}^tc v_4 = -\frac {11}8 a_1 - \frac 5{16} a_2 - \frac 98 a_3 + \frac 14 a_4.
\end{equation*}
So it is enough to choose $c$ so that $a_4$ is sufficiently large. 

For example, by the following MAPLE commands:

\vskip 10pt
\small
\begin{verbatim}
 > b:= matrix(4,1,[1,4,1,19]):
 > A:= augment(transpose(augment(v1,v2,v31,v32)),b);
 > rref(A);
\end{verbatim}
\normalsize

\vskip 10pt
\noindent
it turns out that we can choose $c=[0,2,-7,9]$. 
Then 
\begin{equation*}
\lam(t) = (\diag(t^{-9},t^9,1),
\diag(t^{-2},t^7,t^{-5}),
\diag(t^{-3},t^3)).
\end{equation*}
is a 1PS with the required properties. 

Once we find $\lam(t)$, 
it is easy to verify that the 1PS has 
the required properties by MAPLE commands as follows.

\vskip 10pt

\small
\begin{verbatim}
 > restart: with(linalg): 
 > read "Home/Strata/lib/maple/more332":
 > # beta10
 > b:= matrix(1,8,[-9,9,0,-2,7,-5,-3,3]):
 > beta:= matrix(8,1,[-2,-2,4,-6,0,6,-3,3]):
 > evalm(b &* augment(w6,w8,w11,w14,w16,beta));
\end{verbatim}

\normalsize
\vskip 10pt

``Home'' is the directory where the directory ``Strata'' 
is located. Weights of the coordinates of $V$ 
has to be written in the file more332. 
The result of the above computations is 
$[4,1,1,19,1,0]$. The first five entries are positive 
and the last entry shows that the vector 
$[-9,9,0,-2,7,-5,-3,3]$ is orthogonal to $\be_{10}$. 

The following table shows relevant informations  
for $i$ such that $S_i=\emptyset$. 

\vskip 10pt

\begin{center}
\begin{tabular}{|l|l|l|}
\hline
\size $i$ & $\;$ \hskip 0.8in \size $M_{\be_i}$, $Z_{\be_i}$ 
\hskip 0.8in $\;$
& $\;$ \hskip 1.3in \size 1PS \hskip 1.1in $\;$ \\
\cline{2-3}
& \hskip 0.7in \size zero coordinates 
& \hskip 0.5in \size non-zero coordinates and their weights \\
\hline \hline
\size $1$ & \size $M_{[2],\emptyset,[1]}$, \;
\begin{math}
\Lam^{3,1}_{2,[1,3]}\oplus 
\Lam^{2,1}_{1,[1,2]}\otimes \Lam^{3,1}_{2,[1,3]}
\end{math}
& \size $[2,2,-4,-5,-5,10,-4,4]$  \\
\cline{2-3}
& \size $x_{311},x_{321}=0$ 
& \size $x_{331},x_{112},x_{122},x_{132},x_{212},x_{222},x_{232}, \quad [2,1,1,16,1,1,16]$ \\
\hline \hline 
\size $2$ & \size $M_{[1],[2],[1]}$, \; 
$\Lam^{2,1}_{1,[2,3]}\oplus 1 \oplus 
\Lam^{2,1}_{1,[2,3]}\otimes \Lam^{2,1}_{2,[1,2]}$ 
& \size $[2,-5,3,-4,4,0,-2,2]$ \\
\cline{2-3}
& \size $x_{231},x_{212}=0$
& \size $x_{331},x_{132},x_{222},x_{312},x_{322}, \quad [1,4,1,1,9]$ \\
\hline \hline 
\size $3$ & \size $M_{[2],[1],[1]}$, \;
\begin{math}
\Lam^{2,1}_{2,[2,3]}\oplus 
\Lam^{2,1}_{1,[1,2]}\otimes \Lam^{2,1}_{2,[2,3]}\oplus 1
\end{math}
& \size $[-4,4,0,2,-5,3,-2,2]$  \\
\cline{2-3}
& \size $x_{321},x_{122}=0$
& \size $x_{331},x_{132},x_{222},x_{232},x_{312}, \quad [1,1,1,9,4]$ \\
\hline \hline 
\size $7$ & \size $M_{\emptyset,[2],[1]}$, \; 
$\Lam^{3,1}_{1,[1,3]}\oplus \Lam^{3,1}_{1,[1,3]}\otimes \Lam^{2,1}_{2,[1,2]}$  
& \size $[-5,-5,10,2,2,-4,-4,4]$  \\
\cline{2-3}
& \size $x_{131},x_{231}=0$
& \size $x_{331},x_{112},x_{122},x_{212},x_{222},x_{312},x_{322}, 
\quad [2,1,1,1,1,16,16]$ \\
\hline \hline 
\size $8$ & \size $M_{[2],[2],[1]}$, \;
$1\oplus \Lam^{2,1}_{1,[1,2]}\oplus \Lam^{2,1}_{2,[1,2]}$
& \size $[-5,3,2,-2,2,0,2,-2]$  \\
\cline{2-3}
& \size $x_{132},x_{312}=0$
& \size $x_{331},x_{232},x_{322}, \quad [4,1,2]$ \\
\hline \hline 
\size $9$ & \size $M_{[1,2],[2],[1]}$, \;
$1\oplus \Lam^{2,1}_{2,[1,2]} \oplus 1\oplus \Lam^{2,1}_{2,[1,2]}$ 
& \size $[-2,7,-5,-9,9,0,-3,3]$  \\
\cline{2-3}
& \size $x_{311}=0$ 
& \size $x_{231},x_{321},x_{132},x_{212},x_{222}, \quad [4,1,1,1,19]$ \\
\hline
\end{tabular}
\end{center}

\begin{center}
\begin{tabular}{|l|l|l|}
\hline
\size $i$ & $\;$ \hskip 0.8in \size $M_{\be_i}$, $Z_{\be_i}$ 
\hskip 0.8in $\;$
& $\;$ \hskip 1.3in \size 1PS \hskip 1.1in $\;$ \\
\cline{2-3}
& \hskip 0.7in \size zero coordinates 
& \hskip 0.5in \size non-zero coordinates and their weights \\
\hline \hline 
\size $10$ & \size $M_{[2],[1,2],[1]}$, \; 
$\Lam^{2,1}_{1,[1,2]}\oplus 1\oplus \Lam^{2,1}_{1,[1,2]}\oplus 1$
& \size $[-9,9,0,-2,7,-5,-3,3]$  \\
\cline{2-3}
& \size $x_{131}=0$ 
& \size $x_{231},x_{321},x_{122},x_{222},x_{312}, \quad [1,4,1,19,1]$ \\
\hline \hline 
\size $12$ & \size $M_{[1],\emptyset,[1]}$, \;
$\Lam^{2,1}_{1,[2,3]}\otimes \Lam^{3,1}_{2,[1,3]}$ 
& \size $[0,0,0,-2,1,1,0,0]$  \\
\cline{2-3}
& \size $x_{212},x_{312}=0$
& \size $x_{222},x_{232},x_{322},x_{332}, \quad [1,1,1,1]$ \\
\hline \hline 
\size $13$ & \size $M_{\emptyset,[1],[1]}$, \;
$\Lam^{3,1}_{1,[1,3]}\otimes \Lam^{2,1}_{2,[2,3]}$ 
& \size $[-2,1,1,0,0,0,0,0]$  \\
\cline{2-3}
& \size $x_{122},x_{132}=0$
& \size $x_{222},x_{232},x_{322},x_{332}, \quad [1,1,1,1]$ \\
\hline \hline 
\size $14$ & \size $M_{[1],[2],[1]}$, \;
$\Lam^{2,1}_{1,[2,3]}\oplus \Lam^{2,1}_{1,[2,3]}\otimes \Lam^{2,1}_{2,[1,2]}$
& \size $[0,-5,5,-4,6,-2,-2,2]$  \\
\cline{2-3}
& \size $x_{231},x_{212}=0$
& \size $x_{331},x_{222},x_{312},x_{322}, \quad [1,3,3,13]$ \\
\hline \hline 
\size $15$ & \size $M_{[2],[1],[1]}$, \; 
$\Lam^{2,1}_{2,[2,3]}\oplus \Lam^{2,1}_{1,[1,2]}\otimes \Lam^{2,1}_{2,[2,3]}$
& \size $[-4,6,-2,0,-5,5,-2,2]$  \\
\cline{2-3}
& \size $x_{321},x_{122}=0$
& \size $x_{331},x_{132},x_{222},x_{232}, \quad [1,3,3,13]$ \\
\hline \hline 
\size $16$ & \size $M_{[2],[2],[1]}$, \;
$\Lam^{2,1}_{1,[1,2]}\oplus \Lam^{2,1}_{2,[1,2]}$
& \size $[-1,1,0,-1,1,0,0,0]$  \\
\cline{2-3}
& \size $x_{132},x_{312}=0$
& \size $x_{232},x_{322}, \quad [1,1]$ \\
\hline \hline 
\size $17$ & \size $M_{[1,2],[2],[1]}$, \;
$1^{2\oplus}\oplus \Lam^{2,1}_{2,[1,2]}$
& \size $[0,1,-1,-5,3,2,1,-1]$  \\
\cline{2-3}
& \size $x_{312}=0$
& \size $x_{331},x_{232},x_{322}, \quad [2,2,1]$ \\
\hline \hline 
\size $18$ & \size $M_{[1,2],[2],[1]}$, \;
$1^{2\oplus}\oplus \Lam^{2,1}_{2,[1,2]}$
& \size $[0,3,-3,-3,3,0,-1,1]$  \\
\cline{2-3}
& \size $x_{312}=0$
& \size $x_{231},x_{132},x_{322}, \quad [2,1,1]$ \\
\hline \hline 
\size $19$ & \size $M_{[2],\emptyset,\emptyset}$, \;
$\Lam^{3,1}_{2,[1,3]}\otimes \Lam^{2,1}_{3,[1,2]}$
& \size $[0,0,0,-2,1,1,0,0]$  \\
\cline{2-3}
& \size $x_{311},x_{312}=0$
& \size $x_{321},x_{331},x_{322},x_{332}, \quad [1,1,1,1]$ \\
\hline \hline 
\size $20$ & \size $M_{[1,2],[2],\emptyset}$, \; 
$\Lam^{2,1}_{3,[1,2]}\oplus \Lam^{2,1}_{2,[1,2]}\otimes \Lam^{2,1}_{3,[1,2]}$
& \size $[0,-2,2,-4,6,-2,-5,5]$ \\
\cline{2-3}
& \size $x_{231},x_{311}=0$
& \size $x_{331},x_{232},x_{312},x_{322}, \quad [3,9,3,13]$ \\
\hline \hline 
\size $21$ & \size $M_{[1,2],[1],[1]}$, \;
$\Lam^{2,1}_{2,[2,3]}\oplus \Lam^{2,1}_{2,[2,3]}\oplus 1$ 
& \size $[6,0,-6,-4,-9,13,-5,5]$ \\
\cline{2-3}
& \size $x_{321}=0$
& \size $x_{331},x_{122},x_{132},x_{212}, \quad [2,2,24,1]$ \\
\hline \hline 
\size $23$ & \size $M_{[2],[1,2],[1]}$, \;
$1\oplus \Lam^{2,1}_{1,[1,2]}\oplus 1$
& \size $[-5,3,2,0,1,-1,1,-1]$ \\
\cline{2-3}
& \size $x_{132}=0$
& \size $x_{331},x_{232},x_{322}, \quad [2,1,2]$ \\
\hline \hline 
\size $24$ & \size $M_{[2],[1,2],[1]}$, \;
$1\oplus \Lam^{2,1}_{1,[1,2]}\oplus 1$
& \size $[-3,3,0,0,3,-3,-1,1]$ \\
\cline{2-3}
& \size $x_{132}=0$
& \size $x_{321},x_{232},x_{312}, \quad [2,1,1]$ \\
\hline \hline 
\size $25$ & \size $M_{\emptyset,[2],\emptyset}$, \;
$\Lam^{3,1}_{1,[1,3]}\otimes \Lam^{2,1}_{3,[1,2]}$
& \size $[-2,1,1,0,0,0,0,0]$ \\
\cline{2-3}
& \size $x_{131},x_{132}=0$
& \size $x_{231},x_{331},x_{232},x_{332}, \quad [1,1,1,1]$ \\
\hline \hline 
\size $26$ & \size $M_{[2],[1,2],\emptyset}$, \;
$\Lam^{2,1}_{1,[1,2]}\otimes \Lam^{2,1}_{3,[1,2]}\oplus \Lam^{2,1}_{3,[1,2]}$
& \size $[-4,6,-2,0,-2,2,-5,5]$ \\
\cline{2-3}
& \size $x_{131},x_{321}=0$
& \size $x_{231},x_{132},x_{232},x_{322}, \quad [3,3,13,1]$ \\
\hline \hline 
\size $27$ & \size $M_{[1],[1,2],[1]}$, \;
$\Lam^{2,1}_{1,[2,3]}\oplus 1 \oplus \Lam^{2,1}_{1,[2,3]}$
& \size $[-4,-9,13,6,0,-6,-5,5]$ \\
\cline{2-3}
& \size $x_{231}=0$
& \size $x_{331},x_{122},x_{212},x_{312}, \quad [2,1,2,24]$ \\
\hline \hline 
\size $30$ & \size $M_{[2],\emptyset,[1]}$, \;
$\Lam^{3,1}_{2,[1,3]}$ 
& \size $[0,0,0,-1,0,1,0,0]$ \\
\cline{2-3}
& \size $x_{312},x_{322}=0$
& \size $x_{332}, \quad [1]$ \\
\hline \hline 
\size $31$ & \size $M_{[1,2],[2],[1]}$, \;
$1\oplus \Lam^{2,1}_{2,[1,2]}$ 
& \size $[-1,4,-3,-4,4,0,0,0]$ \\
\cline{2-3}
& \size $x_{312}=0$
& \size $x_{232},x_{322}, \quad [4,1]$ \\
\hline \hline 
\size $32$ & \size $M_{[2],[2],[1]}$, \;
$1\oplus \Lam^{2,1}_{2,[1,2]}$ 
& \size $[0,0,0,-5,3,2,2,-2]$ \\
\cline{2-3}
& \size $x_{312}=0$
& \size $x_{331},x_{322}, \quad [4,1]$ \\
\hline \hline 
\size $33$ & \size $M_{[1,2],[2],[1]}$, \;
$1\oplus \Lam^{2,1}_{2,[1,2]}$ 
& \size $[0,0,0,-7,5,2,4,-4]$ \\
\cline{2-3}
& \size $x_{312}=0$
& \size $x_{231},x_{322}, \quad [6,1]$ \\
\hline
\end{tabular}

\end{center}

\begin{center}

\begin{tabular}{|l|l|l|}
\hline
\size $i$ & $\;$ \hskip 0.8in \size $M_{\be_i}$, $Z_{\be_i}$ 
\hskip 0.8in $\;$
& $\;$ \hskip 1.3in \size 1PS \hskip 1.1in $\;$ \\
\cline{2-3}
& \hskip 0.7in \size zero coordinates 
& \hskip 0.5in \size non-zero coordinates and their weights \\
\hline \hline
\size $34$ & \size $M_{\emptyset,[2],[1]}$, \;
$\Lam^{3,1}_{1,[1,3]}$ 
& \size $[-1,0,1,0,0,0,0,0]$  \\
\cline{2-3}
& \size $x_{132},x_{232}=0$
& \size $x_{332}, \quad [1]$ \\
\hline \hline 
\size $35$ & \size $M_{[2],[1,2],[1]}$, \;
$\Lam^{2,1}_{1,[1,2]}\oplus 1$ 
& \size $[-4,4,0,-1,4,-3,0,0]$ \\
\cline{2-3}
& \size $x_{132}=0$
& \size $x_{232},x_{322}, \quad [8,4]$ \\
\hline \hline 
\size $36$ & \size $M_{[2],[2],[1]}$, \;
$1\oplus \Lam^{2,1}_{1,[1,2]}$ 
& \size $[-5,3,2,0,0,0,2,-2]$ \\
\cline{2-3}
& \size $x_{132}=0$
& \size $x_{331},x_{232}, \quad [4,1]$ \\
\hline \hline 
\size $37$ & \size $M_{[2],[1,2],[1]}$, \;
$1\oplus \Lam^{2,1}_{1,[1,2]}$  
& \size $[-7,5,2,0,0,0,4,-4]$  \\
\cline{2-3}
& \size $x_{132}=0$
& \size $x_{321},x_{232}, \quad [6,1]$ \\
\hline \hline 
\size $43$ & \size $M_{[2],[1],[1]}$, \;
$\Lam^{2,1}_{2,[2,3]}$  
& \size $[0,0,0,0,-1,1,0,0]$ \\
\cline{2-3}
& \size $x_{322}=0$
& \size $x_{332}, \quad [1]$ \\
\hline \hline 
\size $44$ & \size $M_{[1],[2],[1]}$, \;
$\Lam^{2,1}_{1,[2,3]}$ 
& \size $[0,-1,1,0,0,0,0,0]$ \\
\cline{2-3}
& \size $x_{232}=0$
& \size $x_{332}, \quad [1]$ \\
\hline \hline 
\size $46$ & \size $M_{[2],[2],\emptyset}$, \;
$\Lam^{2,1}_{3,[1,2]}$ 
& \size $[0,0,0,0,0,0,-1,1]$ \\
\cline{2-3}
& \size $x_{331}=0$
& \size $x_{332}, \quad [1]$ \\
\hline

\end{tabular}

\end{center}

\vskip 10pt

We shall verify that it is possible to eliminate coordinates 
as in the above table. 

\vskip 10pt

(1) $\beta_1=\tfrac {1} {42} (-2,-2,4,0,0,0,-3,3)$, 
$\beta_7=\tfrac {1} {42} (0,0,0,-2,-2,4,-3,3)$. 
We only consider $\be_1$. 

$Z_{\be_1}$ is spanned by $\coorde_i=e_{jkl}$ 
for the following $i,jkl$. 

\tiny

\vskip 5pt

\begin{center}

\begin{tabular}{|c|c|c|c|c|c|c|c|c|c|}
\hline
$i$ & $\underline 7$ & $\underline 8$ & $9$ & $10$ 
& $11$ & $12$ & $13$ & $14$ & $15$ \\
\hline
$jkl$ & $\underline{311}$ & $\underline{321}$ & $331$ & $112$ 
& $122$ & $132$ & $212$ & $222$ & $232$ \\
\hline
\end{tabular}

\end{center}

\normalsize

\vskip 5pt


The vertical columns mean $\coorde_7=e_{311}$ for example.
Underlines mean that the corresponding coordinates are eliminated. 
We are writing this kind of correspondence 
for the following reason. 

\vskip 10pt
\begin{itemize}
\item[(a)]
The numbering such as $e_{311}$ is convenient 
to determine $Z_{\be_i}$ as a \rep{} of $M^s_{\be_i}$. 
\item[(b)]
We made sure that the 1PS we found has the correct property 
by MAPLE and the numbering such as $\coorde_7$ 
is more convenient. 
\end{itemize}

\vskip 0pt
Lemma \ref{lem:eliminate-standard} implies that 
we may assume that $x_{311},x_{321}=0$. 

%
%
%
%

\vskip 5pt
(2) $\be_2 = \tfrac {1} {22} (-4,2,2,-2,-2,4,-3,3)$,  
$\be_3 = \tfrac {1} {22} (-2,-2,4,-4,2,2,-3,3)$. 
We only consider $\be_2$. 

$Z_{\be_2}$ is spanned by $\coorde_i=e_{jkl}$ 
for the following $i,jkl$. 

\tiny
\vskip 5pt

\begin{center}
 
\begin{tabular}{|c|c|c|c|c|c|c|c|}
\hline
$i$ & $\underline 6$ & $9$ & $12$ & $\underline{13}$ & $14$ & $16$ & $17$ \\
\hline
$jkl$ & $\underline{231}$ & $331$ & $132$ 
& $\underline{212}$ & $222$ & $312$ & $322$ \\
\hline
\end{tabular}

\end{center}
\normalsize

\vskip 5pt


Lemma \ref{lem:eliminate-2m(2)}
implies that we may assume that $x_{231},x_{212}=0$. 

%
%
%
%
%
%
%
%
%
%

\vskip 5pt

\vskip 5pt

(3) $\be_8 = \tfrac {5} {66} (-2,-2,4,-2,-2,4,-3,3)$.  

$Z_{\be_8}$ is spanned by $\coorde_i=e_{jkl}$ 
for the following $i,jkl$. 
\tiny
%
%
\begin{tabular}{|c|c|c|c|c|c|c|}
\hline
$i$ & $9$ & $\underline{12}$ & $15$ & $\underline{16}$ & $17$ \\
\hline
$jkl$ & $331$ & $\underline{132}$ & $232$ & $\underline{312}$ & $322$  \\
\hline
\end{tabular}
%
%
\normalsize

\vskip 5pt

%
Since $\lan e_{132},e_{232}\ran$, 
$\lan e_{312},e_{322}\ran$
are standard \rep s of two $\spl_2$'s, 
we may assume that $x_{132},x_{312}=0$ 
by Lemma \ref{lem:eliminate-standard}. 

%
%
%
%
%
%

\vskip 5pt

(4) $\be_9 = \tfrac {1} {114} (-6,0,6,-2,-2,4,-3,3)$,  
$\be_{10} = \tfrac {1} {114} (-2,-2,4,-6,0,6,-3,3)$. 
We only consider $\be_9$. 


$Z_{\be_9}$ is spanned by $\coorde_i=e_{jkl}$ 
for the following $i,jkl$. 
\tiny
%
%
\begin{tabular}{|c|c|c|c|c|c|c|}
\hline
$i$ & $6$ & $\underline{7}$ & $8$ & $12$ & $13$ & $14$ \\
\hline
$jkl$ & $231$ & $\underline{311}$ & $321$ & $132$ & $212$ & $222$ \\
\hline
\end{tabular}
%
%
\normalsize

\vskip 5pt

By applying Lemma \ref{lem:eliminate-standard} 
to $\lan e_{311},e_{321}\ran$
(and not to $\lan e_{212},e_{222}\ran$), 
we may assume that $x_{311}=0$.

%
%
%
%
%
%
%
%
%
%

\vskip 5pt



%
%
%
%
%
%
%
%

\vskip 5pt

(5) $\be_{12} = \tfrac {1} {6} (-2,1,1,0,0,0,-3,3)$,  
$\be_{13} = \tfrac {1} {6} (0,0,0,-2,1,1,-3,3)$. 
We only consider $\be_{12}$.  


$Z_{\be_{12}}$ is spanned by $e_i=e_{jkl}$ 
for the following $i,jkl$. 
\tiny
%
%
\begin{tabular}{|c|c|c|c|c|c|c|}
\hline
$i$ & $\underline{13}$ & $14$ & $15$ & $\underline{16}$ & $17$ & $18$ \\
\hline
$jkl$ & $\underline{212}$ & $222$ & $232$ 
& $\underline{312}$ & $322$ & $332$ \\
\hline
\end{tabular}
%
%
\normalsize

\vskip 5pt

It can be identified with $\m_{3,2}$. 
Lemma \ref{lem:eliminate-nm-matrix} implies that
we may assume that $x_{212},x_{312}=0$.  

%

\vskip 5pt




\vskip 5pt

(6) $\be_{14} = \tfrac {1} {42} (-14,7,7,-2,-2,4,-3,3)$,  
$\be_{15} = \tfrac {1} {42} (-2,-2,4,-14,7,7,-3,3)$. 
We only consider $\be_{14}$. 


$Z_{\be_{14}}$ is spanned by $e_i=e_{jkl}$ 
for the following $i,jkl$. 
\tiny
%
%
\begin{tabular}{|c|c|c|c|c|c|c|}
\hline
$i$ & $\underline{6}$ & $9$ & $\underline{13}$ & $14$ & $16$ & $17$ \\
\hline
$jkl$ & $\underline{231}$ & $331$ & $\underline{212}$ 
& $222$ & $312$ & $322$ \\
\hline
\end{tabular}
%
%
\normalsize

\vskip 5pt


%
%
%
%
%
%


Lemma \ref{lem:eliminate-2m(2)} implies that
we may assume that $x_{231},x_{212}=0$.

\vskip 5pt

%
%
%
%
%
%
%



\vskip 5pt

(7) $\be_{16} = \tfrac {1} {12} (-1,-1,2,-1,-1,2,-6,6)$.  

$Z_{\be_{16}}$ is spanned by $\coorde_i=e_{jkl}$ 
for the following $i,jkl$. 
\tiny
%
%
\begin{tabular}{|c|c|c|c|c|c|c|}
\hline
$i$ & $\underline{12}$ & $15$ & $\underline{16}$ & $17$ \\
\hline
$jkl$ & $\underline{132}$ & $232$ & $\underline{312}$ & $322$ \\
\hline
\end{tabular}
%
%
\normalsize

\vskip 5pt

Since $\lan e_{132},e_{232}\ran$, 
$\lan e_{312},e_{322}\ran$
are standard \rep s of two $\spl_2$'s, 
we may assume that $x_{132},x_{312}=0$ 
by Lemma \ref{lem:eliminate-standard}.

%
%
%
%
%

\vskip 5pt

(8) $\be_{17} = \tfrac {1} {30} (-10,-1,11,-4,-4,8,-6,6)$,  
$\be_{23} = \tfrac {1} {30} (-4,-4,8,-10,-1,11,-6,6)$. 
We only consider $\be_{17}$.  

%

$Z_{\be_{17}}$ is spanned by $\coorde_i=e_{jkl}$ 
for the following $i,jkl$. 
\tiny
%
%
\begin{tabular}{|c|c|c|c|c|c|c|}
\hline
$i$ & $9$ & $15$ & $\underline{16}$ & $17$ \\
\hline
$jkl$ & $331$ & $232$ & $\underline{312}$ & $322$ \\
\hline
\end{tabular}
%
%
\normalsize

\vskip 5pt

Lemma \ref{lem:eliminate-standard} implies that 
we may assume that $x_{312}=0$.

%
%
%
%

%
%
%
%
%

\vskip 5pt

(9) $\be_{18} = \tfrac {1} {78} (-14,4,10,-8,-8,16,-9,9)$, 
$\be_{24} = \tfrac {1} {78} (-8,-8,16,-14,4,10,-9,9)$. 
We only consider $\be_{18}$. 


$Z_{\be_{18}}$ is spanned by $\coorde_i=e_{jkl}$ 
for the following $i,jkl$. 
\tiny
%
%
\begin{tabular}{|c|c|c|c|c|c|c|}
\hline
$i$ & $6$ & $12$ & $\underline{16}$ & $17$ \\
\hline
$jkl$ & $231$ & $132$ & $\underline{312}$ & $322$ \\
\hline
\end{tabular}
%
%
\normalsize

\vskip 5pt

Lemma \ref{lem:eliminate-standard} implies that 
we may assume that $x_{312}=0$.

%
%
%
%
%
%
%
%
%

\vskip 5pt

(10) $\be_{19} = \tfrac {1} {3} (-1,-1,2,0,0,0,0,0)$, 
$\be_{25} = \tfrac {1} {3} (0,0,0,-1,-1,2,0,0)$. 
We only consider $\be_{19}$. 

%
%

$Z_{\be_{19}}$ is spanned by $\coorde_i=e_{jkl}$ 
for the following $i,jkl$. 
\tiny
%
%
\begin{tabular}{|c|c|c|c|c|c|c|}
\hline
$i$ & $\underline{7}$ & $8$ & $9$ & $\underline{16}$ & $17$ & $18$ \\
\hline
$jkl$ & $\underline{311}$ & $321$ & $331$ 
& $\underline{312}$ & $322$ & $332$ \\
\hline
\end{tabular}
%
%
\normalsize

\vskip 5pt

%
$Z_{\be_{19}}$ can be identified with $\m_{3,2}$. 
So Lemma \ref{lem:eliminate-nm-matrix} implies that 
we may assume that $x_{311},x_{312}=0$.

\vskip 5pt

(11) $\be_{20} = \tfrac {1} {21} (-7,2,5,-1,-1,2,0,0)$,  
$\be_{26} = \tfrac {1} {21} (-1,-1,2,-7,2,5,0,0)$.


$Z_{\be_{20}}$ is spanned by $\coorde_i=e_{jkl}$ 
for the following $i,jkl$. 
\tiny
%
%
\begin{tabular}{|c|c|c|c|c|c|c|}
\hline
$i$ & $\underline{6}$ & $\underline{7}$ & $8$ & $15$ & $16$ & $17$ \\
\hline
$jkl$ & $\underline{231}$ & $\underline{311}$ & $321$ 
& $232$ & $312$ & $322$ \\
\hline
\end{tabular}
%
%
\normalsize

\vskip 5pt

Lemma \ref{lem:eliminate-2m(2)} implies that 
we may assume that $x_{231},x_{311}=0$

%
%
%
%
%
%
%
%

\vskip 5pt

(12) $\be_{21} = \tfrac {1} {78} (-5,-2,7,-2,1,1,-6,6)$,  
$\be_{27} = \tfrac {1} {78} (-2,1,1,-5,-2,7,-6,6)$. 
We only consider $\be_{21}$. 


$Z_{\be_{21}}$ is spanned by $\coorde_i=e_{jkl}$ 
for the following $i,jkl$. 
\tiny
%
%
\begin{tabular}{|c|c|c|c|c|c|c|}
\hline
$i$ & $\underline{8}$ & $9$ & $11$ & $12$ & $13$ \\
\hline
$jkl$ & $\underline{321}$ & $331$ & $122$ & $132$ & $212$ \\
\hline
\end{tabular}
%
%
\normalsize

\vskip 5pt

We apply Lemma \ref{lem:eliminate-standard} 
to $\lan e_{321},e_{331}\ran$ and 
we may assume that $x_{321}=0$. 

%
%
%
%
%
%
%
%
%

\vskip 5pt

\vskip 5pt

(13) $\be_{30}= \tfrac {1} {6} (-2,-2,4,0,0,0,-3,3)$,  
$\be_{34}= \tfrac {1} {6} (0,0,0,-2,-2,4,-3,3)$. 
We only consider $\be_{30}$. 


$Z_{\be_{30}}$ is spanned by $\coorde_i=e_{jkl}$ 
for the following $i,jkl$. 
\tiny
%
%
\begin{tabular}{|c|c|c|c|}
\hline
$i$ & $\underline{16}$ & $\underline{17}$ & $18$ \\
\hline
$jkl$ & $\underline{312}$ & $\underline{322}$ & $332$ \\
\hline
\end{tabular}

%
\normalsize

\vskip 5pt

%
%
Lemma \ref{lem:eliminate-standard} implies that 
we may assume that $x_{312}=x_{322}=0$. 

%
%
%

\vskip 5pt

(14) $\be_{31} = \tfrac {1} {42} (-14,4,10,-2,-2,4,-21,21)$, 
$\be_{35} = \tfrac {1} {42} (-2,-2,4,-14,4,10,-21,21)$. 
We only consider $\be_{31}$. 


$Z_{\be_{31}}$ is spanned by $\coorde_i=e_{jkl}$ 
for the following $i,jkl$. 
\tiny
%
%
\begin{tabular}{|c|c|c|c|c|c|c|}
\hline
$i$ & $15$ & $\underline{16}$ & $17$ \\
\hline
$jkl$ & $232$ & $\underline{312}$ & $322$ \\
\hline
\end{tabular}

%
\normalsize

\vskip 5pt

%
Lemma \ref{lem:eliminate-standard} implies that 
we may assume that $x_{312}=0$. 

%
%
%
%
%
%
%
%

\vskip 5pt

(15) $\be_{32} = \tfrac {1} {42} (-14,-14,28,-2,-2,4,-3,3)$, 
$\be_{36} = \tfrac {1} {42} (-2,-2,4,-14,-14,28,-3,3)$. 
We only consider $\be_{32}$.  


$Z_{\be_{32}}$ is spanned by $\coorde_i=e_{jkl}$ 
for the following $i,jkl$. 
\tiny
%
%
\begin{tabular}{|c|c|c|c|c|c|c|}
\hline
$i$ & $9$ & $\underline{16}$ & $17$ \\
\hline
$jkl$ & $331$ & $\underline{312}$ & $322$ \\
\hline
\end{tabular}

%
\normalsize

\vskip 5pt

%
Lemma \ref{lem:eliminate-standard} implies that 
$x_{312}=0$. 

%
%
%
%
%
%
%
%
%
%
%

\vskip 5pt

(16) $\be_{33} = \tfrac {1} {66} (-22,8,14,-4,-4,8,-3,3)$, 
$\be_{37} = \tfrac {1} {66} (-4,-4,8,-22,8,14,-3,3)$. 
We only consider $\be_{33}$. 


$Z_{\be_{33}}$ is spanned by $\coorde_i=e_{jkl}$ 
for the following $i,jkl$. 
\tiny
%
%
\begin{tabular}{|c|c|c|c|c|c|c|}
\hline
$i$ & $6$ & $\underline{16}$ & $17$ \\
\hline
$jkl$ & $231$ & $\underline{312}$ & $322$ \\
\hline
\end{tabular}

%
\normalsize

\vskip 5pt

%
Lemma \ref{lem:eliminate-standard} implies that 
$x_{312}=0$. 

%
%
%
%
%
%
%
%
%
%

\vskip 5pt

\vskip 5pt

(17)  $\be_{43} =  \tfrac {1} {6} (-2,-2,4,-2,1,1,-3,3)$, 
$\be_{44} =  \tfrac {1} {6} (-2,1,1,-2,-2,4,-3,3)$. 
We only consider $\be_{43}$. 

%

$Z_{\be_{43}}$ is spanned by $\coorde_i=e_{jkl}$ 
for the following $i,jkl$. 
\tiny
%
%
\begin{tabular}{|c|c|c|c|c|c|c|}
\hline
$i$ & $\underline{17}$ & $18$ \\
\hline
$jkl$ & $\underline{322}$ & $332$ \\
\hline
\end{tabular}

%
\normalsize

\vskip 5pt

%
Lemma \ref{lem:eliminate-standard} implies that 
$x_{322}=0$.

%
%
%
%
%


%


%
%
%
%
%

\vskip 5pt

(33) $\be_{46} = \tfrac {1} {3} (-1,-1,2,-1,-1,2,0,0)$. 
%

$Z_{\be_{46}}$ is spanned by $\coorde_i=e_{jkl}$ 
for the following $i,jkl$. 
\tiny
%
%
\begin{tabular}{|c|c|c|c|c|c|c|}
\hline
$i$ & $\underline{9}$ & $18$ \\
\hline
$jkl$ & $\underline{331}$ & $332$ \\
\hline
\end{tabular}

%
\normalsize

\vskip 5pt

%
Lemma \ref{lem:eliminate-standard} implies that 
we may assume that $x_{331}=0$.

%
%
%
%

\section{Non-empty strata for the case (2)}
\label{sec:non-empty-strata2}

In this section and the next, we put $G_1=\gl_6,G_2=\gl_2$ 
and $G=G_1\times G_2$.  We consider the case (2). 
As before, we identify $\wedge^2 \aff^6$ with the space of 
alternating $6\times 6$ matrices with diagonal entries $0$.  
The set $\gB$ consists 
of $81$ $\be_i$'s. We use the table in 
Section 8 \cite{tajima-yukie-GIT1}. 
We shall prove that $S_{\be_i}\not=\emptyset$ for 
\begin{equation}
\label{eq:list-non-empty-62}
i = 8,13,18,35,46,66,67,75,76,78,80,81
\end{equation}
for the \pv{} (2) in this section. We shall 
prove that $S_{\be_i}=\emptyset$ for other $\be_i$'s in the next section
We proceed as in Section \ref{sec:non-empty-strata1}. 

The following table describes $M_{\be}$, $Z_{\be}$ as a 
\rep{} of $M_{\be}$, the coordinates of $Z_{\be},W_{\be}$ 
and $G_k\backslash S_{\be\, k}\cong P_{\be\,k}\backslash Y^{\sst}_{\be\,k}$. 

\vskip 5pt

\begin{center} 
 
\begin{tabular}{|c|l|l|}
\hline 
\size $i$ & \hskip 1in \size $M_{\be_i}$
& \hskip 0.9in \size $Z_{\be_i}$ \\
\hline
\size $G_k\backslash S_{\be_i\, k}$  & 
\hskip 0.5in \size
coordinates of $Z_{\be_i}$
& \size \hskip 0.4in
coordinates of $W_{\be_i}$  \\
\hline \hline 
\size $8$ & \size $M_{\emptyset,[1]}\cong \gl_6\times \gl_1^2$
& \size $\Lam^{6,2}_{1,[1,6]}$ \\
\hline
\size SP & \size
\begin{math}
\begin{matrix}
x_{122},x_{132},x_{142},x_{152},x_{162},x_{232},x_{242},x_{252} \\
x_{262},x_{342},x_{352},x_{362},x_{452},x_{462},x_{562} \hfill
\end{matrix}
\end{math} 
& \size \hskip 0.8in -  \\
\hline \hline 
\size $13$ & \size $M_{[2],[1]}\cong \gl_4\times \gl_2\times \gl_1^2$
& \size $\Lam^{4,2}_{1,[3,6]}
\oplus \Lam^{2,1}_{1,[1,2]}\otimes \Lam^{4,1}_{1,[3,6]}$ \\
\hline
\size SP & \size
\begin{math}
\begin{matrix}
x_{341},x_{351},x_{361},x_{451},x_{461},x_{561},x_{132}, \\
x_{142},x_{152},x_{162},x_{232},x_{242},x_{252},x_{262} \hfill
\end{matrix}
\end{math} 
& \size $x_{342},x_{352},x_{362},x_{452},x_{462},x_{562}$  \\
\hline \hline 
\size $18$ & \size $M_{[4],\emptyset}\cong \gl_4\times \gl_2^2$
& \size $\Lam^{4,1}_{1,[1,4]}\otimes \Lam^{2,1}_{1,[5,6]}
\otimes \Lam^{2,1}_{2,[1,2]}$ \\
\hline
\size SP & \size
\begin{math}
\begin{matrix}
x_{151},x_{161},x_{251},x_{261},x_{351},x_{361},x_{451},x_{461}, \\
x_{152},x_{162},x_{252},x_{262},x_{352},x_{362},x_{452},x_{462} \hfill
\end{matrix}
\end{math} 
& \size $x_{561},x_{562}$ \\
\hline \hline 
\size $35$ & \size $M_{[1,4],\emptyset}\cong \gl_3\times \gl_2^2\times \gl_1$
& \size 
$\Lam^{3,1}_{1,[2,4]}\otimes \Lam^{2,1}_{1,[5,6]}\otimes \Lam^{2,1}_{2,[1,2]}$ \\
\hline
\size SP & \size
\begin{math}
\begin{matrix}
x_{251},x_{261},x_{351},x_{361},x_{451},x_{461}, \\
x_{252},x_{262},x_{352},x_{362},x_{452},x_{462} \hfill
\end{matrix}
\end{math} 
& \size $x_{561},x_{562}$ \\
\hline \hline 
\size $46$ & \size $M_{[2,4],[1]}\cong \gl_2^3\times \gl_1^2$
& \size 
$\Lam^{2,1}_{1,[3,4]}\otimes \Lam^{2,1}_{1,[5,6]}\oplus 
\Lam^{2,1}_{1,[1,2]}\otimes \Lam^{2,1}_{1,[5,6]}\oplus 1$ \\
\hline
\size SP & \size
\begin{math}
\begin{matrix}
x_{351},x_{361},x_{451},x_{461}, \hfill \\
x_{152},x_{162},x_{252},x_{262},x_{342}
\end{matrix}
\end{math} 
& \size $x_{561},x_{352},x_{362},x_{452},x_{462},x_{562}$ \\
\hline \hline 
\size $66$ & \size $M_{[2,4],[1]}\cong \gl_2^3\times \gl_1^2$
& \size 
$1\oplus \Lam^{2,1}_{1,[1,2]}\otimes \Lam^{2,1}_{1,[5,6]}\oplus 1$ \\
\hline
\size SP & \size
\begin{math}
\begin{matrix}
x_{561},x_{152},x_{162},x_{252},x_{262},x_{342}
\end{matrix}
\end{math} 
& \size $x_{352},x_{362},x_{452},x_{462},x_{562}$ \\
\hline \hline 
\size $67$ & \size $M_{[2],\emptyset}\cong \gl_4\times \gl_2^2$
& \size 
$\Lam^{2,1}_{2,[1,2]}\otimes \Lam^{4,2}_{1,[3,6]}$ \\
\hline
\size $\Ex_2(k)$ & \size
\begin{math}
\begin{matrix}
x_{341},x_{351},x_{361},x_{451},x_{461},x_{561}, \\
x_{342},x_{352},x_{362},x_{452},x_{462},x_{562} \hfill
\end{matrix}
\end{math} 
& \size \hskip 0.8in -  \\
\hline \hline 
\size $74$ & \size $M_{[2,4],[1]}\cong \gl_2^3\times \gl_1^2$
& \size 
$1\oplus \Lam^{2,1}_{1,[3,4]}\otimes\Lam^{2,1}_{1,[5,6]}$ \\
\hline
\size SP & \size
\begin{math}
\begin{matrix}
x_{561},x_{352},x_{362},x_{452},x_{462}
\end{matrix}
\end{math} 
& \size $x_{562}$ \\
\hline \hline 
\size $75$ & \size $M_{[1,2,5],[1]}\cong \gl_3\times \gl_1^5$
& \size 
$\Lam^{3,1}_{1,[3,5]}\oplus 1\oplus \Lam^{3,2}_{1,[3,5]}$ \\
\hline
\size SP & \size
\begin{math}
\begin{matrix}
x_{361},x_{461},x_{561},x_{262},x_{342},x_{352},x_{452}
\end{matrix}
\end{math} 
& \size $x_{362},x_{462},x_{562}$ \\
\hline \hline 
\size $76$ & \size $M_{[4],[1]}\cong \gl_4\times \gl_2\times \gl_1^2$
& \size 
$1\oplus \Lam^{4,2}_{1,[1,4]}$ \\
\hline
\size SP & \size
\begin{math}
\begin{matrix}
x_{561},x_{122},x_{132},x_{142},x_{232},x_{242},x_{342}
\end{matrix}
\end{math} 
& \size $x_{152},x_{162},x_{252},x_{262},x_{352},x_{362},x_{452},x_{462},x_{562}$ \\
\hline \hline 
\size $78$ & \size $M_{[2],[1]}\cong \gl_4 \times \gl_2\times \gl_1^2$
& \size 
$\Lam^{4,2}_{1,[3,6]}$ \\
\hline
\size SP & \size
\begin{math}
\begin{matrix}
x_{342},x_{352},x_{362},x_{452},x_{462},x_{562}
\end{matrix}
\end{math} 
& \size \hskip 0.8in -  \\
\hline \hline 
\size $80$ & \size $M_{[3,5],\emptyset}\cong \gl_3\times \gl_2^2\times \gl_1$
& \size 
$\Lam^{2,1}_{1,[4,5]}\otimes \Lam^{2,1}_{2,[1,2]}$ \\
\hline
\size SP & \size
\begin{math}
\begin{matrix}
x_{461},x_{561},x_{462},x_{562}
\end{matrix}
\end{math} 
& \size \hskip 0.8in -  \\
\hline \hline 
\size $81$ & \size $M_{[4],[1]}\cong \gl_4\times \gl_2\times \gl_1^2$
& \size  $1$ \\
\hline
\size SP & \size
\begin{math}
\begin{matrix}
x_{562}
\end{matrix}
\end{math} 
& \size \hskip0.8in - \\
\hline 
\end{tabular}

\end{center}

\vskip 10pt

We now verify that $S_{\be_i}\not=\emptyset$ and 
determine $G_k\backslash S_{\be_i\, k}$
for the above $i$'s.

(1) $\be_8=\tfrac {1} {2} (0,0,0,0,0,0,-1,1)$.  

We identify the element 
$(g_1,\diag(t_{21},t_{22}))\in M_{\emptyset,[1]}$
with $g=(g_1,t_{21},t_{22})\in \gl_6\times \gl_1^2$.
On $M^1_{\be_8}$, 
$\chi_{\be_8}(g)=t_{21}^{-1}t_{22}=t_{22}^2$.  

Let $A(x)\in \wedge^2 \aff^6$ be the 
element such that the $(i,j)$-entry  is $x_{ij2}$ for $i<j$.   
We identify $\Lam^{6,2}_{1,[1,6]}$ with 
$\wedge^2 \aff^6$ by the map $x\mapsto A(x)$. 
Then the action of $g$ as above on $\wedge^2 \aff^6$ 
is $\wedge^2 \aff^6\ni A(x)\mapsto t_{22}g_1A(x){}^tg_1\in \wedge^2 \aff^6$. 
Let $P(x)$ be the Pfaffian of $A(x)$, which is a homogeneous cubic
polynomial. Then on $G_{\text{st},\be_8}$, 
$P(gx)=t_{22}^3 (\det g_1)P(x)=t_{22}^3 P(x)$. 
So $P(x)$ is invariant under the action of $G_{\text{st},\be_8}$. 
Therefore, 
\begin{equation*}
Z^{\sst}_{\be_8\, k} = \{x\in Z_{\be_8\, k}\mid P(x)\not=0\}. 
\end{equation*}
Witt's theorem implies that 
\begin{math} 
Z^{\sst}_{\be_8\, k}
\end{math}
is a single $M_{\be_8, k}$-orbit. 
Since $W_{\be_8}=\{0\}$, 
\begin{math}
Y^{\sst}_{\be_8\,k}
\end{math}
is also a single $P_{\be_8\, k}$-orbit. 

(2) $\be_{13}=\tfrac {1} {66} (-4,-4,2,2,2,2,-3,3)$.  

Note that $M_{[2],[1]}\cong \gl_4\times \gl_2\times \gl_1^2$.
The element 
$(\diag(g_1,g_2),\diag(t_{21},t_{22}))\in M_{[2],[1]}$
is identified with $g=(g_2,g_1,t_{21},t_{22})\in \gl_4\times \gl_2\times \gl_1^2$.
On $M^1_{\be_{13}}$, 
\begin{equation*}
\chi_{\be_{13}}(g)=(\det g_1)^{-4}(\det g_2)^2t_{21}^{-3}t_{22}^3=((\det g_2)t_{22})^6.
\end{equation*}
%

Let $\{\coorde_1,\coorde_2,\coorde_3,\coorde_4\}$ be the
standard basis of $\aff^4$ and $e_{ij}=\coorde_i\wedge \coorde_j$.  
For $x\in Z_{\be_{13}}$, let 
\begin{align*}
A(x) & = x_{341} e_{12}
+ x_{351} e_{13}
+ x_{361} e_{14}
+ x_{451} e_{23}
+ x_{461} e_{24}
+ x_{561} e_{34} \\
v_1(x) & = \sum_{i=3}^6 x_{1i2} \coorde_{i-2}, \;
v_2(x) = \sum_{i=3}^6 x_{2i2} \coorde_{i-2}.
\end{align*}
Note that $3,4,5,6$ correspond to $1,2,3,4$ of 
$\{\coorde_1,\coorde_2,\coorde_3,\coorde_4\}$. 

Let $P_1(x)$ be the polynomial such that 
\begin{equation*}
A(x) \wedge v_1(x) \wedge v_2(x) = P_1(x) \, \coorde_1\wedge \cdots \wedge \coorde_4.
\end{equation*}
Since $g$ acts on $A(x)$ (resp. $(v_1(x)\; v_2(x))$) 
by the natural action of $g_2\in \gl_4$ 
and the scalar multiplication by $t_{21}$ 
(resp. by multiplication by $g_2$ from the left, 
${}^tg_1$ from the right 
and the scalar multiplication by $t_{22}$), 
$P_1(gx) = (\det g_1)(\det g_2) t_{21}t_{22}^2P_1(x)$. 
Let $P_2(x)$ be the polynomial such that 
\begin{equation*}
A(x) \wedge A(x) = P_2(x) \, \coorde_1\wedge \cdots \wedge \coorde_4.
\end{equation*}
$P_2(x)$ is the Pfaffian of $A(x)$ and 
$P_2(gx) = (\det g_2)t_{21}^2$.  

We put $P(x) = P_1(x)^3P_2(x)$. 
Then on $M^1_{\be_{13}}$,  $P(gx) = (\det g_2)t_{22}P(x)$. 
So, $P(x)$ is invariant under the action of $G_{\text{st},\be_{13}}$. 
Therefore, 
\begin{equation*}
Z^{\sst}_{\be_{13}\, k} = \{x\in Z_{\be_{13}\, k}\mid P_1(x),P_2(x)\not=0\}. 
\end{equation*}

Let $R(13)\in Z_{\be_{13}\, k}$ be the element such that 
$A(R(13))=e_{12}+e_{34}$, $v_1(R(13)) = \coorde_3$,
$v_2(R(13)) = \coorde_4$. 
\begin{prop}
$Y^{\sst}_{\be_{13}\, k} = P_{\be_{13}\, k}R(13)$. 
\end{prop}
\begin{proof}
We first prove that
$Z^{\sst}_{\be_{13}\, k} = M_{\be_{13}\, k}R(13)$. 
Let $x\in Z^{\sst}_{\be_{13}\, k}$. 
By Witt's theorem, we may assume that 
$A(x)=e_{12}+e_{34}$. 
Let 
\begin{equation*}
\tau_1= 
\begin{pmatrix}
0 & I_2 \\
I_2 & 0 
\end{pmatrix},\;
\tau_2=
\begin{pmatrix}
0 & 1 \\
-1 & 0 
\end{pmatrix},\;
\sig_1= (\diag(I_2,\tau_1),I_2),\;
\sig_2= (\diag(I_2,\tau_2,I_2),I_2). 
\end{equation*}
Then $\sig_1,\sig_2\in M_{[2],[1]}$ fix $A(x)$. 

By assumption, $v_1(x)\not=0$. By applying $\sig_1,\sig_2$
(multiple times) if necessary, 
we may assume that $x_{152}\not=0$. 
By applying an element of the form 
$(\diag(I_4,g),I_2)$ with $g\in \spl_2$, 
we may assume that $x_{152}=1,x_{162}=0$. 
For $u=(u_1,u_2,u_3)\in k^3$, let 
\begin{equation*}
\nu_1(u)  
=\begin{pmatrix}
1 & u_1 & 0 & u_2 \\
0 & 1 & 0 & 0 \\
0 & u_2 & 1 & u_3 \\
0 & 0 & 0 & 1 
\end{pmatrix},\;
\nu_2(u) 
= \begin{pmatrix}
1 & 0 & 0 & 0 \\
u_1 & 1 & u_2 & 0 \\
0 & 0 & 1 & 0 \\
u_2 & 0 & u_3 & 1 
\end{pmatrix}. 
\end{equation*}
Then $(\diag(I_2,\nu_1(u)),I_2),(\diag(I_2,\nu_2(u)),I_2)$ 
fix $A(x)$.  
By applying elements of the forms 
$(\diag(I_2,\nu_2(0,u_2,0)),I_2)$, 
$\sig_2(\diag(I_2,\nu_2(0,u_2,0)),I_2)\sig_2^{-1}$ 
if necessary, we may assume that $v_1(x)=\coorde_3$. 
 
By assumption, $x_{262}\not=0$. 
We may assume that $x_{262}=1$ 
by applying an element of the form 
$(\diag(I_2,t^{-1},1,t^{-1},1),tI_2)$. 
Elements of the forms 
$(\diag(I_2,\nu_1(u)),I_2)$, 
$\sig_2(\diag(I_2,\nu_1(u)),I_2)\sig_2^{-1}$ 
fix $A(x),\coorde_3$. 
Applying elements of these forms, 
we may assume that $x_{232},x_{242}=0$. 
By applying an element of the form 
$(\diag(I_4,{}^tn_2(u)),I_2)$, $v_2(x)$ becomes 
$\coorde_4$.

Suppose that $x\in Y^{\sst}_{\be_{13}\, k}$. By the above consideration,
we may assume that $x$ is is the form $(R(13),w)$ where 
$w=(x_{342},x_{352},x_{362},x_{452},x_{462},x_{562})$. 
For $u_1=(u_{1ij})_{1\leq j<i\leq 6}\in k^{15}$, $u_2 = u_{221}\in k$, 
let 
\begin{equation}
\label{eq:nu-62}
n(u)=n(u_1,u_2)=(n_6(u_1),n_2(u_2)).
\end{equation}
We assume that $u_{1ij}=0$ unless 
$(i,j)=(3,1),(3,2),(4,1),(4,2),(5,2)$. 
Then $n(u)\in U_{\be_{13}}$.
Let $n(u)(R(13),w)=(R(13),w')$. Then 
$w'=(x'_{342},x'_{352},x'_{362},x'_{452},x'_{462},x'_{562})$ where 
\begin{align*}
x'_{342} & = x_{342} + u_{221}, \quad
x'_{352} = x_{352} + u_{131}, \quad
x'_{362} = x_{362} + u_{132}, \\
x'_{452} & = x_{452} + u_{141}, \quad
 x'_{462} = x_{462} + u_{142}, \quad
x'_{562} = x_{562} + u_{152}+u_{221}.
\end{align*}
Therefore, there exists such $u$ such that $w'=(0,0,0,0,0,0)$. 
\end{proof}

(3) $\be_{18}=\tfrac {1} {12} (-1,-1,-1,-1,2,2,0,0)$. 


Note that $M_{[4],\emptyset}\cong \gl_4\times \gl_2^2$.
We identify the element 
$(\diag(g_{11},g_{12}),g_2)\in M_{[4],\emptyset}$
with $g=(g_{11},g_{12},g_2)\in \gl_4\times \gl_2^2$.
On $M^1_{\be_{18}}$, 
\begin{equation*}
\chi_{\be_{18}}(g)=(\det g_{11})^{-1}(\det g_{12})^2
= (\det g_{12})^3. 
\end{equation*}
%


Let 
\begin{equation*}
A(x) = \begin{pmatrix}
x_{151} & x_{152} & x_{161} & x_{162} \\
x_{251} & x_{252} & x_{261} & x_{262} \\ 
x_{351} & x_{352} & x_{361} & x_{362} \\ 
x_{451} & x_{452} & x_{461} & x_{462} 
\end{pmatrix}
\end{equation*}
We identify $Z_{\be_{18}}$ with $\m_4$ by the map $x\mapsto A(x)$. 
Let $P(x)=\det A(x)$. 
Then $P(gx)=(\det g_{11})(\det g_{12})^2(\det g_2)^2P(x)$. 
If $g\in M^1_{\be_{18}}$ then 
$P(gx)=\det g_{12}P(x)$. So $P(x)$ is invariant under the action 
of $G_{\text{st},\be_{18}}$. Therefore, 
\begin{equation*}
Z^{\sst}_{\be_{18}\, k} = \{x\in Z_{\be_{18}\, k}\mid P(x)\not=0\}. 
\end{equation*}

Let $R(18)\in Z^{\sst}_{\be_{18}\, k}$ be the element 
such that $A(R(18))=I_4$. 
Since $\{A\in \m_4(k)\mid \det A\not=0\}=\gl_4(k)I_4$,
$Z^{\sst}_{\be_{18}\, k}=M_{\be_{18}\, k}R(18)$. 

Suppose that $x\in Y^{\sst}_{\be_{18}\,k}$. 
By the above consideration, we may assume that 
$x=(R(18),w)$ where $w=(x_{561},x_{562})$.
Let $n(u)=(n_6(u_1),1)$ where $u_1=(u_{1ij})$
and $u_{1ij}=0$ unless $(i,j)=(5,3),(5,4)$. 
Then $n(u)\in U_{\be_{18}}$. 
Let $n(u)(R(18),w)=(R(18),w')$. Then 
$w'=(x'_{561},x'_{562})$ where
\begin{equation*}
x'_{561} = x_{561} + u_{153},\;
x'_{562} = x_{562} + u_{154}.
\end{equation*}
Therefore, there exists $u_{153},u_{154}\in k$ 
such that $w'=(0,0)$. 

(4) $\be_{35}=\tfrac {1} {6} (-2,0,0,0,1,1,0,0)$. 

Note that $M_{[1,4],\emptyset}\cong \gl_3\times \gl_2^2\times \gl_1$.
The element 
$(\diag(t_1,g_{11},g_{12}),g_2)\in M_{[1,4],\emptyset}$
is identified with $g=(g_{11},g_{12},g_2,t_1)\in \gl_3\times \gl_1^2\times \gl_1$.
The action of $M_{[1,4],\emptyset}$ on $Z_{\be_{35}}$ 
does not depend on $t_1$. 
On $M^1_{\be_{35}}$, 
$\chi_{\be_{35}}(g)=t_1^{-2}(\det g_{12})= (\det g_{11})^2(\det g_{12})^3$. 

We are in the situation of Lemma \ref{lem:Phi-equivariant} and 
so there is a map $\Phi:Z_{\be_{35}}\to \m_2$ 
such that 
\begin{equation*}
\Phi(gx)=(\det g_{11})(\det g_{12})(\det g_2)\theta(g_{12})\Phi(x)
{}^t\theta(g_2). 
\end{equation*}
Let $P(x)=\det \Phi(x)$. Then 
$P(gx)=(\det g_{11})^2(\det g_{12})^3(\det g_2)^3P(x)$. 
If $g\in M^1_{\be_{35}}$ then 
$P(gx)=(\det g_{11})^2(\det g_{12})^3P(x)$. 
So $P(x)$ is invariant under the action of 
$G_{\text{st},\be_{35}}$. 

Let $R(35)\in Z_{\be_{35}}$ be the element such that 
$R(35)=(x_{ijk})$ where 
\begin{equation*}
\begin{matrix}
x_{251} = -1, & x_{261} = 0, \\
x_{252} = 0, & x_{262} = 1,
\end{matrix}, \quad
\begin{matrix}
x_{351} = 0, & x_{361} = 1, \\
x_{352} = 0, & x_{362} = 0,
\end{matrix} \quad
\begin{matrix}
x_{451} = 0, & x_{461} = 0, \\
x_{452} = 1, & x_{462} = 0.
\end{matrix}
\end{equation*}
Then by (\ref{eq:RtoI}), $\Phi(R(35))=I_2$. 
As in the case (1) of Section \ref{sec:non-empty-strata1}, 
$Z^{\sst}_{\be_{35}\,k}=M_{\be_{35}\,k} R(35)$. 

\begin{prop}
$Y^{\sst}_{\be_{35}\, k} = P_{\be_{35}\, k}R(35)$. 
\end{prop}
\begin{proof}
Let $x\in Y^{\sst}_{\be_{35}\, k}$. By the above argument, 
we may assume that $x$ is in the form 
$x=(R(35),w)$ where $w=(x_{561},x_{562})$
Let $n(u)$ be the element (\ref{eq:nu-62})
where $u_{1ij}=0$ unless $(i,j)=(5,2),(5,3)$ and $u_{221}=0$. 
Let $n(u)x = (R(35),w')$. Then 
$w'=(x_{561}',x_{562}')$ where 
\begin{equation*}
x_{561}' = x_{561} + u_{153}, \quad
x_{562}' = x_{562} + u_{152}.
\end{equation*}
Therefore, there exist such $u$ such that $w'=(0,0)$.
\end{proof}

(5) $\be_{46}=\tfrac {1} {18} (-2,-2,0,0,2,2,-1,1)$.  

We identify the element 
$(\diag(g_{11},g_{12},g_{13}),\diag(t_{21},t_{22}))\in M_{[2,4],\emptyset}$
with 
\begin{equation*}
g=(g_{11},g_{12},g_{13},t_{21},t_{22})\in \gl_2^3\times \gl_1^2. 
\end{equation*}
On $M^1_{\be_{46}}$, 
\begin{equation*}
\chi_{\be_{46}}(g)=(\det g_{11})^{-2}(\det g_{13})^2t_{21}^{-1}t_{22}
= (\det g_{12})^2(\det g_{13})^4t_{22}^2.
\end{equation*}
%

For $x\in Z_{\be_{46}}$, let 
\begin{equation*}
A(x) = \begin{pmatrix}
x_{351} & x_{361} \\
x_{451} & x_{461}
\end{pmatrix}, \quad 
B(x) = \begin{pmatrix}
x_{152} & x_{162} \\
x_{252} & x_{262}
\end{pmatrix}.  
\end{equation*}
The map 
\begin{math}
Z_{\be_{46}}\ni x \mapsto (A(x),B(x),x_{342})
\in \m_2\oplus \m_2\oplus \aff^1
\end{math}
is an isomorphism. It is easy to see that 
$A(gx) = t_{21}g_{12}A(x){}^tg_{13}$, 
$B(gx) = t_{22}g_{11}B(x){}^tg_{13}$. 
Let 
\begin{equation*}
P_1(x)=\det A(x),\; P_2(x)=\det B(x),\;
P(x) = P_1(x)^2P_2(x)^2x_{342}. 
\end{equation*}
Then $P_1(gx)=t_{21}^2(\det g_{12})(\det g_{13})P_1(x)$, 
$P_2(gx)=t_{22}^2(\det g_{11})(\det g_{13})P_2(x)$. 
So for $g\in G_{\text{st},\be_{46}}$,  
$P(gx) = (\det g_{12})(\det g_{13})^2t_{22}P(x)$. 
Therefore, $P(x)$ is invariant under the action of $G_{\text{st},\be_{46}}$. 

Let $R(46)\in Z_{\be_{46}\, k}$ be the element such that 
$A(R(46))=B(R(46))=I_2$ and $x_{342}=1$. 
Then $P(R(46))=1$. Since $\gl_2^3$ is acting, 
it is easy to see that $Z_{\be_{46}\, k}^{\sst}=M_{\be_{46}\, k}R(46)$. 

\begin{prop}
$Y^{\sst}_{\be_{46}\, k} = P_{\be_{46}\, k}R(46)$. 
\end{prop}
\begin{proof}
Let $x\in Y^{\sst}_{\be_{46}\, k}$. By the above argument, 
we may assume that $x$ is in the form 
$x=(R(46),w)$ where $w=(x_{561},x_{352},x_{362},x_{452},x_{462},x_{562})$. 
Let $n(u)$ be the element (\ref{eq:nu-62}). 
We assume that $u_{1ij}=0$ unless 
$(i,j)=(3,2),(4,1),(4,2),(6,1),(6,3)$. 
Then $n(u)\in U_{\be_{46}\,k}$. 
Let $n(u)x = (R(46),w')$. Then 
$w'=(x'_{561},x'_{352},x'_{362},x'_{452},x'_{462},x'_{562})$ 
where 
\begin{align*}
x'_{561} & = x_{561} - u_{163}, \;
x'_{352} = x_{351} + u_{221}, \;
x'_{362} = x_{362} + u_{132}, \\
x'_{452} & = x_{452} + u_{141}, \;
x'_{462} = x_{462} + u_{221} + u_{142} - u_{163}, \;
x'_{562} = x_{562} - u_{163}u_{221} - u_{161}. 
\end{align*}
Then it is easy to see that there exists $u$ such that
$w'=(0,0,0,0,0,0)$. 
\end{proof}

(6) $\be_{66}=\tfrac {1} {6} (-1,-1,0,0,1,1,-1,1)$.  

Note that $M_{[2,4],[1]}\cong \gl_2^3\times \gl_1^2$.
The element $(\diag(g_{11},g_{12},g_{13}),\diag(t_{21},t_{22}))\in M_{[2,4],[1]}$
is identified with 
$g=(g_{11},g_{12},g_{13},t_{21},t_{22})\in \gl_2^3\times \gl_1^2$.
On $M^1_{\be_{66}}$, 
\begin{equation*}
\chi_{\be_{66}}(g)=(\det g_{11})^{-1}(\det g_{13})t_{21}^{-1}t_{22}
= (\det g_{12})(\det g_{13})^2t_{22}^2.
\end{equation*}
%


Let 
\begin{equation*}
A(x) = \begin{pmatrix}
x_{152} & x_{162} \\
x_{252} & x_{262}
\end{pmatrix}, \;
P(x)=(\det A(x))x_{561}^2x_{342}^2.
\end{equation*}
%
Since $A(gx)= t_{22} g_{11} A(x) {}^t g_{13}$, 
$P(gx) = t_{21}^2t_{22}^4(\det g_{11})(\det g_{12})^2(\det g_{13})^3$. 
If $g\in M^1_{\be_{66}}$ then 
$P(gx) = t_{22}^2(\det g_{12})(\det g_{13})^2$. 
So $P(x)$ is invariant under the action of $G_{\text{st},\be_{66}}$. 

Let $R(66)\in Z_{\be_{66}\, k}$ be the element such that 
$A(R(66))=I_2$ and $x_{561},x_{342}=1$. 
Then $P(R(66))=1$. It is easy to see that 
$Z_{\be_{66}\, k}^{\sst}=M_{\be_{66}\, k}R(66)$. 

\begin{prop}
$Y^{\sst}_{\be_{66}\, k} = P_{\be_{66}\, k}R(66)$. 
\end{prop}
\begin{proof}
Let $x\in Y^{\sst}_{\be_{66}\, k}$. By the above argument, 
we may assume that $x$ is in the form 
$x=(R(66),w)$ where $w=(x_{352},x_{362},x_{452},x_{462},x_{562})$. 

Let $n(u)$ be the element (\ref{eq:nu-62}). 
We assume that $u_{1ij}=0$ unless
$(i,j)=(3,1),(3,2)$, $(4,1),(4,2),(5,2)$ 
and that $u_{221}=0$. 
Then $n(u)\in U_{\be_{66}\,k}$. 
Let $n(u)x = (R(66),w')$. Then 
$w'=(x'_{352},x'_{362},x'_{452},x'_{462},x'_{562})$ 
where 
\begin{align*}
x'_{352} & = x_{352} + u_{131}, \;
x'_{362} = x_{362} + u_{132}, \;
x'_{452} = x_{452} + u_{141}, \\
x'_{462} & = x_{462} + u_{142}, \;
x'_{562} = x_{562} + u_{152}. 
\end{align*}
It is easy to see that there exists $u$ such that
$w'=(0,0,0,0,0)$. 
\end{proof}

(7) $\be_{67}=\tfrac {1} {6} (-2,-2,1,1,1,1,0,0)$. 

Note that $M_{[2],\emptyset}\cong \gl_4\times \gl_2^2$.
We identify the element 
$(\diag(g_{11},g_{12}),g_2)\in M_{[2],\emptyset}$
with $g=(g_{12},g_{11},g_2)\in \gl_4\times \gl_2^2$.
On $M^1_{\be_{67}}$, 
\begin{equation*}
\chi_{\be_{67}}(g)=(\det g_{11})^{-2}(\det g_{12})
= (\det g_{12})^3. 
\end{equation*}
%


The space $Z_{\be_{67}}$ can be identified with 
$\wedge^2 \aff^4\otimes \aff^2$ where $g\in M_{\be_{67}}$
acts by 
\begin{equation*}
\wedge^2 \aff^4\otimes \aff^2 \ni A\otimes v
\mapsto (g_{12} A {}^t g_{12})\otimes g_2v 
\in \wedge^2 \aff^4\otimes \aff^2. 
\end{equation*}

This \rep{} was considered in \S 4 \cite{wryu} and 
the set of rational orbits is in bijective correspondence
with $\Ex_2(k)$. Note that there was an assumption on 
$\ch(k)$ in \cite{wryu}, but it is not necessary. 
The point is that one can verify 
that the \rep{} is regular in the sense of Definition 2.1 
\cite[p.310]{kato-yukie-jordan} by simple Lie algebra computations. 
There is a polynomial $P(x)$ of degree $4$ such that 
$P(gx)=(\det g_{12})^2(\det g_2)^2 P(x)$. 
If $g\in M^1_{\be_{67}}$ then $P(gx)=(\det g_{12})^2 P(x)$. 
So $P(x)$ is invariant under the action of $G_{\text{st},\be_{67}}$. 
As we pointed above, 
$M_{\be_{67}\,k}\backslash Z^{\sst}_{\be_{67}\,k}$ 
is in bijective correspondence with $\Ex_2(k)$. 
Since $W_{\be_{67}}=\{0\}$, 
$P_{\be_{67}\,k}\backslash Y^{\sst}_{\be_{67}\,k}$ 
is also in bijective correspondence with $\Ex_2(k)$.

(8) $\be_{74}=\tfrac {1} {6} (-2,-2,0,0,2,2,-1,1)$.  

Note that $M_{[2,4],[1]}\cong \gl_2^3\times \gl_1^2$.
The element 
$(\diag(g_{11},g_{12},g_{13}),t_{21},t_{22})\in M_{[2,4],[1]}$
is identified with 
$g=(g_{11},g_{12},g_{13},t_{21},t_{22})\in \gl_2^3\times \gl_1^2$.
On $M^1_{\be_{74}}$, 
\begin{equation*}
\chi_{\be_{74}}(g)=(\det g_{11})^{-2}(\det g_{13})^2 t_{21}^{-1}t_{22}
= (\det g_{12})^2(\det g_{13})^4 t_{22}^2.
\end{equation*}
%


Let 
\begin{equation*}
A(x) = 
\begin{pmatrix}
x_{352} & x_{362} \\
x_{452} & x_{462}
\end{pmatrix},\;
P(x)=x_{561}\det A(x). 
\end{equation*}
%
Since $A(gx)=t_{22}g_{12}A(x){}^t g_{13}$, 
$P(x) = t_{21}t_{22}^2(\det g_{12})(\det g_{13})^2$. 
If $g\in M^1_{\be_{74}}$ then 
$P(x) = t_{22}(\det g_{12})(\det g_{13})^2$. 
So $P(x)$ is invariant under the action of $G_{\text{st},\be_{74}}$. 

Let $R(74)\in Z^{\sst}_{\be_{74}\,k}$ be the element such that 
$A(R(74))=I_2$ and $x_{561}=1$. It is easy to see that 
$Z_{\be_{74}\, k}^{\sst}=M_{\be_{74}\, k}R(74)$. 

\begin{prop}
$Y^{\sst}_{\be_{74}\, k} = P_{\be_{74}\, k}R(74)$. 
\end{prop}
\begin{proof}
Let $x\in Y^{\sst}_{\be_{74}\, k}$. By the above argument, 
we may assume that $x$ is in the form 
$x=(R(74),x_{562})$. Then $(I_6,n_2(-x_{562})) x = (R(74),0)$. 
\end{proof}

(9)  $\be_{75}=\tfrac {1} {30} (-10,-4,2,2,2,8,-3,3)$.  

Note that $M_{[1,2,5],[1]}\cong \gl_3\times \gl_1^5$.
The element 
$(\diag(t_{11},t_{12},g_1,t_{13}),\diag(t_{21},t_{22}))$ $\in M_{[1,2,5],[1]}$
is identified with 
$g=(g_1,t_{11},t_{12},t_{13},t_{21},t_{22})\in \gl_3\times \gl_1^5$.
On $M^1_{\be_{75}}$, 
\begin{equation*}
\chi_{\be_{75}}(g)=t_{11}^{-10}t_{12}^{-4}(\det g_1)^2t_{13}^8 t_{21}^{-3}t_{22}^3
= t_{12}^6 t_{13}^{18}t_{22}^6(\det g_1)^{12}. 
\end{equation*}
%


Let $\{\coorde_1,\coorde_2,\coorde_3\}$ be the standard basis 
of $\aff^3$ and 
\begin{equation*}
A(x) = x_{361}\,\coorde_1+ x_{461}\,\coorde_2+x_{561}\,\coorde_3,\;
B(x) = x_{342}\,\coorde_1\wedge \coorde_2+ x_{352}\,\coorde_1\wedge \coorde_3
+ x_{452}\, \coorde_2\wedge \coorde_3.
\end{equation*}
Let $P_1(x)$ be the polynomial such that 
$A(x)\wedge B(x)=P_1(x)\, \coorde_1\wedge \coorde_2\wedge \coorde_3$. 
Then 
\begin{equation*}
P_1(gx) = t_{13} t_{21}t_{22} \det g_1 P_1(x). 
\end{equation*}
We put $P(x) = P_1(x)^2x_{262}$. 
Then $P(gx) = t_{12}t_{13}^3 t_{21}^2t_{22}^3 (\det g_1)^2 P(x)$. 
If $g\in M^1_{\be_{75}}$ then 
$P(gx) = t_{12}t_{13}^3 t_{22} (\det g_1)^2 P(x)$. 
So $P(x)$ is invariant under the action of $G_{\text{st},\be_{75}}$. 

Let $R(75)\in Z^{\sst}_{\be_{75}\,k}$ be the element such that 
$A(R(75))= \coorde_1,\; B(R(75))=\coorde_2\wedge \coorde_3$ 
and $x_{262}=1$. 

\begin{prop}
$Y^{\sst}_{\be_{75}\, k} = P_{\be_{75}\, k}R(75)$. 
\end{prop}
\begin{proof}
Let $x\in Z^{\sst}_{\be_{75}\, k}$. We may assume that 
$A(x)=\coorde_1$. Since $A(x)\wedge B(x)\not=0$, 
$x_{452}\not=0$. Applying an element of the form 
$(\diag(1,1,1,t,1,1),I_2)$ ($t\in\mk$), we may 
assume that $x_{452}=1$. 
Let $u=(u_{ij})_{1\leq j<i\leq 6}$
where $u_{ij}=0$ unless $(i,j)=(4,3),(5,3)$ 
and $u_{43}=-x_{352},u_{53}=x_{342}$. Then 
$({}^tn_6(u),I_2)\in M_{\be_{75}\, k}$ and 
$({}^tn_6(u),I_2)x=R(75)$. 
This implies that $Z^{\sst}_{\be_{75}\, k} = M_{\be_{75}\, k}R(75)$. 

So we may assume that $x$ is in the form $(R(75),w)$ 
where $w=(x_{362},x_{462},x_{562})$. 
Let $n(u)$ be the element (\ref{eq:nu-62})
where $u_{1ij}=0$ unless $(i,j)=(6,4),(6,5)$. 
Then $n(u)\in U_{\be_{75}\,k}$. 
Let $n(u)x = (R(75),w')$. Then 
$w'=(x_{362}',x_{462}',x_{562}')$ where 
\begin{equation*}
x_{362}' = x_{362} + u_{221}, \quad
x_{462}' = x_{462} + u_{165}, \quad
x_{562}' = x_{562} - u_{164}. 
\end{equation*}
Therefore, there exists $u$ such that 
$w'=(0,0,0)$. 
\end{proof}

(10) $\be_{76}=\tfrac {1} {30} (-1,-1,-1,-1,2,2,-3,3)$. 

Note that $M_{[4],[1]}\cong \gl_4\times \gl_2\times \gl_1^2$.
The element 
$(\diag(g_{11},g_{12}),\diag(t_{21},t_{22}))\in M_{[4],[1]}$
is identified with 
$g=(g_{11},g_{12},t_{21},t_{22})\in \gl_4\times \gl_2\times \gl_1^2$.
On $M^1_{\be_{76}}$, 
\begin{equation*}
\chi_{\be_{76}}(g)=(\det g_{11})^{-1}(\det g_{12})^2t_{21}^{-3}t_{22}^3
= t_{22}^6 (\det g_{12})^3. 
\end{equation*}
%


Let 
\begin{equation*}
A(x) = 
\begin{pmatrix}
0 & x_{122} & x_{132} & x_{142} \\
-x_{122} & 0 & x_{232} & x_{242} \\
-x_{132} & -x_{232} & 0 & x_{342} \\
-x_{142} & -x_{242} & -x_{342} & 0 
\end{pmatrix}.
\end{equation*}
Then $A(gx) = t_{22} g_{11}A(x){}^t g_{11}$.
We identify $Z_{\be_{76}}$ with 
$\wedge^2 \aff^4\oplus \aff^1$ by the map 
$x\mapsto (A(x),x_{561})$. 
Let $P_1(x)$ be the Pfaffian of $A(x)$.  
Then $P_1(gx)= t_{22}^2 (\det g_{11}) P_1(x)$. 
We put $P(x)=x_{561}^4 P_1(x)^3$. Then 
$P(gx) = t_{21}^4 t_{22}^6 (\det g_{11})^3 (\det g_{12})^4P(x)$. 
If $g\in M^1_{\be_{76}}$ then
$P(gx) = t_{22}^2 (\det g_{12})P(x)$.  
So $P(x)$ is invariant under the action of $G_{\text{st},\be_{76}}$. 

Let $R(76)\in Z^{\sst}_{\be_{76}\,k}$ be the element such that 
$x_{561},x_{122},x_{342}=1$ and other coordinates are $0$. 
(Note that $P(R(76))=1$). By Witt's theorem, 
$Z^{\sst}_{\be_{76}\,k}=M_{\be_{76}\,k}R(76)$. 

\begin{prop}
$Y^{\sst}_{\be_{76}\, k} = P_{\be_{76}\, k}R(76)$. 
\end{prop}
\begin{proof}
Suppose that $x\in Y^{\sst}_{\be_{76}\,}$.  
We may assume that $x$ is in the form 
$x=(R(76),w)$ where 
$w=(x_{152},x_{162},x_{252},x_{262},x_{352},x_{362},x_{452},x_{462},x_{562})$. 

Let $n(u)$ be the element (\ref{eq:nu-62})
where $u_{1ij}=0$ unless $i=5,6, j=1\ccd 4$. 
Then $n(u)\in U_{\be_{76}\,k}$. 
Let $n(u)x = (R(76),w')$. Then 
\begin{equation*}
w'=(x'_{152},x'_{162},x'_{252},x'_{262},x'_{352},x'_{362},x'_{452},x'_{462},x'_{562})
\end{equation*}
where 
\begin{align*}
x'_{152} & = x_{152} + u_{152}, \;  
x'_{162} = x_{162} + u_{162}, \;
x'_{252} = x_{252} - u_{151}, \\
x'_{262} & = x_{262} - u_{161}, \;
x'_{352} = x_{352} + u_{154}, \;
x'_{362} = x_{362} + u_{164}, \\
x'_{452} & = x_{452} -u_{153}, \;
x'_{462} = x_{462} -u_{163}, \\
x'_{562} & = x_{562} + u_{221} + u_{151}u_{162}-u_{161}u_{152}
+ u_{153}u_{164}-u_{163}u_{154}.
\end{align*}
By $u_{151}\ccd u_{154},u_{161}\ccd u_{164}$, 
we can make $x'_{152}\ccd x'_{462}=0$. 
Then assuming $u_{151}\ccd u_{154},u_{161}\ccd u_{164}=0$, 
we make $x'_{562}=0$ using $u_{221}$. 
\end{proof}

(11) $\be_{78}=\tfrac {1} {6} (-2,-2,1,1,1,1,-3,3)$.  

Note that $M_{[2],[1]}\cong \gl_4\times \gl_2\times \gl_1^2$.
The element 
$(\diag(g_{11},g_{12}),\diag(t_{21},t_{22}))\in M_{[2],[1]}$
is identified with 
$g=(g_{12},g_{11},t_{21},t_{22})\in \gl_4\times \gl_2\times \gl_1^2$.
On $M^1_{\be_{78}}$, 
\begin{equation*}
\chi_{\be_{78}}(g)=(\det g_{11})^{-2}(\det g_{12})t_{21}^{-3}t_{22}^3
= t_{22}^6 (\det g_{12})^3. 
\end{equation*}
%


We identify $Z_{\be_{78}}$ with
$\wedge^2 \aff^4$. Let $P(x)$ be the
Pfaffian of $x$ as an element of $\wedge^2 \aff^4$. 
Then $P(gx) = t_{22}^2(\det g_{12})P(x)$. 
So $P(x)$ is invariant under the action of $G_{\text{st},\be_{78}}$. 
Let $R(78)\in Z^{\sst}_{\be_{78}\,k}$ be the 
element such that $x_{342},x_{562}=1$ and other coordinates are $0$.
By Witt's theorem, $Z^{\sst}_{\be_{78}\,k}=M_{\be_{78}\,k}R(78)$. 
Since $W_{\be_{78}}=\{0\}$, 
$Y^{\sst}_{\be_{78}\,k}=P_{\be_{78}\,k}R(78)$ also.

(12) $\be_{80}=\tfrac {1} {6} (-2,-2,-2,1,1,4,0,0)$.  

Note that $M_{[3,5],\emptyset}\cong \gl_3\times \gl_2^2\times \gl_1$.
The element 
$(\diag(g_{11},g_{12},t_1),g_2)\in M_{[3,5],\emptyset}$
is identified with 
$g=(g_{11},g_{12},g_2,t_1)\in \gl_3\times \gl_2^2\times \gl_1$.
On $M^1_{\be_{80}}$, 
\begin{equation*}
\chi_{\be_{80}}(g)=(\det g_{11})^{-2}(\det g_{12})t_1^4
= t_1^6 (\det g_{12})^3. 
\end{equation*}
%


Let 
\begin{equation*}
A(x) = 
\begin{pmatrix}
x_{461} & x_{462} \\
x_{561} & x_{562}
\end{pmatrix}. 
\end{equation*}
We identify $Z_{\be_{80}}$ with $\m_2$ 
by the map $x\mapsto A(x)$. 
Then $A(gx) = t_1g_{12}A(x){}^t g_2$. 
Let $P(x)=\det A(x)$.
Then on $G_{\text{st},\be_{80}}$, $P(gx)=t_1^2(\det g_{12})P(x)$. 
So $P(x)$ is invariant under the action of $G_{\text{st},\be_{80}}$.  
Let $R(80)\in Z^{\sst}_{\be_{80}\,k}$ be the element 
such that $x_{461},x_{562}=1,x_{561},x_{462}=0$ 
(note that $P(R(80))=1$). 
It is easy to see that $Z^{\sst}_{\be_{80}\,k} = M_{\be_{80}\,k}R(80)$. 
Since $W_{\be_{80}}=\{0\}$, 
$Y^{\sst}_{\be_{80}\,k} = P_{\be_{80}\,k}R(80)$ also.

(13) $\be_{81}=\tfrac {1} {6} (-2,-2,-2,-2,4,4,-3,3)$

This case is obvious.

\begin{thm}
\label{thm:main2-detail}
\begin{itemize}
\item[(1)]
For the \pv{} (2), $S_{\be_i}\not=\empty$ 
if and only if $i$ is one of the numbers in (\ref{eq:list-non-empty-62}).
\item[(2)]
For $i$ in (\ref{eq:list-non-empty-62}), 
$S_{\be_i\,k}$ is a single $P_{\be_i\,k}$-orbit 
except for $i=67$.  
\item[(3)]
$G_k\backslash S_{\be_{67}\,k}$ 
is in bijective correspondence with $\Ex_2(k)$. 
\end{itemize}
\end{thm}

Considerations of this section proves the 
above theorem except for 
the ``only if'' part of (2). 
We shall prove the ``only if'' part of (2) 
in the next section.

\section{Empty strata for the case (2)}
\label{sec:empty-strata2}

\newcommand{\tallstrut}{\rule[0pt]{0cm}{0pt}}
\newcommand{\tinsert}{\quad\tallstrut}

In this section we prove that $S_{\be_i}=\emptyset$ 
for $i$ not in (\ref{eq:list-non-empty-62}) for 
the \pv{} (2). We proceed as in Section \ref{sec:empty-strata1}.

In the following table, for each $i$, we list 
which coordinates of $x$ we can eliminate and 
the 1PS with the property that weights of non-zero
coordinates are all positive. 

\vskip 10pt

\begin{center}

\begin{tabular}{|l|l|l|}
\hline
\size $i$ & $\;$ \hskip 0.8in \size $M_{\be_i}$, $Z_{\be_i}$ \hskip 0.8in $\;$ 
& $\;$ \hskip 1.2in \size 1PS \hskip 1.2in $\;$ \\
\cline{2-3}
& \hskip 0.7in \size zero coordinates 
& \hskip 0.5in \size non-zero coordinates and their weights \\
\hline \hline
\size  $1$ & \size $M_{[5],[1]}$, \;
\begin{math}
\Lam^{5,1}_{1,[1,5]}\oplus \Lam^{5,2}_{1,[1,5]}
\end{math}
& \size $[-2,-2,-2,-2,13,-5,-5,5]$  \\
\cline{2-3}
& \size $x_{161},x_{261},x_{361},x_{461}=0$ 
& \size 
\begin{math}
\begin{matrix}
x_{561},x_{122},x_{132},x_{142},x_{152},x_{232}, \\
x_{242},x_{252},x_{342},x_{352},x_{452}, \hfill 
\end{matrix}
\end{math}
\;  
\begin{math}
\begin{matrix}
[3,1,1,1,16,1, \hfill \\
1,16,1,16,16] \hfill
\end{matrix}
\end{math}
\\
\hline \hline 
\size  $2$ & \size $M_{[3],[1]}$, \; 
\begin{math}
\Lam^{3,2}_{1,[4,6]}\oplus \Lam^{3,1}_{1,[1,3]}\otimes \Lam^{3,1}_{1,[4,6]}
\end{math}
& \size $[-13,5,11,-7,-7,11,-3,3]$ \\
\cline{2-3}
& \size $x_{451},x_{461},x_{142},x_{152}=0$ 
& \size 
\begin{math}
\begin{matrix}
x_{561},x_{162},x_{242},x_{252}, \\
x_{262},x_{342},x_{352},x_{362}, \hfill 
\end{matrix}
\end{math}
\;  $[1,1,1,1,19,7,7,25]$ \\
\hline \hline 
\size  $3$ & \size $M_{[1],\emptyset}$, \; 
$\Lam^{5,2}_{1,[2,6]}\otimes \Lam^{2,1}_{2,[1,2]}$
& \size $[0,-12,-2,8,-2,8,-5,5]$ \\ 
\cline{2-3}
& \size 
\begin{math}
\begin{matrix}
x_{231},x_{241},x_{251},x_{261},x_{351}, \hfill \\
x_{361},x_{451},x_{461},x_{232},x_{252}=0
\end{matrix}
\end{math}
& \size 
\begin{math}
\begin{matrix}
x_{341},x_{561},x_{242},x_{262},x_{342}, \\
x_{352},x_{362},x_{452},x_{462},x_{562}, \hfill 
\end{matrix}
\end{math} 
\;  $[1,1,1,1,11,1,11,11,21,11]$ \\
\hline \hline 
\size  $4$ & \size $M_{[1,4],[1]}$, \; 
$\Lam^{3,1}_{1,[2,4]}\otimes \Lam^{2,1}_{1,[5,6]}
\oplus \Lam^{2,1}_{1,[5,6]}\oplus \Lam^{3,2}_{1,[2,4]}$
& \size $[0,-14,10,10,-3,-3,-6,6]$ \\
\cline{2-3}
& \size 
\begin{math}
\begin{matrix}
x_{251},x_{261},x_{351},x_{152}=0
\end{matrix}
\end{math}
& \size 
\begin{math}
\begin{matrix}
x_{361},x_{451},x_{461},x_{162},x_{232},x_{242},x_{342}, 
\end{matrix}
\end{math} 
\;  $[1,1,1,3,2,2,26]$ \\
\hline \hline 
\size  $5$ & \size $M_{[3,5],[1]}$, \; 
$\Lam^{2,1}_{1,[4,5]}\oplus \Lam^{3,1}_{1,[1,3]}\oplus 1$
& \size $[-1,-1,2,-3,5,-2,-2,2]$ \\
\cline{2-3}
& \size 
\begin{math}
\begin{matrix}
x_{461},x_{162},x_{262}=0
\end{matrix}
\end{math}
& \size 
\begin{math}
\begin{matrix}
x_{561},x_{362},x_{452},
\end{matrix}
\end{math} 
\;  $[1,2,4]$ \\
\hline \hline 
\size  $6$ & \size $M_{[3,5],[1]}$, \;
$\Lam^{3,1}_{1,[1,3]}\oplus 1\oplus \Lam^{3,1}_{1,[1,3]}\otimes \Lam^{2,1}_{1,[4,5]}$
& \size $[-7,-8,12,5,5,-7,-4,4]$ \\
\cline{2-3}
& \size 
\begin{math}
\begin{matrix}
x_{161},x_{261}=0
\end{matrix}
\end{math}
& \size 
\begin{math}
\begin{matrix}
x_{361},x_{451},x_{142},x_{152},x_{242},x_{252},x_{342},x_{352}, 
\end{matrix}
\end{math} 
\;  $[1,6,2,2,1,1,21,21]$ \\
\hline
\end{tabular}

\end{center}

\begin{center}

\begin{tabular}{|l|l|l|}
\hline
\size $i$ & $\;$ \hskip 0.8in \size $M_{\be_i}$, $Z_{\be_i}$ \hskip 0.8in $\;$ 
& $\;$ \hskip 1.2in \size 1PS \hskip 1.2in $\;$ \\
\cline{2-3}
& \hskip 0.7in \size zero coordinates 
& \hskip 0.5in \size non-zero coordinates and their weights \\
\hline \hline 
\size  $7$ & \size $M_{[5],\emptyset}$, \;
$\Lam^{5,1}_{1,[1,5]}\otimes \Lam^{2,1}_{2,[1,2]}$
& \size $[-1,-1,0,1,1,0,0,0]$ \\
\cline{2-3}
& \size 
\begin{math}
\begin{matrix}
x_{161},x_{261},x_{361},x_{162},x_{262},x_{362}=0,
\end{matrix}
\end{math}
& \size 
\begin{math}
\begin{matrix}
x_{461},x_{561},x_{462},x_{562}, 
\end{matrix}
\end{math} 
\;  $[1,1,1,1]$ \\
\hline \hline 
\size  $9$ & \size $M_{[1,4],[1]}$, \;
$1\oplus \Lam^{2,1}_{1,[5,6]}\oplus \Lam^{3,2}_{1,[2,4]}$
& \size $[-2,-3,0,3,-3,5,2,-2]$ \\
\cline{2-3}
& \size 
\begin{math}
\begin{matrix}
x_{152},x_{232},x_{242}=0
\end{matrix}
\end{math}
& \size 
\begin{math}
\begin{matrix}
x_{561},x_{162},x_{342},
\end{matrix}
\end{math} 
\;  $[4,1,1]$ \\
\hline \hline 
\size  $10$ & \size $M_{[1,5],[1]}$, 
$\Lam^{4,1}_{1,[2,5]}\oplus 1\oplus \Lam^{4,2}_{1,[2,5]}$ 
& \size $[1,-5,-5,5,5,-1,-2,2]$ \\
\cline{2-3}
& \size 
\begin{math}
\begin{matrix}
x_{261},x_{361},x_{461},x_{232},x_{242}=0
\end{matrix}
\end{math}
& \size 
\begin{math}
\begin{matrix}
x_{561},x_{162},x_{252},x_{342},x_{352},x_{452},  
\end{matrix}
\end{math} 
\;  $[2,2,2,2,2,12]$ \\
\hline \hline 
\size  $11$ & \size $M_{[1,2,5],[1]}$, \;
$1\oplus \Lam^{3,2}_{1,[3,5]}\oplus 1\oplus \Lam^{3,1}_{1,[3,5]}$ 
& \size $[-3,10,-14,-14,22,-1,-5,5]$ \\
\cline{2-3}
& \size 
\begin{math}
\begin{matrix}
x_{341},x_{351}=0
\end{matrix}
\end{math}
& \size 
\begin{math}
\begin{matrix}
x_{261},x_{451},x_{162},x_{232},x_{242},x_{252}, 
\end{matrix}
\end{math} 
\;  $[4,3,1,1,1,37]$ \\
\hline \hline 
\size  $12$ & \size $M_{[2,5],\emptyset}$, \;
$\Lam^{2,1}_{1,[1,2]}\otimes \Lam^{2,1}_{2,[1,2]}
\oplus \Lam^{3,2}_{1,[3,5]}\otimes \Lam^{2,1}_{2,[1,2]}$
& \size $[1,1,-5,-5,6,2,0,0]$ \\
\cline{2-3}
& \size 
\begin{math}
\begin{matrix}
x_{341},x_{342}=0
\end{matrix}
\end{math}
& \size 
\begin{math}
\begin{matrix}
x_{161},x_{261},x_{351},x_{451},x_{162},x_{262},x_{352},x_{452},
\end{matrix}
\end{math} 
\;  $[3,3,1,1,3,3,1,1]$ \\
\hline \hline 
\size  $14$ & \size $M_{[1,3,5],[1]}$, \;
$\Lam^{2,1}_{1,[2,3]}\oplus 1^{2\oplus}
\oplus \Lam^{2,1}_{1,[2,3]}\otimes \Lam^{2,1}_{1,[4,5]}$
& \size $[1,-5,3,-3,5,-1,-1,1]$ \\
\cline{2-3}
& \size 
\begin{math}
\begin{matrix}
x_{261},x_{242}=0
\end{matrix}
\end{math}
& \size 
\begin{math}
\begin{matrix}
x_{361},x_{451},x_{162},x_{252},x_{342},x_{352},
\end{matrix}
\end{math} 
\;  $[1,1,1,1,1,9]$ \\
\hline \hline 
\size  $15$ & \size $M_{[2,3,5],[1]}$, \;
$1^{2\oplus}\oplus \Lam^{2,1}_{1,[1,2]}\oplus \Lam^{2,1}_{1,[4,5]}$ 
& \size $[-2,2,-1,-4,4,1,2,-2]$ \\
\cline{2-3}
& \size 
\begin{math}
\begin{matrix}
x_{162},x_{342}=0
\end{matrix}
\end{math}
& \size 
\begin{math}
\begin{matrix}
x_{361},x_{451},x_{262},x_{352},
\end{matrix}
\end{math} 
\;  $[2,2,1,1]$ \\
\hline \hline 
\size  $16$ & \size $M_{[2,4,5],[1]}$, \;
$\Lam^{2,1}_{1,[1,2]}\oplus \Lam^{2,1}_{1,[3,4]}\oplus \Lam^{2,1}_{1,[1,2]}\oplus 1$ 
& \size $[5,5,-25,19,-11,7,-7,7]$ \\
\cline{2-3}
& \size 
\begin{math}
\begin{matrix}
x_{161},x_{351}=0
\end{matrix}
\end{math}
& \size 
\begin{math}
\begin{matrix}
x_{261},x_{451},x_{152},x_{252},x_{342}, 
\end{matrix}
\end{math} 
\;  $[5,1,1,1,1]$ \\
\hline \hline 
\size  $17$ & \size $M_{[1,2,4,5],[1]}$, \;
$1\oplus \Lam^{2,1}_{1,[3,4]}\oplus 1^{3\oplus}$ 
& \size $[-1,3,-7,7,-3,1,-2,2]$ \\
\cline{2-3}
& \size 
\begin{math}
\begin{matrix}
x_{351}=0
\end{matrix}
\end{math}
& \size 
\begin{math}
\begin{matrix}
x_{261},x_{451},x_{162},x_{252},x_{342}, 
\end{matrix}
\end{math} 
\;  $[2,2,2,2,2]$ \\
\hline \hline 
\size $19$ & \size $M_{[5],[1]}$, \;
$\Lam^{5,1}_{1,[1,5]}$ 
& \size $[-1,0,0,0,1,0,0,0]$ \\
\cline{2-3}
& \size
\begin{math}
\begin{matrix}
x_{162},x_{262},x_{362},x_{462}=0
\end{matrix}
\end{math}
& \size 
\begin{math}
\begin{matrix}
x_{562}, 
\end{matrix}
\end{math} 
\;  $[1]$ \\
\hline \hline 
\size $20$ & \size $M_{[4,5],[1]}$, \;
$1\oplus \Lam^{4,1}_{1,[1,4]}$ 
& \size $[-2,-2,-2,6,-3,3,8,-8]$ \\
\cline{2-3}
& \size
\begin{math}
\begin{matrix}
x_{162},x_{262},x_{362}=0
\end{matrix}
\end{math}
& \size 
\begin{math}
\begin{matrix}
x_{561},x_{462},
\end{matrix}
\end{math} 
\;  $[8,1]$ \\
\hline
\hline
\size  $21$ & \size $M_{[1,3],[1]}$, \;
\begin{math}
\Lam^{3,2}_{1,[4,6]}\oplus \Lam^{2,1}_{1,[2,3]}\otimes \Lam^{3,1}_{1,[4,6]}
\end{math}
& \size $[3,-5,2,-6,1,5,-5,5]$  \\
\cline{2-3}
& \size $x_{451},x_{461},x_{242}=0$ 
& \size 
\begin{math}
\begin{matrix}
x_{561},x_{252},x_{262},x_{342},x_{352},x_{362}, 
\end{matrix}
\end{math}
\;  
\begin{math}
\begin{matrix}
[1,1,5,1,8,12]
\end{matrix}
\end{math}
\\
\hline \hline 
\size  $22$ & \size $M_{[3,4],[1]}$, \;
\begin{math}
\Lam^{2,1}_{1,[5,6]}\oplus \Lam^{3,1}_{1,[1,3]}\otimes \Lam^{2,1}_{1,[5,6]}
\end{math}
& \size $[-4,2,2,0,-1,1,0,0]$ \\
\cline{2-3}
& \size $x_{451},x_{152},x_{162}=0$ 
& \size 
\begin{math}
\begin{matrix}
x_{461},x_{252},x_{262},x_{352},x_{362}, 
\end{matrix}
\end{math}
\;  
\begin{math}
[1,1,3,1,3]
\end{math}
\\
\hline \hline 
\size  $23$ & \size $M_{[3,4,5],[1]}$, \;
\begin{math}
1\oplus \Lam^{3,1}_{1,[1,3]}\oplus 1
\end{math}
& \size $[-4,-1,2,3,0,0,1,-1]$  \\
\cline{2-3}
& \size $x_{162},x_{262}=0$ 
& \size 
\begin{math}
\begin{matrix}
x_{561},x_{362},x_{452}, 
\end{matrix}
\end{math}
\;  
\begin{math}
[1,1,2]
\end{math}
\\
\hline \hline 
\size  $24$ & \size $M_{[3,5],\emptyset}$, \;
\begin{math}
\Lam^{3,1}_{1,[1,3]}\otimes \Lam^{2,1}_{2,[1,2]}\oplus \Lam^{2,1}_{2,[1,2]}
\end{math}
& \size $[-4,2,2,0,0,0,-1,1]$  \\
\cline{2-3}
& \size $x_{161},x_{451},x_{162}=0$ 
& \size 
\begin{math}
\begin{matrix}
x_{261},x_{361},x_{262},x_{362},x_{452}, 
\end{matrix}
\end{math}
\;  
\begin{math}
[1,1,3,3,1]
\end{math}
\\
\hline \hline 
\size  $25$ & \size $M_{[3,5],[1]}$, \;
\begin{math}
\Lam^{2,1}_{1,[4,5]}\oplus \Lam^{3,1}_{1,[1,3]}
\end{math}
& \size $[-1,0,1,-1,1,0,0,0]$ \\
\cline{2-3}
& \size $x_{461},x_{162},x_{262}=0$
& \size 
\begin{math}
\begin{matrix}
x_{561},x_{362}, 
\end{matrix}
\end{math}
\;  
\begin{math}
[1,1]
\end{math}
\\
\hline \hline  
\size  $26$ & \size $M_{[2,5],[1]}$, \;
\begin{math}
\Lam^{3,1}_{1,[3,5]}\oplus \Lam^{2,1}_{1,[1,2]}\otimes \Lam^{3,1}_{1,[3,5]}
\end{math}
& \size $[-12,16,-18,9,9,-4,-4,4]$ \\
\cline{2-3}
& \size $x_{361},x_{461},x_{132}=0$ 
& \size 
\begin{math}
\begin{matrix}
x_{561},x_{142},x_{152},x_{232}x_{242},x_{252}, 
\end{matrix}
\end{math}
\;  
\begin{math}
[1,1,1,2,29,29]
\end{math}
\\
\hline \hline 
\size  $27$ & \size $M_{[2,5],[1]}$, \;
\begin{math}
\Lam^{2,1}_{1,[1,2]}\oplus \Lam^{3,2}_{1,[3,5]}
\end{math}
& \size $[-1,1,-1,-1,2,0,0,0]$ \\
\cline{2-3}
& \size $x_{162},x_{342},x_{352}=0$ 
& \size 
\begin{math}
\begin{matrix}
x_{262},x_{452}, 
\end{matrix}
\end{math}
\;  
\begin{math}
[1,1]
\end{math}
\\
\hline \hline 
\size  $28$ & \size $M_{[1,4,5],[1]}$, \;
\begin{math}
1^{2\oplus}\oplus \Lam^{3,2}_{1,[2,4]}
\end{math}
& \size $[0,-17,-3,11,0,9,7,-7]$ \\
\cline{2-3}
& \size $x_{232},x_{242}=0$
& \size 
\begin{math}
\begin{matrix}
x_{561},x_{162},x_{342}, 
\end{matrix}
\end{math}
\;  
\begin{math}
[16,2,1]
\end{math}
\\
\hline \hline 
\size  $29$ & \size $M_{[1,5],[1]}$, \;
\begin{math}
\Lam^{4,1}_{1,[2,5]}\oplus \Lam^{4,2}_{1,[2,5]}
\end{math}
& \size $[0,-7,-12,14,9,-4,-4,4]$ \\
\cline{2-3}
& \size $x_{261},x_{361},x_{461},x_{232},x_{242}=0$
& \size 
\begin{math}
\begin{matrix}
x_{561},x_{252},x_{342},x_{352},x_{452}, 
\end{matrix}
\end{math}
\;  
\begin{math}
[1,6,6,1,27]
\end{math}
\\
\hline 

\end{tabular}

\end{center}

\begin{center}

\begin{tabular}{|l|l|l|}
\hline
\size $i$ & $\;$ \hskip 0.8in \size $M_{\be_i}$, $Z_{\be_i}$ \hskip 0.8in $\;$ 
& $\;$ \hskip 1.2in \size 1PS \hskip 1.2in $\;$ \\
\cline{2-3}
& \hskip 0.7in \size zero coordinates 
& \hskip 0.5in \size non-zero coordinates and their weights \\
\hline \hline 
\size  $30$ & \size $M_{[1,3],[1]}$, \;
\begin{math}
\Lam^{3,2}_{1,[4,6]}\oplus \Lam^{3,1}_{1,[4,6]}\oplus 1
\end{math}
& \size $[5,-1,-1,-8,1,4,-4,4]$ \\
\cline{2-3}
& \size $x_{451},x_{461}=0$
& \size 
\begin{math}
\begin{matrix}
x_{561},x_{142},x_{152},x_{162},x_{232}, 
\end{matrix}
\end{math}
\;  
\begin{math}
[1,1,10,13,2]
\end{math}
\\
\hline \hline 
\size  $31$ & \size $M_{[2,5],[1]}$, \;
\begin{math}
\Lam^{2,1}_{1,[1,2]}\oplus \Lam^{3,2}_{1,[3,5]}
\end{math}
& \size $[-1,1,-1,0,1,0,0,0]$ \\
\cline{2-3}
& \size $x_{161},x_{342},x_{352}=0$ 
& \size 
\begin{math}
\begin{matrix}
x_{261},x_{452}, 
\end{matrix}
\end{math}
\;  
\begin{math}
[1,1]
\end{math}
\\
\hline \hline 
\size  $32$ & \size $M_{[1],[1]}$, \;
\begin{math}
\Lam^{5,2}_{1,[2,6]}
\end{math}
& \size $[0,-4,1,1,1,1,0,0]$ \\
\cline{2-3}
& \size $x_{232},x_{242},x_{252},x_{262}=0$. 
& \size 
\begin{math}
\begin{matrix}
x_{342},x_{352},x_{362},x_{452},x_{462},x_{562}, 
\end{matrix}
\end{math}
\;  
\begin{math}
[2,2,2,2,2,2]
\end{math}
\\
\hline \hline 
\size  $33$ & \size $M_{[1,4],[1]}$, \;
\begin{math}
1\oplus \Lam^{3,1}_{1,[2,4]}\otimes \Lam^{2,1}_{1,[5,6]}
\end{math}
& \size $[-21,-5,13,13,0,0,12,-12]$ \\
\cline{2-3}
& \size $x_{252},x_{262}=0$
& \size 
\begin{math}
\begin{matrix}
x_{561},x_{352},x_{362},x_{452},x_{462}, 
\end{matrix}
\end{math}
\;  
\begin{math}
[12,1,1,1,1]
\end{math}
\\
\hline \hline 
\size  $34$ & \size $M_{[1,3,5],[1]}$, \;
\begin{math}
\Lam^{2,1}_{1,[4,5]}\oplus 1\oplus \Lam^{2,1}_{1,[2,3]}\otimes \Lam^{2,1}_{1,[4,5]}
\end{math}
& \size $[3,-8,8,-8,8,-3,-4,4]$ \\
\cline{2-3}
& \size $x_{461},x_{242}=0$
& \size 
\begin{math}
\begin{matrix}
x_{561},x_{162},x_{252},x_{342},x_{352}, 
\end{matrix}
\end{math}
\;  
\begin{math}
[1,4,4,4,20]
\end{math}
\\
\hline \hline 
\size  $36$ & \size $M_{[1,2,4],[1]}$, \;
\begin{math}
\Lam^{2,1}_{1,[3,4]}\otimes \Lam^{2,1}_{1,[5,6]}\oplus \Lam^{2,1}_{1,[5,6]}\oplus 1
\end{math}
& \size $[0,-2,-2,6,-5,3,0,0]$ \\
\cline{2-3}
& \size $x_{351},x_{252}=0$
& \size 
\begin{math}
\begin{matrix}
x_{361},x_{451},x_{461},x_{262},x_{342}, 
\end{matrix}
\end{math}
\;  
\begin{math}
[1,1,9,1,4]
\end{math}
\\
\hline \hline 
\size  $37$ & \size $M_{[4],[1]}$, \;
\begin{math}
1\oplus \Lam^{4,1}_{1,[1,4]}\otimes \Lam^{2,1}_{1,[5,6]}
\end{math}
& \size $[-5,-5,3,3,2,2,4,-4]$ \\
\cline{2-3}
& \size $x_{152},x_{162},x_{252},x_{262}=0$ 
& \size 
\begin{math}
\begin{matrix}
x_{561},x_{352},x_{362},x_{452},x_{462},
\end{matrix}
\end{math}
\;  
\begin{math}
[8,1,1,1,1]
\end{math}
\\
\hline \hline 
\size  $38$ & \size $M_{[1,3,4],[1]}$, \;
\begin{math}
1\oplus \Lam^{2,1}_{1,[5,6]}\oplus \Lam^{2,1}_{1,[2,3]}
\end{math}
& \size $[1,-12,12,-11,5,5,0,0]$ \\
\cline{2-3}
& \size $x_{242}=0$
& \size 
\begin{math}
\begin{matrix}
x_{561},x_{152},x_{162},x_{342}, 
\end{matrix}
\end{math}
\;  
\begin{math}
[1,9,13,1]
\end{math}
\\
\hline \hline 
\size  $39$ & \size $M_{[1,3,5],[1]}$, \;
\begin{math}
\Lam^{2,1}_{1,[4,5]}\oplus \Lam^{2,1}_{1,[2,3]}\oplus 1
\end{math}
& \size $[0,-1,1,-3,5,-2,-2,2]$ \\
\cline{2-3}
& \size $x_{461},x_{262}=0$ 
& \size 
\begin{math}
\begin{matrix}
x_{561},x_{362},x_{452}, 
\end{matrix}
\end{math}
\;  
\begin{math}
[1,1,4]
\end{math}
\\
\hline \hline 
\size  $40$ & \size $M_{[1,2,3,5],[1]}$, \;
\begin{math}
1^{3\oplus}\Lam^{2,1}_{1,[4,5]}
\end{math}
& \size $[1,-1,-4,-7,7,4,2,-2]$ \\
\cline{2-3}
& \size $x_{342}=0$
& \size 
\begin{math}
\begin{matrix}
x_{361},x_{451},x_{262},x_{352}, 
\end{matrix}
\end{math}
\;  
\begin{math}
[2,2,1,1]
\end{math}
\\
\hline\hline
\size  $41$ & \size $M_{[2,3,4,5],[1]}$, \;
\begin{math}
1^{2\oplus}\Lam^{2,1}_{1,[1,2]}\oplus 1
\end{math}
& \size $[-7,3,4,1,0,-1,1,-1]$ \\
\cline{2-3}
& \size $x_{162}=0$
& \size 
\begin{math}
\begin{matrix}
x_{361},x_{451},x_{262},x_{352}, 
\end{matrix}
\end{math}
\;  
\begin{math}
[4,2,1,3]
\end{math}
\\
\hline \hline 
\size  $42$ & \size $M_{[1,3,5],[1]}$, \;
\begin{math}
\Lam^{2,1}_{1,[2,3]}\oplus 1^{3\oplus}
\end{math}
& \size $[3,-18,14,4,4,-7,-6,6]$ \\
\cline{2-3}
& \size $x_{261}=0$ 
& \size 
\begin{math}
\begin{matrix}
x_{361},x_{451},x_{162},x_{232}, 
\end{matrix}
\end{math}
\;  
\begin{math}
[1,2,2,2]
\end{math}
\\
\hline \hline 
\size  $43$ & \size $M_{[1,2,3,5],[1]}$, \;
\begin{math}
1^{3\oplus}\oplus \Lam^{2,1}_{1,[4,5]}
\end{math}
& \size $[0,5,-7,-9,9,2,1,-1]$ \\
\cline{2-3}
& \size $x_{342}=0$
& \size 
\begin{math}
\begin{matrix}
x_{261},x_{451},x_{162},x_{352}, 
\end{matrix}
\end{math}
\;  
\begin{math}
[8,1,1,1]
\end{math}
\\
\hline \hline 
\size  $44$ & \size $M_{[1,4,5],[1]}$, \;
\begin{math}
\Lam^{3,1}_{1,[2,4]}\oplus 1\oplus \Lam^{3,1}_{1,[2,4]}
\end{math}
& \size $[0,-5,-5,13,0,-3,-6,6]$ \\
\cline{2-3}
& \size $x_{261},x_{361}=0$ 
& \size 
\begin{math}
\begin{matrix}
x_{461},x_{162},x_{252},x_{352},x_{452}, 
\end{matrix}
\end{math}
\;  
\begin{math}
[4,3,1,1,19]
\end{math}
\\
\hline \hline 
\size  $45$ & \size $M_{[1,3,4,5],[1]}$, \;
\begin{math}
\Lam^{2,1}_{1,[2,3]}\oplus 1^{2\oplus}\oplus \Lam^{2,1}_{1,[2,3]}
\end{math}
& \size $[-13,-5,27,0,8,-17,-6,6]$ \\
\cline{2-3}
& \size $x_{261}=0$
& \size 
\begin{math}
\begin{matrix}
x_{361},x_{451},x_{152},x_{242},x_{342}, 
\end{matrix}
\end{math}
\;  
\begin{math}
[4,2,1,1,33]
\end{math}
\\
\hline \hline 
\size  $47$ & \size $M_{[1,5],[1]}$, \;
\begin{math}
\Lam^{4,1}_{1,[2,5]}
\end{math}
& \size $[0,-1,0,0,1,0,0,0]$ \\
\cline{2-3}
& \size $x_{262},x_{362},x_{462}=0$ 
& \size 
\begin{math}
\begin{matrix}
x_{562}, 
\end{matrix}
\end{math}
\;  
\begin{math}
[1]
\end{math}
\\
\hline \hline 
\size  $48$ & \size $M_{[3,5],[1]}$, \;
\begin{math}
\Lam^{3,1}_{1,[1,3]}\oplus 1
\end{math}
& \size $[-6,-1,4,3,3,-3,0,0]$ \\
\cline{2-3}
& \size $x_{162},x_{262}=0$ 
& \size 
\begin{math}
\begin{matrix}
x_{362},x_{452}, 
\end{matrix}
\end{math}
\;  
\begin{math}
[1,6]
\end{math}
\\
\hline \hline 
\size  $49$ & \size $M_{[1,4,5],[1]}$, \;
\begin{math}
1\oplus \Lam^{3,1}_{1,[2,4]}
\end{math}
& \size $[-1,-5,0,5,0,1,5,-5]$ \\
\cline{2-3}
& \size $x_{262},x_{362}=0$
& \size 
\begin{math}
\begin{matrix}
x_{561},x_{462}, 
\end{matrix}
\end{math}
\;  
\begin{math}
[6,1]
\end{math}
\\
\hline \hline 
\size  $50$ & \size $M_{[3,5],[1]}$, \;
\begin{math}
1\oplus \Lam^{3,1}_{1,[1,3]}
\end{math}
& \size $[-11,-1,9,0,0,3,9,-9]$ \\
\cline{2-3}
& \size $x_{162},x_{262}=0$
& \size 
\begin{math}
\begin{matrix}
x_{451},x_{362}, 
\end{matrix}
\end{math}
\;  
\begin{math}
[9,3]
\end{math}
\\
\hline \hline 
\size  $51$ & \size $M_{[1,4],[1]}$, \;
\begin{math}
1\oplus \Lam^{3,2}_{1,[2,4]}
\end{math}
& \size $[-2,-8,8,0,1,1,7,-7]$ \\
\cline{2-3}
& \size $x_{232},x_{242}=0$
& \size 
\begin{math}
\begin{matrix}
x_{561},x_{342}, 
\end{matrix}
\end{math}
\;  
\begin{math}
[9,1]
\end{math}
\\
\hline 
\end{tabular}

\end{center}

\begin{center}

\begin{tabular}{|l|l|l|}
\hline
\size $i$ & $\;$ \hskip 0.8in \size $M_{\be_i}$, $Z_{\be_i}$ \hskip 0.8in $\;$ 
& $\;$ \hskip 1.2in \size 1PS \hskip 1.2in $\;$ \\
\cline{2-3}
& \hskip 0.7in \size zero coordinates 
& \hskip 0.5in \size non-zero coordinates and their weights \\
\hline \hline 
\size  $52$ & \size $M_{[1,4],[1]}$, \;
\begin{math}
\Lam^{2,1}_{1,[5,6]}\otimes \Lam^{3,1}_{1,[2,4]}
\end{math}
& \size $[0,-2,1,1,0,0,0,0]$ \\
\cline{2-3}
& \size $x_{252},x_{262}=0$ 
& \size 
\begin{math}
\begin{matrix}
x_{352},x_{362},x_{452},x_{462}, 
\end{matrix}
\end{math}
\;  
\begin{math}
[1,1,1,1]
\end{math}
\\
\hline \hline 
\size  $53$ & \size $M_{[1,3,4],[1]}$, \;
\begin{math}
\Lam^{2,1}_{1,[5,6]}\oplus \Lam^{2,1}_{1,[5,6]}\otimes \Lam^{2,1}_{1,[2,3]}
\end{math}
& \size $[0,-4,6,-2,-5,5,-2,2]$ \\
\cline{2-3}
& \size $x_{451},x_{252}=0$
& \size 
\begin{math}
\begin{matrix}
x_{461},x_{262},x_{352},x_{362}, 
\end{matrix}
\end{math}
\;  
\begin{math}
[1,3,3,13]
\end{math}
\\
\hline \hline 
\size  $54$ & \size $M_{[4],[1]}$, \;
\begin{math}
\Lam^{4,1}_{1,[1,4]}\otimes \Lam^{2,1}_{1,[5,6]}
\end{math}
& \size $[-1,-1,1,1,0,0,0,0]$ \\
\cline{2-3}
& \size $x_{152},x_{162},x_{252},x_{262}=0$  
& \size 
\begin{math}
\begin{matrix}
x_{352},x_{362},x_{452},x_{462}, 
\end{matrix}
\end{math}
\;  
\begin{math}
[1,1,1,1]
\end{math}
\\
\hline \hline 
\size  $55$ & \size $M_{[1,3,4,5],[1]}$, \;
\begin{math}
1\oplus \Lam^{2,1}_{1,[2,3]}\oplus 1
\end{math}
& \size $[0,-4,4,-1,4,-3,0,0]$ \\
\cline{2-3}
& \size $x_{262}=0$
& \size 
\begin{math}
\begin{matrix}
x_{561},x_{362},x_{452}, 
\end{matrix}
\end{math}
\;  
\begin{math}
[1,1,3]
\end{math}
\\
\hline \hline 
\size  $56$ & \size $M_{[2,5],[1]}$, \;
\begin{math}
\Lam^{3,1}_{1,[3,5]}\oplus \Lam^{2,1}_{1,[1,2]}\oplus \Lam^{3,2}_{1,[3,5]}
\end{math}
& \size $[-9,5,3,2,1,-2,2,-2]$ \\
\cline{2-3}
& \size $x_{361},x_{461},x_{162}=0$
& \size 
\begin{math}
\begin{matrix}
x_{561},x_{262},x_{342},x_{352},x_{452}, 
\end{matrix}
\end{math}
\;  
\begin{math}
[1,1,3,2,1]
\end{math}
\\
\hline \hline 
\size  $57$ & \size $M_{[1,2,4],[1]}$, \;
\begin{math}
1\oplus \Lam^{2,1}_{1,[5,6]}\oplus 1
\end{math}
& \size $[0,-2,1,1,-4,4,1,-1]$ \\
\cline{2-3}
& \size $x_{252}=0$
& \size 
\begin{math}
\begin{matrix}
x_{561},x_{262},x_{342}, 
\end{matrix}
\end{math}
\;  
\begin{math}
[1,1,1]
\end{math}
\\
\hline \hline 
\size  $58$ & \size $M_{[1,3,4,5],[1]}$, \;
\begin{math}
1^{2\oplus}\oplus \Lam^{2,1}_{1,[2,3]}
\end{math}
& \size $[0,-3,3,1,-2,1,0,0]$ \\
\cline{2-3}
& \size $x_{252}=0$ 
& \size 
\begin{math}
\begin{matrix}
x_{461},x_{162},x_{352}, 
\end{matrix}
\end{math}
\;  
\begin{math}
[2,1,1] 
\end{math}
\\
\hline \hline 
\size  $59$ & \size $M_{[2,5],\emptyset}$, \;
\begin{math}
\Lam^{3,1}_{1,[3,5]}\otimes \Lam^{2,1}_{2,[1,2]}
\end{math}
& \size $[0,0,-2,1,1,0,0,0]$ \\
\cline{2-3}
& \size $x_{361},x_{362}=0$ 
& \size 
\begin{math}
\begin{matrix}
x_{461},x_{561},x_{462},x_{562}, 
\end{matrix}
\end{math}
\;  
\begin{math}
[1,1,1,1]
\end{math}
\\
\hline \hline 
\size  $60$ & \size $M_{[3],\emptyset}$, \;
\begin{math}
\Lam^{3,2}_{1,[4,6]}\otimes \Lam^{2,1}_{2,[1,2]}
\end{math}
& \size $[0,0,0,-1,-1,2,0,0]$ \\
\cline{2-3}
& \size $x_{451},x_{452}=0$ 
& \size 
\begin{math}
\begin{matrix}
x_{461},x_{561},x_{462},x_{562}, 
\end{matrix}
\end{math}
\;  
\begin{math}
[1,1,1,1]
\end{math}
\\
\hline\hline
\size  $61$ & \size $M_{[1,3,5],\emptyset}$, \;
\begin{math}
\Lam^{2,1}_{1,[2,3]}\otimes \Lam^{2,1}_{2,[1,2]}\oplus \Lam^{2,1}_{2,[1,2]}
\end{math}
& \size $[1,-5,5,-2,-2,3,-5,5]$ \\
\cline{2-3}
& \size $x_{261},x_{451}=0$
& \size 
\begin{math}
\begin{matrix}
x_{361},x_{262},x_{362},x_{452}, 
\end{matrix}
\end{math}
\;  
\begin{math}
[3,3,13,1]
\end{math}
\\
\hline \hline 
\size  $62$ & \size $M_{[1,5],\emptyset}$, \;
\begin{math}
\Lam^{4,1}_{1,[2,5]}\otimes \Lam^{2,1}_{2,[1,2]}
\end{math}
& \size $[0,-1,-1,1,1,0,0,0]$ \\
\cline{2-3}
& \size $x_{261},x_{361},x_{262},x_{362}=0$ 
& \size 
\begin{math}
\begin{matrix}
x_{461},x_{561},x_{462},x_{562}, 
\end{matrix}
\end{math}
\;  
\begin{math}
[1,1,1,1]
\end{math}
\\
\hline \hline 
\size  $63$ & \size $M_{[2,3,5],[1]}$, \;
\begin{math}
\Lam^{2,1}_{1,[4,5]}\oplus 1^{2\oplus}
\end{math}
& \size $[0,0,1,-4,4,-1,-2,2]$ \\
\cline{2-3}
& \size $x_{461}=0$ 
& \size 
\begin{math}
\begin{matrix}
x_{561},x_{362},x_{452}, 
\end{matrix}
\end{math}
\;  
\begin{math}
[1,2,2]
\end{math}
\\
\hline \hline 
\size  $64$ & \size $M_{[1,3,5],[1]}$, \;
\begin{math}
\Lam^{2,1}_{1,[2,3]}\oplus 1^{2\oplus}
\end{math}
& \size $[1,-8,8,0,0,-1,-6,6]$ \\
\cline{2-3}
& \size $x_{261}=0$
& \size 
\begin{math}
\begin{matrix}
x_{361},x_{162},x_{452}, 
\end{matrix}
\end{math}
\;  
\begin{math}
[1,6,6]
\end{math}
\\
\hline \hline 
\size  $65$ & \size $M_{[1,3,5],[1]}$, \;
\begin{math}
(\Lam^{2,1}_{1,[4,5]})^{2\oplus}\oplus 1
\end{math}
& \size $[8,-2,-2,-11,11,-4,-5,5]$ \\
\cline{2-3}
& \size $x_{461}=0$ 
& \size 
\begin{math}
\begin{matrix}
x_{561},x_{142},x_{152},x_{232}, 
\end{matrix}
\end{math}
\;  
\begin{math}
[2,2,24,1]
\end{math}
\\
\hline \hline 
\size  $68$ & \size $M_{[2,5],[1]}$, \;
\begin{math}
\Lam^{3,1}_{1,[3,5]}
\end{math}
& \size $[0,0,-1,0,1,0,0,0]$ \\
\cline{2-3}
& \size $x_{362},x_{462}=0$ 
& \size 
\begin{math}
\begin{matrix}
x_{562}, 
\end{matrix}
\end{math}
\;  
\begin{math}
[1]
\end{math}
\\
\hline \hline 
\size  $69$ & \size $M_{[3],[1]}$, \;
\begin{math}
\Lam^{3,2}_{1,[4,6]}
\end{math}
& \size $[0,0,0,-1,-1,2,0,0]$ \\
\cline{2-3}
& \size $x_{452},x_{462}=0$
& \size 
\begin{math}
\begin{matrix}
x_{562}, 
\end{matrix}
\end{math}
\;  
\begin{math}
[1]
\end{math}
\\
\hline \hline 
\size  $70$ & \size $M_{[1,3,5],[1]}$, \;
\begin{math}
\Lam^{2,1}_{1,[2,3]}\oplus 1
\end{math}
& \size $[-1,-4,4,2,2,-3,0,0]$ \\
\cline{2-3}
& \size $x_{262}=0$  
& \size 
\begin{math}
\begin{matrix}
x_{362},x_{452}, 
\end{matrix}
\end{math}
\;  
\begin{math}
[1,4]
\end{math}
\\
\hline \hline 
\size  $71$ & \size $M_{[2,4,5],[1]}$, \;
\begin{math}
1\oplus \Lam^{2,1}_{1,[3,4]}
\end{math}
& \size $[0,0,-4,4,-1,1,4,-4]$ \\
\cline{2-3}
& \size $x_{362}=0$
& \size 
\begin{math}
\begin{matrix}
x_{561},x_{462}, 
\end{matrix}
\end{math}
\;  
\begin{math}
[4,1]
\end{math}
\\
\hline \hline 
\size  $72$ & \size $M_{[3,4],[1]}$, \;
\begin{math}
1\oplus \Lam^{2,1}_{1,[5,6]}
\end{math}
& \size $[0,0,0,-2,-3,5,2,-2]$ \\
\cline{2-3}
& \size $x_{452}=0$ 
& \size 
\begin{math}
\begin{matrix}
x_{561},x_{462}, 
\end{matrix}
\end{math}
\;  
\begin{math}
[4,1]
\end{math}
\\
\hline \hline 
\size  $73$ & \size $M_{[1,3,5],[1]}$, \;
\begin{math}
1\oplus \Lam^{2,1}_{1,[2,3]}
\end{math}
& \size $[-1,-6,6,0,0,1,6,-6]$ \\
\cline{2-3}
& \size $x_{262}=0$
& \size 
\begin{math}
\begin{matrix}
x_{451},x_{362}, 
\end{matrix}
\end{math}
\;  
\begin{math}
[6,1]
\end{math}
\\
\hline 
\end{tabular}

\end{center}

\vskip 10pt

\begin{center}

\begin{tabular}{|l|l|l|}
\hline
\size $i$ & $\;$ \hskip 0.8in \size $M_{\be_i}$, $Z_{\be_i}$ \hskip 0.8in $\;$ 
& $\;$ \hskip 1.2in \size 1PS \hskip 1.2in $\;$ \\
\cline{2-3}
& \hskip 0.7in \size zero coordinates 
& \hskip 0.5in \size non-zero coordinates and their weights \\
\hline \hline 
\size  $77$ & \size $M_{[3,5],[1]}$, \;
\begin{math}
\Lam^{2,1}_{1,[4,5]}
\end{math}
& \size $[0,0,0,-1,1,0,0,0]$ \\
\cline{2-3}
& \size $x_{462}=0$ 
& \size 
\begin{math}
\begin{matrix}
x_{562}, 
\end{matrix}
\end{math}
\;  
\begin{math}
[1]
\end{math}
\\
\hline \hline 
\size  $79$ & \size $M_{[4],\emptyset}$, \;
\begin{math}
\Lam^{2,1}_{2,[1,2]}
\end{math}
& \size $[0,0,0,0,0,0,-1,1]$ \\
\cline{2-3}
& \size $x_{561}=0$ 
& \size 
\begin{math}
\begin{matrix}
x_{562}, 
\end{matrix}
\end{math}
\;  
\begin{math}
[1]
\end{math}
\\
\hline 

\end{tabular}

\end{center}

\vskip 10pt

(1) $\be_1= \tfrac {1} {48} (-1,-1,-1,-1,-1,5,-3,3)$. 


$Z_{\be_1}$ is spanned by $\coorde_i=e_{jkl}$ 
for the following $i,jkl$. 

\tiny
\vskip 5pt

\begin{center}
 
\begin{tabular}{|c|c|c|c|c|c|c|c|c|c|c|c|c|c|c|c|}
\hline
$i$ & $\underline{5}$ & $\underline{9}$ 
& $\underline{12}$ & $\underline{14}$ & $15$ & $16$ & $17$ 
& $18$ & $19$ & $21$ & $22$ & $23$ & $25$ & $26$ & $28$ \\
\hline
$jkl$ & $\underline{161}$ & $\underline{261}$ 
& $\underline{361}$ & $\underline{461}$ & $561$ 
& $122$ & $132$ & $142$ & $152$ & $232$ 
& $242$ & $252$ & $342$ & $352$ & $452$ \\
\hline
\end{tabular}

\end{center}
\normalsize

\vskip 5pt




Lemma \ref{lem:eliminate-standard} implies that 
we may assume that
$x_{161},x_{261},x_{361},x_{461}=0$. 

\vskip 5pt

(2) $\be_2=\tfrac {1} {8} (-1,-1,-1,1,1,1,-1,1)$.  

$Z_{\be_2}$ is spanned by $\coorde_i=e_{jkl}$ 
for the following $i,jkl$. 

\tiny
\vskip 5pt

\begin{center}
 
\begin{tabular}{|c|c|c|c|c|c|c|c|c|c|c|c|c|c|c|c|}
\hline
$i$ & $\underline{13}$ & $\underline{14}$ 
& $15$ & $\underline{18}$ & $\underline{19}$ & $20$ & $22$ 
& $23$ & $24$ & $25$ & $26$ & $27$ \\
\hline
$jkl$ & $\underline{451}$ & $\underline{461}$ 
& $561$ & $\underline{142}$ & $\underline{152}$ 
& $162$ & $242$ & $252$ & $262$ & $342$ 
& $352$ & $362$ \\
\hline
\end{tabular}

\end{center}
\normalsize

\vskip 5pt

%


%
%
%
%

Note that $M^s_{[3],[1]}\cong \spl_3\times \spl_3$. 
We apply Lemma \ref{lem:alernating-matrix} 
(resp. Lemma \ref{lem:eliminate-nm-matrix})
to the action of $\spl_3$ of the second factor 
(resp. the first factor) and 
may assume that $x_{451},x_{461}=0$ 
(resp. $x_{142},x_{152}=0$).

\vskip 5pt

(3) $\be_3=\tfrac {1} {15} (-5,1,1,1,1,1,0,0)$.  


$Z_{\be_3}$ is spanned by $\coorde_i=e_{jkl}$ 
for the following $i,jkl$. 

\tiny
\vskip 5pt

\begin{center}
 
\begin{tabular}{|c|c|c|c|c|c|c|c|c|c|c|c|c|c|c|c|c|c|c|c|c|c|}
\hline
$i$ & $\underline{6}$ & $\underline{7}$ & $\underline{8}$ 
& $\underline{9}$ & $10$ & $\underline{11}$ 
& $\underline{12}$ & $\underline{13}$ & $\underline{14}$ & $15$ \\
\hline
$jkl$ & $\underline{231}$ & $\underline{241}$ 
& $\underline{251}$ & $\underline{261}$ & $341$ 
& $\underline{351}$ & $\underline{361}$ 
& $\underline{451}$ & $\underline{461}$ & $561$ \\
\hline
$i$ & $\underline{21}$ & $22$ & $\underline{23}$ 
& $24$ & $25$ & $26$ & $27$ & $28$ & $29$ & $30$ \\
\hline
$jkl$ & $\underline{232}$ & $242$ & $\underline{252}$ & $262$ & $342$ 
& $352$ & $362$ & $452$ & $462$ & $562$ \\
\hline
\end{tabular}

\end{center}
\normalsize

\vskip 5pt

%

Note that $M^s_{[1],\emptyset}\cong \spl_5\times \spl_2$.  
We identify $Z_{\be_3}$
with the space of pairs $(A_1,A_2)$ of $5\times 5$ alternating matrices.
By Lemma \ref{lem:alernating-matrix} (2), we may assume that 
the only possible non-zero entries of $A_1$ are the 
$(2,3),(4,5)$-entries $x_{341},x_{561}$. 
So $x_{231},x_{241},x_{251},x_{261}$, 
$x_{351},x_{361},x_{451},x_{461}=0$. 
Matrices of the form $\diag(I_2,g_1,g_2)$ with $g_1,g_2\in \spl_2$ 
do not change this condition. By the action of $\diag(I_2,g_1,g_2)$ on $A_2$, 
the first column of $A_2$ can be identified with two 
standard \rep s of $\spl_2$. Therefore, Lemma \ref{lem:eliminate-standard} 
implies that we may further assume that $x_{232},x_{252}=0$. 

\vskip 5pt 

(4) $\be_4=\tfrac {1} {120} (-7,-1,-1,-1,5,5,-3,3)$. 


$Z_{\be_4}$ is spanned by $\coorde_i=e_{jkl}$ 
for the following $i,jkl$. 

\tiny
\vskip 5pt

\begin{center}
 
\begin{tabular}{|c|c|c|c|c|c|c|c|c|c|c|c|c|c|c|c|}
\hline
$i$ & $\underline{8}$ & $\underline{9}$ & $\underline{11}$ 
& $12$ & $13$ & $14$ & $\underline{19}$ 
& $20$ & $21$ & $22$ & $25$  \\
\hline
$jkl$ & $\underline{251}$ & $\underline{261}$ 
& $\underline{351}$ & $361$ & $451$ 
& $461$ & $\underline{152}$ & $162$ & $232$ & $242$ 
& $342$  \\
\hline
\end{tabular}

\end{center}
\normalsize

\vskip 5pt



We identify $\lan e_{251}\ccd e_{461} \ran$ with $\m_{3,2}$.  
Lemma \ref{lem:eliminate-nm-matrix} implies that we may assume that 
$x_{251},x_{261}=0$. 
Matrices of the form $\diag(I_2,g_1,g_2)$ with $g_1,g_2\in\spl_2$
do not change this condition. 
By Lemma \ref{lem:eliminate-2m(2)}, we may assume that 
$x_{351},x_{152}=0$. 

\vskip 5pt

(5) $\be_5=\tfrac {7} {138} (-4,-4,-4,2,2,8,-3,3)$. 

$Z_{\be_5}$ is spanned by $\coorde_i=e_{jkl}$ 
for the following $i,jkl$. 

\tiny
\vskip 5pt

\begin{center}

\begin{tabular}{|c|c|c|c|c|c|c|c|c|c|c|c|c|c|c|c|}
\hline
$i$ & $\underline{14}$ & $15$ & $\underline{20}$ 
& $\underline{24}$ & $27$ & $28$ \\
\hline
$jkl$ & $\underline{461}$ & $561$ 
& $\underline{162}$ & $\underline{262}$ & $362$ 
& $452$ \\
\hline
\end{tabular}

\end{center}
\normalsize

\vskip 5pt

Note that $M^s_{[3,5],[1]}\cong \spl_3\times \spl_2$. 
Lemma \ref{lem:eliminate-standard} implies that
we may assume that $x_{461},x_{162},x_{262}=0$. 

\vskip 5pt

(6) $\be_6=\tfrac {1} {138} (-4,-4,-4,2,2,8,-3,3)$. 


$Z_{\be_6}$ is spanned by $\coorde_i=e_{jkl}$ 
for the following $i,jkl$. 

\tiny
\vskip 5pt

\begin{center}
 
\begin{tabular}{|c|c|c|c|c|c|c|c|c|c|c|c|c|c|c|c|}
\hline
$i$ & $\underline{5}$ & $\underline{9}$ & $12$ 
& $13$ & $18$ & $19$ & $22$ 
& $23$ & $25$ & $26$  \\
\hline
$jkl$ & $\underline{161}$ & $\underline{261}$ 
& $361$ & $451$ & $142$ 
& $152$ & $242$ & $252$ & $342$ & $352$ \\
\hline
\end{tabular}

\end{center}
\normalsize

\vskip 5pt

Lemma \ref{lem:eliminate-standard} implies that 
we may assume that $x_{161},x_{261}=0$. 

(7) $\be_7=\tfrac {2} {15} (-1,-1,-1,-1,-1,5,0,0)$.  
%
%

$Z_{\be_7}$ is spanned by $\coorde_i=e_{jkl}$ 
for the following $i,jkl$. 

\tiny
\vskip 5pt

\begin{center}

 
\begin{tabular}{|c|c|c|c|c|c|c|c|c|c|c|c|c|c|c|c|}
\hline
$i$ & $\underline{5}$ & $\underline{9}$ & $\underline{12}$ 
& $14$ & $15$ & $\underline{20}$ & $\underline{24}$ 
& $\underline{27}$ & $29$ & $30$  \\
\hline
$jkl$ & $\underline{161}$ & $\underline{261}$ 
& $\underline{361}$ & $461$ & $561$ 
& $\underline{162}$ & $\underline{262}$ & $\underline{362}$ & $462$ & $562$ \\
\hline
\end{tabular}

\end{center}
\normalsize

\vskip 5pt

We identify $Z_{\be_7}$
with $\m_{5,2}$. Lemma \ref{lem:eliminate-nm-matrix} implies that
we may assume that $x_{161},x_{261}$, $x_{361},x_{162},x_{262},x_{362}=0$.

(8) $\be_9=\tfrac {2} {87} (-7,-1,-1,-1,5,5,-6,6)$.  
%

$Z_{\be_9}$ is spanned by $\coorde_i=e_{jkl}$ 
for the following $i,jkl$. 

\tiny
\vskip 5pt

\begin{center}

 
\begin{tabular}{|c|c|c|c|c|c|c|c|c|c|c|c|c|c|c|c|}
\hline
$i$ & $15$ & $\underline{19}$ & $20$ 
& $\underline{21}$ & $\underline{22}$ & $25$  \\
\hline
$jkl$ & $561$ & $\underline{152}$ & $162$ 
& $\underline{232}$ & $\underline{242}$ & $342$ \\
\hline
\end{tabular}

\end{center}
\normalsize

\vskip 5pt

We apply Lemma \ref{lem:eliminate-standard} 
(resp. Lemma \ref{lem:alernating-matrix} (2))
to $\lan e_{152},e_{162}\ran$ (resp. $\lan e_{232},e_{242},e_{342}\ran$) 
and may assume that $x_{152},x_{232},x_{242}=0$. 

\vskip 5pt

(9) $\be_{10}=\tfrac {1} {10} (-2,0,0,0,0,2,-1,1)$.  

$Z_{\be_{10}}$ is spanned by $\coorde_i=e_{jkl}$ 
for the following $i,jkl$. 

\tiny
\vskip 5pt

\begin{center}

 
\begin{tabular}{|c|c|c|c|c|c|c|c|c|c|c|c|c|c|c|c|}
\hline
$i$ & $\underline{9}$ & $\underline{12}$ 
& $\underline{14}$ & $15$ & $20$ & $\underline{21}$ 
& $\underline{22}$ & $23$ & $25$ & $26$ & $28$ \\
\hline
$jkl$ & $\underline{261}$ & $\underline{361}$ & $\underline{461}$ & $561$ 
& $162$ & $\underline{232}$ & $\underline{242}$ 
& $252$ & $342$ & $352$ & $452$ \\
\hline
\end{tabular}

\end{center}
\normalsize

\vskip 5pt

Lemma \ref{lem:eliminate-standard} implies that
we may assume that $x_{261},x_{361},x_{461}=0$. 
Elements of the form $g=(\diag(1,g_1,I_2),I_2)$ with $g_1\in\spl_3$ 
do not change this condition. We consider the action of $g$ on 
\begin{equation*}
\begin{pmatrix}
0 & x_{232} & x_{242} \\
-x_{232}  & 0 & x_{342} \\
-x_{242} & -x_{342} & 0 
\end{pmatrix}. 
\end{equation*}
Lemma \ref{lem:alernating-matrix} (2) implies that 
we may assume that $x_{232},x_{242}=0$. 

\vskip 5pt

(10) $\be_{11}=\tfrac {1} {210} (-10,-4,2,2,2,8,-3,3)$.  
%

$Z_{\be_{11}}$ is spanned by $\coorde_i=e_{jkl}$ 
for the following $i,jkl$. 

\tiny
\vskip 5pt

\begin{center}

 
\begin{tabular}{|c|c|c|c|c|c|c|c|c|c|c|c|c|c|c|c|}
\hline
$i$ & $9$ & $\underline{10}$ 
& $\underline{11}$ & $13$ & $20$ & $21$
& $22$ & $23$ \\
\hline
$jkl$ & $261$ & $\underline{341}$ & $\underline{351}$ & $451$
& $162$ & $232$ & $242$ 
& $252$ \\
\hline
\end{tabular}

\end{center}
\normalsize

\vskip 5pt

Lemma \ref{lem:alernating-matrix} implies that we may assume that 
$x_{341},x_{351}=0$.  

\vskip 5pt

(11) $\be_{12}=\tfrac {1} {51} (-5,-5,1,1,1,7,0,0)$.  


$Z_{\be_{12}}$ is spanned by $\coorde_i=e_{jkl}$ 
for the following $i,jkl$. 

\tiny
\vskip 5pt

\begin{center}

 
\begin{tabular}{|c|c|c|c|c|c|c|c|c|c|c|c|c|c|c|c|}
\hline
$i$ & $5$ & $9$ & $\underline{10}$ & $11$
& $13$ & $20$ & $24$ & $\underline{25}$ & $26$ & $28$ \\
\hline
$jkl$ & $161$ & $261$ & $\underline{341}$ & $351$ & $451$
& $162$ & $262$ & $\underline{342}$ & $352$ & $452$ \\
\hline
\end{tabular}

\end{center}
\normalsize

\vskip 5pt

We identify $\lan e_{341},e_{351},e_{451},e_{342},e_{352},e_{452}\ran$ 
with $\m_{3,2}$. The situation is not exactly the same with 
that in Lemma \ref{lem:eliminate-nm-matrix},
but almost the same argument works and we may assume that 
$x_{341},x_{342}=0$. 

\vskip 5pt

(12) $\be_{14}=\tfrac {1} {24} (-3,-1,-1,1,1,3,-1,1)$.  

$Z_{\be_{14}}$ is spanned by $\coorde_i=e_{jkl}$ 
for the following $i,jkl$. 

\tiny
\vskip 5pt

\begin{center}

 
\begin{tabular}{|c|c|c|c|c|c|c|c|c|c|c|c|c|c|c|c|}
\hline
$i$ & $\underline{9}$ & $12$ 
& $13$ & $20$ & $\underline{22}$ & $23$ & $25$ & $26$ \\
\hline
$jkl$ & $\underline{261}$ & $361$ & $451$
& $162$ & $\underline{242}$ & $252$ & $342$ & $352$ \\
\hline
\end{tabular}

\end{center}
\normalsize

\vskip 5pt

Lemma \ref{lem:eliminate-2m(2)} implies that we may assume that 
$x_{261},x_{242}=0$. 

\vskip 5pt

(13) $\be_{15}=\tfrac {5} {282} (-8,-8,-2,4,4,10,-3,3)$.  

$Z_{\be_{15}}$ is spanned by $\coorde_i=e_{jkl}$ 
for the following $i,jkl$. 

\tiny
\vskip 5pt

\begin{center}

 
\begin{tabular}{|c|c|c|c|c|c|c|c|c|c|c|c|c|c|c|c|}
\hline
$i$ & $12$ & $13$ & $\underline{20}$ & $24$ & $\underline{25}$ & $26$ \\
\hline
$jkl$ & $361$ & $451$ & $\underline{162}$ 
& $262$ & $\underline{342}$ & $352$ \\
\hline
\end{tabular}

\end{center}
\normalsize

\vskip 5pt

We apply Lemma \ref{lem:eliminate-standard} 
to $\lan e_{162},e_{262}\ran$, $\lan e_{342},e_{352}\ran$ 
and may assume that $x_{162},x_{342}=0$. 

\vskip 5pt

(14) $\be_{16}=\tfrac {1} {264} (-7,-7,-1,-1,5,11,-3,3)$.  

$Z_{\be_{16}}$ is spanned by $\coorde_i=e_{jkl}$ 
for the following $i,jkl$. 

\tiny
\vskip 5pt

\begin{center}

 
\begin{tabular}{|c|c|c|c|c|c|c|c|c|c|c|c|c|c|c|c|}
\hline
$i$ & $\underline{5}$ & $9$ & $\underline{11}$ 
& $13$ & $19$ & $23$ & $25$ \\
\hline
$jkl$ & $\underline{161}$ & $261$ & $\underline{351}$ 
& $451$ & $152$ & $252$ & $342$ \\
\hline
\end{tabular}

\end{center}
\normalsize

\vskip 5pt

We apply Lemma \ref{lem:eliminate-standard} 
to $\lan e_{161},e_{261}\ran$, $\lan e_{351},e_{451}\ran$, 
and may assume that $x_{161},x_{351}=0$.

(15) $\be_{17}=\tfrac {1} {42} (-4,-2,0,0,2,4,-1,1)$. 

$Z_{\be_{17}}$ is spanned by $\coorde_i=e_{jkl}$ 
for the following $i,jkl$. 

\tiny
\vskip 5pt

\begin{center}

 
\begin{tabular}{|c|c|c|c|c|c|c|c|c|c|c|c|c|c|c|c|}
\hline
$i$ & $9$ & $\underline{11}$ & $13$ & $20$ & $23$ & $25$ \\
\hline
$jkl$ & $261$ & $\underline{351}$ 
& $451$ & $162$ & $252$ & $342$ \\
\hline
\end{tabular}

\end{center}
\normalsize

\vskip 5pt

Lemma \ref{lem:eliminate-standard} implies that
we may assume that $x_{351}=0$. 

\vskip 5pt

(16) $\be_{19}=\tfrac {1} {30} (-4,-4,-4,-4,-4,20,-15,15)$ 

$Z_{\be_{19}}$ is spanned by $\coorde_i=e_{jkl}$ 
for the following $i,jkl$. 
\tiny
%
%
%
\begin{tabular}{|c|c|c|c|c|c|c|c|c|c|c|c|c|c|c|c|}
\hline
$i$ & $\underline{20}$ & $\underline{24}$ 
& $\underline{27}$ & $\underline{29}$ & $30$ \\
\hline
$jkl$ & $\underline{162}$ & $\underline{262}$ 
& $\underline{362}$ & $\underline{462}$ & $562$ \\
\hline
\end{tabular}

%
\normalsize

\vskip 5pt

Lemma \ref{lem:eliminate-standard} implies that 
we may assume that $x_{162},x_{262},x_{362},x_{462}=0$. 

\vskip 5pt

(17) $\be_{20}=\tfrac {1} {78} (-14,-14,-14,-14,4,52,-9,9)$.  

$Z_{\be_{20}}$ is spanned by $\coorde_i=e_{jkl}$ 
for the following $i,jkl$. 
\tiny
%
%
%
\begin{tabular}{|c|c|c|c|c|c|c|c|c|c|c|c|c|c|c|c|}
\hline
$i$ & $15$ & $\underline{20}$ 
& $\underline{24}$ & $\underline{27}$ & $29$ \\
\hline
$jkl$ & $561$ & $\underline{162}$ 
& $\underline{262}$ & $\underline{362}$ & $462$ \\
\hline
\end{tabular}

%
\normalsize

\vskip 5pt

Lemma \ref{lem:eliminate-standard} implies that 
we may assume that $x_{162},x_{262},x_{362}=0$. 

\vskip 5pt

(18) $\be_{21}=\tfrac {1} {102} (-34,-4,-4,14,14,14,-9,9)$.  

$Z_{\be_{21}}$ is spanned by $\coorde_i=e_{jkl}$ 
for the following $i,jkl$. 

\tiny
\vskip 5pt

\begin{center}

 
\begin{tabular}{|c|c|c|c|c|c|c|c|c|c|c|c|c|c|c|c|}
\hline
$i$ & $\underline{13}$ & $\underline{14}$ & $15$ 
& $\underline{22}$ & $23$ & $24$ & $25$ & $26$ & $27$ \\
\hline
$jkl$ & $\underline{451}$ & $\underline{461}$ 
& $561$ & $\underline{242}$ & $252$ & $262$ 
& $342$ & $352$ & $362$ \\
\hline
\end{tabular}

\end{center}
\normalsize

\vskip 5pt

We apply Lemma \ref{lem:alernating-matrix} 
(resp. Lemma \ref{lem:eliminate-standard})
to $\lan e_{451},e_{461},e_{561}\ran$ 
(resp. $\lan e_{242},e_{342}\ran$)
and may assume that $x_{451},x_{461},x_{242}=0$.  

\vskip 5pt

(19) $\be_{22}=\tfrac {1} {30} (-4,-4,-4,2,5,5,-3,3)$.  


$Z_{\be_{22}}$ is spanned by $\coorde_i=e_{jkl}$ 
for the following $i,jkl$. 

\tiny
\vskip 5pt

\begin{center}

 
\begin{tabular}{|c|c|c|c|c|c|c|c|c|c|c|c|c|c|c|c|}
\hline
$i$ & $\underline{13}$ & $14$ & $\underline{19}$ 
& $\underline{20}$ & $23$  & $24$ & $26$ & $27$ \\
\hline
$jkl$ & $\underline{451}$ & $461$ & $\underline{152}$ 
& $\underline{162}$ & $252$ 
& $262$ & $352$ & $362$ \\
\hline
\end{tabular}

\end{center}
\normalsize

\vskip 5pt

Lemma \ref{lem:eliminate-2m(32)} implies that 
we may assume that $x_{451},x_{152},x_{162}=0$.

\vskip 5pt

(20) $\be_{23}=\tfrac {1} {42} (-8,-8,-8,-2,10,16,-9,9)$.  

$Z_{\be_{23}}$ is spanned by $\coorde_i=e_{jkl}$ 
for the following $i,jkl$. 
\tiny
%
%
%
\begin{tabular}{|c|c|c|c|c|c|c|c|c|c|c|c|c|c|c|c|}
\hline
$i$ & $15$ & $\underline{20}$ & $\underline{24}$ 
& $27$ & $28$ \\
\hline
$jkl$ & $561$ & $\underline{162}$ 
& $\underline{262}$ & $362$ & $452$ \\
\hline
\end{tabular}
%
%
\normalsize

\vskip 5pt

Lemma \ref{lem:eliminate-standard} implies that we may assume that 
$x_{162},x_{262}=0$. 

\vskip 5pt

(21) $\be_{24}=\tfrac {1} {15} (-2,-2,-2,1,1,4,0,0)$.  

$Z_{\be_{24}}$ is spanned by $\coorde_i=e_{jkl}$ 
for the following $i,jkl$. 

\tiny
\vskip 5pt

\begin{center}

 
\begin{tabular}{|c|c|c|c|c|c|c|c|c|c|c|c|c|c|c|c|}
\hline
$i$ & $\underline{5}$ & $9$ & $12$ 
& $\underline{13}$ & $\underline{20}$ & $24$ & $27$ & $28$ \\
\hline
$jkl$ & $\underline{161}$ & $261$ 
& $361$ & $\underline{451}$ & $\underline{162}$ & $262$ 
& $362$ & $452$ \\
\hline
\end{tabular}

\end{center}
\normalsize

\vskip 5pt

Lemma \ref{lem:eliminate-2m(32)} implies that we may assume that 
$x_{161},x_{162},x_{451}=0$. 

\vskip 5pt

(22) $\be_{25}=\tfrac {1} {102} (-16,-16,-16,-10,-10,68,-3,3)$.  

$Z_{\be_{25}}$ is spanned by $\coorde_i=e_{jkl}$ 
for the following $i,jkl$. 
\tiny
%
%
%
\begin{tabular}{|c|c|c|c|c|c|c|c|c|c|c|c|c|c|c|c|}
\hline
$i$ & $\underline{14}$ & $15$ & $\underline{20}$ 
& $\underline{24}$ & $27$ \\
\hline
$jkl$ & $\underline{461}$ & $561$ 
& $\underline{162}$ & $\underline{262}$ & $362$ \\
\hline
\end{tabular}

%
\normalsize

\vskip 5pt

We apply Lemma \ref{lem:eliminate-standard} to 
$\lan e_{461},e_{561}\ran$, $\lan e_{162},e_{262},e_{362}\ran$,  
and may assume that $x_{461},x_{162},x_{262}=0$. 

\vskip 5pt

(23) $\be_{26}=\tfrac {1} {42} (-2,-2,0,0,0,4,-3,3)$. 
%

$Z_{\be_{26}}$ is spanned by $\coorde_i=e_{jkl}$ 
for the following $i,jkl$. 

\tiny
\vskip 5pt

\begin{center}

 
\begin{tabular}{|c|c|c|c|c|c|c|c|c|c|c|c|c|c|c|c|}
\hline
$i$ & $\underline{12}$ & $\underline{14}$ & $15$ & $\underline{17}$ & $18$ 
& $19$ & $21$ & $22$ & $23$ \\
\hline
$jkl$ & $\underline{361}$ & $\underline{461}$ & $561$ & $\underline{132}$ 
& $142$ & $152$ & $232$ & $242$ & $252$ \\
\hline
\end{tabular}

\end{center}
\normalsize

\vskip 5pt

Lemma \ref{lem:eliminate-2m(3-32)} implies that we may assume that 
$x_{361},x_{461},x_{132}=0$. 

\vskip 5pt

(24) $\be_{27}=\tfrac {1} {102} (-10,-10,2,2,2,14,-51,51)$.  

$Z_{\be_{27}}$ is spanned by $\coorde_i=e_{jkl}$ 
for the following $i,jkl$. 
\tiny
%
%
%
\begin{tabular}{|c|c|c|c|c|c|c|c|c|c|c|c|c|c|c|c|}
\hline
$i$ & $\underline{20}$ & $24$ & $\underline{25}$ 
& $\underline{26}$ & $28$ \\
\hline
$jkl$ & $\underline{162}$ & $262$ 
& $\underline{342}$ & $\underline{352}$ & $452$ \\
\hline
\end{tabular}

%
\normalsize

\vskip 5pt

We apply Lemma \ref{lem:eliminate-standard}
(resp. Lemma \ref{lem:alernating-matrix}) to 
$\lan e_{162},e_{262}\ran$ 
(resp. $\lan e_{342},e_{352},e_{452}\ran$) 
and may assume that $x_{162},x_{342},x_{352}=0$. 

\vskip 5pt

(25) $\be_{28}=\tfrac {1} {222} (-44,-2,-2,-2,10,40,-27,27)$.  

$Z_{\be_{23}}$ is spanned by $\coorde_i=e_{jkl}$ 
for the following $i,jkl$. 
\tiny
%
%
%
\begin{tabular}{|c|c|c|c|c|c|c|c|c|c|c|c|c|c|c|c|}
\hline
$i$ & $15$ & $20$ & $\underline{21}$ 
& $\underline{22}$ & $25$ \\
\hline
$jkl$ & $561$ & $162$ 
& $\underline{232}$ & $\underline{242}$ & $342$ \\
\hline
\end{tabular}

%
\normalsize

\vskip 5pt

Lemma \ref{lem:alernating-matrix} implies that 
we may assume that $x_{232},x_{242}=0$.  

\vskip 5pt

(26) $\be_{29}=\tfrac {1} {78} (-26,4,4,4,4,10,-3,3)$.  
%

$Z_{\be_{29}}$ is spanned by $\coorde_i=e_{jkl}$ 
for the following $i,jkl$. 

\tiny
\vskip 5pt

\begin{center}

 
\begin{tabular}{|c|c|c|c|c|c|c|c|c|c|c|c|c|c|c|c|}
\hline
$i$ & $\underline{9}$ & $\underline{12}$ & $\underline{14}$ 
& $15$ & $\underline{21}$ 
& $\underline{22}$ & $23$ & $25$ & $26$ & $28$ \\
\hline
$jkl$ & $\underline{261}$ & $\underline{361}$ & $\underline{461}$ & $561$
& $\underline{232}$ 
& $\underline{242}$ & $252$ & $342$ & $352$ & $452$ \\
\hline
\end{tabular}

\end{center}
\normalsize

\vskip 5pt

Lemma \ref{lem:eliminate-standard} implies that 
we may assume that $x_{261},x_{361},x_{461}=0$. 
Matrices of the form $\diag(1,g_1,I_2)$ with 
$g_1\in \gl_3$ do not change this condition. 
Then Lemma \ref{lem:alernating-matrix} implies that
we may assume that $x_{232},x_{242}=0$. 

\vskip 5pt

(27) $\be_{30}=\tfrac {1} {48} (-4,-1,-1,2,2,2,-3,3)$.  

$Z_{\be_{30}}$ is spanned by $\coorde_i=e_{jkl}$ 
for the following $i,jkl$. 

\tiny
\vskip 5pt

\begin{center}

 
\begin{tabular}{|c|c|c|c|c|c|c|c|c|c|c|c|c|c|c|c|}
\hline
$i$ & $\underline{13}$ & $\underline{14}$ & $15$ 
& $18$ & $19$ & $20$ & $21$ \\
\hline
$jkl$ & $\underline{451}$ & $\underline{461}$ 
& $561$ & $142$ & $152$ & $162$ & $232$ \\
\hline
\end{tabular}

\end{center}
\normalsize

\vskip 5pt

Lemma \ref{lem:alernating-matrix} implies that 
we may assume that $x_{451},x_{461}=0$.

\vskip 5pt

(28) $\be_{31}=\tfrac {1} {174} (-16,-16,2,2,2,26,-3,3)$. 

$Z_{\be_{31}}$ is spanned by $\coorde_i=e_{jkl}$ 
for the following $i,jkl$. 
\tiny
%
%
%
\begin{tabular}{|c|c|c|c|c|c|c|c|c|c|c|c|c|c|c|c|}
\hline
$i$ & $\underline{5}$ & $9$ & $\underline{25}$ 
& $\underline{26}$ & $28$ \\
\hline
$jkl$ & $\underline{161}$ & $261$ 
& $\underline{342}$ & $\underline{352}$ & $452$ \\
\hline
\end{tabular}

%
\normalsize

\vskip 5pt

We apply Lemma \ref{lem:eliminate-standard} 
(resp. Lemma \ref{lem:alernating-matrix}) 
to $\lan e_{161},e_{261}\ran$, $\lan e_{342},e_{352},e_{452}\ran$
and may assume that $x_{161},x_{342},x_{352}=0$.  

\vskip 5pt

(29) $\be_{32}=\tfrac {1} {30} (-10,2,2,2,2,2,-15,15)$.  

$Z_{\be_{32}}$ is spanned by $\coorde_i=e_{jkl}$ 
for the following $i,jkl$. 

\tiny
\vskip 5pt

\begin{center}

 
\begin{tabular}{|c|c|c|c|c|c|c|c|c|c|c|c|c|c|c|c|}
\hline
$i$ & $\underline{21}$ & $\underline{22}$ & $\underline{23}$ 
& $\underline{24}$ & $25$ & $26$ & $27$ & $28$ & $29$ & $30$ \\
\hline
$jkl$ & $\underline{232}$ & $\underline{242}$ 
& $\underline{252}$ & $\underline{262}$ & $342$ 
& $352$ & $362$ & $452$ & $462$ & $562$ \\
\hline
\end{tabular}

\end{center}
\normalsize

\vskip 5pt

Lemma \ref{lem:alernating-matrix} implies that we may assume that 
$x_{232},x_{242},x_{252},x_{262}=0$. 

\vskip 5pt

(30) $\be_{33}=\tfrac {1} {102} (-34,-10,-10,-10,32,32,-21,21)$. 

$Z_{\be_{33}}$ is spanned by $\coorde_i=e_{jkl}$ 
for the following $i,jkl$. 

\tiny
\vskip 5pt

\begin{center}

 
\begin{tabular}{|c|c|c|c|c|c|c|c|c|c|c|c|c|c|c|c|}
\hline
$i$ & $15$ & $\underline{23}$ & $\underline{24}$ 
& $26$ & $27$ & $28$ & $29$ \\
\hline
$jkl$ & $561$ & $\underline{252}$ 
& $\underline{262}$ & $352$ & $362$ & $452$ & $462$ \\
\hline
\end{tabular}

\end{center}
\normalsize

\vskip 5pt

Lemma \ref{lem:eliminate-nm-matrix} implies that 
we may assume that $x_{252},x_{262}=0$. 

\vskip 5pt
 
(31) $\be_{34}=\tfrac {1} {22} (-4,-2,-2,2,2,4,-3,3)$.  

$Z_{\be_{34}}$ is spanned by $\coorde_i=e_{jkl}$ 
for the following $i,jkl$. 

\tiny
\vskip 5pt

\begin{center}

 
\begin{tabular}{|c|c|c|c|c|c|c|c|c|c|c|c|c|c|c|c|}
\hline
$i$ & $\underline{14}$ & $15$ & $20$ 
& $\underline{22}$ & $23$ & $25$ & $26$ \\
\hline
$jkl$ & $\underline{461}$ & $561$ 
& $162$ & $\underline{242}$ & $252$ & $342$ & $352$ \\
\hline
\end{tabular}

\end{center}
\normalsize

\vskip 5pt

Lemma \ref{lem:eliminate-2m(2)} implies that we may assume that 
$x_{461}=x_{242}=0$. 

\vskip 5pt

(32) $\be_{36}=\tfrac {1} {66} (-22,2,4,4,6,6,-1,1)$.  

$Z_{\be_{36}}$ is spanned by $\coorde_i=e_{jkl}$ 
for the following $i,jkl$. 

\tiny
\vskip 5pt

\begin{center}

 
\begin{tabular}{|c|c|c|c|c|c|c|c|c|c|c|c|c|c|c|c|}
\hline
$i$ & $\underline{11}$ & $12$ & $13$ 
& $14$ & $\underline{23}$ & $24$ & $25$ \\
\hline
$jkl$ & $\underline{351}$ & $361$ 
& $451$ & $461$ & $\underline{252}$ 
& $262$ & $342$ \\
\hline
\end{tabular}

\end{center}
\normalsize

\vskip 5pt

Lemma \ref{lem:eliminate-2m(2)} implies that we may assume that 
$x_{351},x_{252}=0$. 

\vskip 5pt

(33) $\be_{37}=\tfrac {5} {66} (-2,-2,-2,-2,4,4,-3,3)$.  

$Z_{\be_{37}}$ is spanned by $\coorde_i=e_{jkl}$ 
for the following $i,jkl$. 

\tiny
\vskip 5pt

\begin{center}

 
\begin{tabular}{|c|c|c|c|c|c|c|c|c|c|c|c|c|c|c|c|}
\hline
$i$ & $15$ & $\underline{19}$ & $\underline{20}$ 
& $\underline{23}$ & $\underline{24}$ & $26$ & $27$ & $28$ & $29$ \\
\hline
$jkl$ & $561$ & $\underline{152}$ 
& $\underline{162}$ & $\underline{252}$ 
& $\underline{262}$ & $352$ & $362$ & $452$ & $462$ \\
\hline
\end{tabular}

\end{center}
\normalsize

\vskip 5pt

Lemma \ref{lem:eliminate-nm-matrix} implies that we may assume that
$x_{152},x_{162},x_{252},x_{262}=0$. 

\vskip 5pt

(34)  $\be_{38}=\tfrac {1} {246} (-34,-28,-28,26,32,32,-33,33)$.  

$Z_{\be_{38}}$ is spanned by $\coorde_i=e_{jkl}$ 
for the following $i,jkl$. 
\tiny
%
%
%
\begin{tabular}{|c|c|c|c|c|c|c|c|c|c|c|c|c|c|c|c|}
\hline
$i$ & $15$ & $19$ & $20$ 
& $\underline{22}$ & $25$ \\
\hline
$jkl$ & $561$ & $152$ 
& $162$ & $\underline{242}$ & $342$ \\
\hline
\end{tabular}

%
\normalsize

\vskip 5pt

Lemma \ref{lem:eliminate-standard} implies that
we may assume that $x_{242}=0$.  

\vskip 5pt

(35) $\be_{39}=\tfrac {1} {66} (-22,-10,-10,8,8,26,-9,9)$. 

$Z_{\be_{39}}$ is spanned by $\coorde_i=e_{jkl}$ 
for the following $i,jkl$. 
\tiny
%
%
%
\begin{tabular}{|c|c|c|c|c|c|c|c|c|c|c|c|c|c|c|c|}
\hline
$i$ & $\underline{14}$ & $15$ & $\underline{24}$ 
& $27$ & $28$ \\
\hline
$jkl$ & $\underline{461}$ & $561$ 
& $\underline{262}$ & $362$ & $452$ \\
\hline
\end{tabular}

%
\normalsize

\vskip 5pt

We apply Lemma \ref{lem:eliminate-standard} to 
$\lan e_{461},e_{561}\ran$, $\lan e_{262},e_{362}\ran$
and may assume that $x_{461},x_{262}=0$. 

\vskip 5pt

(36) $\be_{40}=\tfrac {1} {114} (-38,-2,4,10,10,16,-3,3)$.  

$Z_{\be_{40}}$ is spanned by $\coorde_i=e_{jkl}$ 
for the following $i,jkl$. 
\tiny
%
%
%
\begin{tabular}{|c|c|c|c|c|c|c|c|c|c|c|c|c|c|c|c|}
\hline
$i$ & $12$ & $13$ & $24$ 
& $\underline{25}$ & $26$ \\
\hline
$jkl$ & $361$ & $451$ 
& $262$ & $\underline{342}$ & $352$ \\
\hline
\end{tabular}

%
\normalsize

\vskip 5pt

Lemma \ref{lem:eliminate-standard} implies that
we may assume that $x_{342}=0$. 

\vskip 5pt

(37) $\be_{41}=\tfrac {1} {30} (-4,-4,-2,0,4,6,-1,1)$. 

$Z_{\be_{41}}$ is spanned by $\coorde_i=e_{jkl}$ 
for the following $i,jkl$. 
\tiny
%
%
%
\begin{tabular}{|c|c|c|c|c|c|c|c|c|c|c|c|c|c|c|c|}
\hline
$i$ & $12$ & $13$ & $\underline{20}$ 
& $24$ & $26$ \\
\hline
$jkl$ & $361$ & $451$ 
& $\underline{162}$ & $262$ & $352$ \\
\hline
\end{tabular}

%
\normalsize

\vskip 5pt

Lemma \ref{lem:eliminate-standard} 
implies that we may assume that $x_{162}=0$. 

\vskip 5pt

(38) $\be_{42}=\tfrac {1} {102} (-7,-1,-1,2,2,5,-3,3)$. 

$Z_{\be_{42}}$ is spanned by $\coorde_i=e_{jkl}$ 
for the following $i,jkl$. 
\tiny
%
%
%
\begin{tabular}{|c|c|c|c|c|c|c|c|c|c|c|c|c|c|c|c|}
\hline
$i$ & $\underline{9}$ & $12$ & $13$ 
& $20$ & $21$ \\
\hline
$jkl$ & $\underline{261}$ & $361$ 
& $451$ & $162$ & $232$ \\
\hline
\end{tabular}

%
\normalsize

\vskip 5pt

Lemma \ref{lem:eliminate-standard} implies that we may assume that 
$x_{261}=0$. 

\vskip 5pt

(39) $\be_{43}=\tfrac {1} {114} (-14,-8,-2,4,4,16,-3,3)$. 

$Z_{\be_{43}}$ is spanned by $\coorde_i=e_{jkl}$ 
for the following $i,jkl$. 
\tiny
%
%
%
\begin{tabular}{|c|c|c|c|c|c|c|c|c|c|c|c|c|c|c|c|}
\hline
$i$ & $9$ & $13$ & $20$ 
& $\underline{25}$ & $26$ \\
\hline
$jkl$ & $261$ & $451$ 
& $162$ & $\underline{342}$ & $352$ \\
\hline
\end{tabular}

%
\normalsize

\vskip 5pt

Lemma \ref{lem:eliminate-standard} implies that 
 we may assume that $x_{342}=0$. 

\vskip 5pt

(40) $\be_{44}=\tfrac {1} {42} (-8,-2,-2,-2,4,10,-3,3)$.  

$Z_{\be_{44}}$ is spanned by $\coorde_i=e_{jkl}$ 
for the following $i,jkl$. 

\tiny
\vskip 5pt

\begin{center}

 
\begin{tabular}{|c|c|c|c|c|c|c|c|c|c|c|c|c|c|c|c|}
\hline
$i$ & $\underline{9}$ & $\underline{12}$ & $14$ 
& $20$ & $23$ & $26$ & $28$ \\
\hline
$jkl$ & $\underline{261}$ & $\underline{361}$ 
& $461$ & $162$ & $252$ & $352$ & $452$ \\
\hline
\end{tabular}

\end{center}
\normalsize


We apply Lemma \ref{lem:eliminate-standard} to 
$\lan e_{261},e_{361},e_{461}\ran$ and  
may assume that $x_{261},x_{361}=0$. 

\vskip 5pt

(41) $\be_{45}=\tfrac {1} {114} (-6,-2,-2,0,4,6,-3,3)$.  

$Z_{\be_{45}}$ is spanned by $\coorde_i=e_{jkl}$ 
for the following $i,jkl$. 
\tiny
%
%
%
\begin{tabular}{|c|c|c|c|c|c|c|c|c|c|c|c|c|c|c|c|}
\hline
$i$ & $\underline{9}$ & $12$ & $13$ 
& $19$ & $22$ & $25$ \\
\hline
$jkl$ & $\underline{261}$ & $361$ 
& $451$ & $152$ & $242$ & $342$ \\
\hline
\end{tabular}

%
\normalsize

\vskip 5pt

We apply Lemma \ref{lem:eliminate-standard}
to $\lan e_{261},e_{361}\ran$
and may assume that $x_{261}=$0. 

\vskip 5pt

(42) $\be_{47}=\tfrac {1} {12} (-4,-1,-1,-1,-1,8,-6,6)$.  

$Z_{\be_{47}}$ is spanned by $\coorde_i=e_{jkl}$ 
for the following $i,jkl$. 
\tiny
%
%
%
\begin{tabular}{|c|c|c|c|c|c|c|c|c|c|c|c|c|c|c|c|}
\hline
$i$ & $\underline{24}$ & $\underline{27}$ & $\underline{29}$ 
& $30$ \\
\hline
$jkl$ & $\underline{262}$ & $\underline{362}$ 
& $\underline{462}$ & $562$ \\
\hline
\end{tabular}

%
\normalsize

\vskip 5pt
Lemma \ref{lem:eliminate-standard} 
implies that we may assume that 
$x_{262},x_{362},x_{462}=0$. 

\vskip 5pt
 
(43) $\be_{48}=\tfrac {1} {30} (-4,-4,-4,2,2,8,-15,15)$.  

$Z_{\be_{48}}$ is spanned by $\coorde_i=e_{jkl}$ 
for the following $i,jkl$. 
\tiny
%
%
%
\begin{tabular}{|c|c|c|c|c|c|c|c|c|c|c|c|c|c|c|c|}
\hline
$i$ & $\underline{20}$ & $\underline{24}$ & $27$ 
& $28$ \\
\hline
$jkl$ & $\underline{162}$ & $\underline{262}$ 
& $362$ & $452$ \\
\hline
\end{tabular}

%
\normalsize

\vskip 5pt
Lemma \ref{lem:eliminate-standard} implies that
we may assume that $x_{162},x_{262}=0$.

\vskip 5pt

(44) $\be_{49}=\tfrac {1} {30} (-10,-4,-4,-4,2,20,-3,3)$.  

$Z_{\be_{49}}$ is spanned by $\coorde_i=e_{jkl}$ 
for the following $i,jkl$. 
\tiny
%
%
%
\begin{tabular}{|c|c|c|c|c|c|c|c|c|c|c|c|c|c|c|c|}
\hline
$i$ & $15$ & $\underline{24}$ & $\underline{27}$ 
& $29$ \\
\hline
$jkl$ & $561$ & $\underline{262}$ 
& $\underline{362}$ & $462$ \\
\hline
\end{tabular}

%
\normalsize

\vskip 5pt

Lemma \ref{lem:eliminate-standard} implies that
we may assume that $x_{262},x_{362}=0$.

\vskip 5pt

(45) $\be_{50}=\tfrac {1} {48} (-7,-7,-7,5,5,11,-3,3)$. 

$Z_{\be_{50}}$ is spanned by $\coorde_i=e_{jkl}$ 
for the following $i,jkl$. 
\tiny
%
%
%
\begin{tabular}{|c|c|c|c|c|c|c|c|c|c|c|c|c|c|c|c|}
\hline
$i$ & $13$ & $\underline{20}$ & $\underline{24}$ 
& $27$ \\
\hline
$jkl$ & $451$ & $\underline{162}$ 
& $\underline{262}$ & $362$ \\
\hline
\end{tabular}

%
\normalsize

\vskip 5pt

Lemma \ref{lem:eliminate-standard} implies that 
we may assume that $x_{162},x_{262}=0$.

\vskip 5pt

(46) $\be_{51}=\tfrac {1} {48} (-16,2,2,2,5,5,-3,3)$. 

$Z_{\be_{51}}$ is spanned by $\coorde_i=e_{jkl}$ 
for the following $i,jkl$. 
\tiny
%
%
%
\begin{tabular}{|c|c|c|c|c|c|c|c|c|c|c|c|c|c|c|c|}
\hline
$i$ & $15$ & $\underline{21}$ & $\underline{22}$ 
& $25$ \\
\hline
$jkl$ & $561$ 
& $\underline{232}$ & $\underline{242}$ & $342$ \\
\hline
\end{tabular}

%
\normalsize

\vskip 5pt

Lemma \ref{lem:alernating-matrix} implies that
we may assume that $x_{232},x_{242}=0$. 

\vskip 5pt

(47) $\be_{52}=\tfrac {1} {6} (-2,0,0,0,1,1,-3,3)$.  

$Z_{\be_{52}}$ is spanned by $\coorde_i=e_{jkl}$ 
for the following $i,jkl$. 
\tiny
%
%
%
\begin{tabular}{|c|c|c|c|c|c|c|c|c|c|c|c|c|c|c|c|}
\hline
$i$ & $\underline{23}$ & $\underline{24}$ & $26$ 
& $27$ & $28$ & $29$ \\
\hline
$jkl$ & $\underline{252}$ & $\underline{262}$ 
& $352$ & $362$ & $452$ & $462$ \\
\hline
\end{tabular}

%
\normalsize

\vskip 5pt

Lemma \ref{lem:eliminate-nm-matrix} implies that
we may assume that $x_{252},x_{262}=0$. 

\vskip 5pt

(48) $\be_{53}=\tfrac {1} {42} (-14,-2,-2,4,7,7,-3,3)$. 

$Z_{\be_{53}}$ is spanned by $\coorde_i=e_{jkl}$ 
for the following $i,jkl$. 
\tiny
%
%
%
\begin{tabular}{|c|c|c|c|c|c|c|c|c|c|c|c|c|c|c|c|}
\hline
$i$ & $\underline{13}$ & $14$ & $\underline{23}$ 
& $24$ & $26$ & $27$ \\
\hline
$jkl$ & $\underline{451}$ & $461$ 
& $\underline{252}$ & $262$ & $352$ & $362$ \\
\hline
\end{tabular}

%
\normalsize

\vskip 5pt

Lemma \ref{lem:eliminate-2m(2)} implies that
we may assume that $x_{451},x_{252}=0$. 

\vskip 5pt

(49) $\be_{54}=\tfrac {1} {12} (-1,-1,-1,-1,2,2,-6,6)$.  

$Z_{\be_{54}}$ is spanned by $\coorde_i=e_{jkl}$ 
for the following $i,jkl$. 

\tiny
\vskip 5pt

\begin{center}

 
\begin{tabular}{|c|c|c|c|c|c|c|c|c|c|c|c|c|c|c|c|}
\hline
$i$ & $\underline{19}$ & $\underline{20}$ & $\underline{23}$ 
& $\underline{24}$ & $26$ & $27$ & $28$ & $29$ \\
\hline
$jkl$ & $\underline{152}$ & $\underline{162}$ 
& $\underline{252}$ & $\underline{262}$ & $352$ 
& $362$ & $452$ & $462$ \\
\hline
\end{tabular}

\end{center}
\normalsize

\vskip 5pt

Lemma \ref{lem:eliminate-nm-matrix} implies that
we may assume that $x_{152},x_{162},x_{252},x_{262}=0$. 

\vskip 5pt

(50) $\be_{55}=\tfrac {1} {30} (-10,-4,-4,-1,8,11,-6,6)$. 

$Z_{\be_{55}}$ is spanned by $\coorde_i=e_{jkl}$ 
for the following $i,jkl$. 
\tiny
%
%
%
\begin{tabular}{|c|c|c|c|c|c|c|c|c|c|c|c|c|c|c|c|}
\hline
$i$ & $15$ & $\underline{24}$ 
& $27$ & $28$ \\
\hline
$jkl$ & $561$ 
& $\underline{262}$ & $362$ & $452$ \\
\hline
\end{tabular}

%
\normalsize

\vskip 5pt

Lemma \ref{lem:eliminate-standard} implies that
we may assume that $x_{262}=0$. 

\vskip 5pt

(51) $\be_{56}=\tfrac {1} {24} (-5,-5,1,1,1,7,-3,3)$.  

$Z_{\be_{56}}$ is spanned by $\coorde_i=e_{jkl}$ 
for the following $i,jkl$. 

\tiny
\vskip 5pt

\begin{center}

 
\begin{tabular}{|c|c|c|c|c|c|c|c|c|c|c|c|c|c|c|c|}
\hline
$i$ & $\underline{12}$ & $\underline{14}$ & $15$ 
& $\underline{20}$ & $24$ & $25$ & $26$ & $28$ \\
\hline
$jkl$ & $\underline{361}$ & $\underline{461}$ 
& $561$ & $\underline{162}$ & $262$ 
& $342$ & $352$ & $452$ \\
\hline
\end{tabular}

\end{center}
\normalsize

\vskip 5pt

We apply Lemma \ref{lem:eliminate-standard} to
$\lan e_{361},e_{461},e_{561}\ran$, 
$\lan e_{162},e_{262}\ran$
and may assume that 
$x_{361},x_{461}$, $x_{162}=0$. 

\vskip 5pt

(52) $\be_{57}=\tfrac {1} {24} (-8,-2,1,1,4,4,-3,3)$.  

$Z_{\be_{57}}$ is spanned by $\coorde_i=e_{jkl}$ 
for the following $i,jkl$. 
\tiny
%
%
%
\begin{tabular}{|c|c|c|c|c|c|c|c|c|c|c|c|c|c|c|c|}
\hline
$i$ & $15$ & $\underline{23}$ & $24$ 
& $25$ \\
\hline
$jkl$ & $561$ & $\underline{252}$ 
& $262$ & $342$ \\
\hline
\end{tabular}

%
\normalsize

\vskip 5pt

Lemma \ref{lem:eliminate-standard} implies that
we may assume that $x_{252}=0$. 

\vskip 5pt

(53) $\be_{58}=\tfrac {1} {78} (-14,-8,-8,4,10,16,-9,9)$.  

$Z_{\be_{58}}$ is spanned by $\coorde_i=e_{jkl}$ 
for the following $i,jkl$. 
\tiny
%
%
%
\begin{tabular}{|c|c|c|c|c|c|c|c|c|c|c|c|c|c|c|c|}
\hline
$i$ & $14$ & $20$ 
& $\underline{23}$ & $26$ \\
\hline
$jkl$ & $461$ & $162$ & $\underline{252}$ & $352$ \\
\hline
\end{tabular}

%
\normalsize

\vskip 5pt

Lemma \ref{lem:eliminate-standard} implies that
we may assume that $x_{252}=0$. 

\vskip 5pt

(54) $\be_{59}=\tfrac {1} {3} (-1,-1,0,0,0,2,0,0)$. 

$Z_{\be_{59}}$ is spanned by $\coorde_i=e_{jkl}$ 
for the following $i,jkl$. 
\tiny
%
%
%
\begin{tabular}{|c|c|c|c|c|c|c|c|c|c|c|c|c|c|c|c|}
\hline
$i$ & $\underline{12}$ & $14$ & $15$ 
& $\underline{27}$ & $29$ & $30$ \\
\hline
$jkl$ & $\underline{361}$ & $461$ 
& $561$ & $\underline{362}$ & $462$ 
& $562$ \\
\hline
\end{tabular}

%
\normalsize

\vskip 5pt

Lemma \ref{lem:eliminate-nm-matrix} implies that
we may assume that $x_{361},x_{362}=0$. 

\vskip 5pt
 
(55) $\be_{60}=\tfrac {1} {3} (-1,-1,-1,1,1,1,0,0)$.  

$Z_{\be_{60}}$ is spanned by $\coorde_i=e_{jkl}$ 
for the following $i,jkl$. 
\tiny
%
%
%
\begin{tabular}{|c|c|c|c|c|c|c|c|c|c|c|c|c|c|c|c|}
\hline
$i$ & $\underline{13}$ & $14$ & $15$ 
& $\underline{28}$ & $29$ & $30$ \\
\hline
$jkl$ & $\underline{451}$ & $461$ 
& $561$ & $\underline{452}$ & $462$ 
& $562$ \\ 
\hline
\end{tabular}

%
\normalsize

\vskip 5pt

Similarly as in the case (11),  we may assume that 
$x_{451},x_{452}=0$. 

\vskip 5pt 

(56) $\be_{61}=\tfrac {1} {21} (-7,-1,-1,2,2,5,0,0)$. 

$Z_{\be_{61}}$ is spanned by $\coorde_i=e_{jkl}$ 
for the following $i,jkl$. 
\tiny
%
%
%
\begin{tabular}{|c|c|c|c|c|c|c|c|c|c|c|c|c|c|c|c|}
\hline
$i$ & $\underline{9}$ & $12$ & $\underline{13}$ 
& $24$ & $27$ & $28$ \\
\hline
$jkl$ & $\underline{261}$ & $361$ 
& $\underline{451}$ & $262$ 
& $362$ & $452$ \\
\hline
\end{tabular}

%
\normalsize

\vskip 5pt

Lemma \ref{lem:eliminate-2m(2)} implies that
we may assume that $x_{261},x_{451}=0$. 

\vskip 5pt

(57) $\be_{62}=\tfrac {1} {12} (-4,-1,-1,-1,-1,8,0,0)$. 

$Z_{\be_{62}}$ is spanned by $\coorde_i=e_{jkl}$ 
for the following $i,jkl$. 

\tiny
\vskip 5pt

\begin{center}

 
\begin{tabular}{|c|c|c|c|c|c|c|c|c|c|c|c|c|c|c|c|}
\hline
$i$ & $\underline{9}$ & $\underline{12}$ & $14$ 
& $15$ & $\underline{24}$ & $\underline{27}$ & $29$ & $30$ \\
\hline
$jkl$ & $\underline{261}$ & $\underline{361}$ 
& $461$ & $561$ & $\underline{262}$ & $\underline{362}$ 
& $462$ & $562$ \\ 
\hline
\end{tabular}

\end{center}
\normalsize

\vskip 5pt

Lemma \ref{lem:eliminate-nm-matrix} implies that
we may assume that $x_{261},x_{361},x_{262},x_{362}=0$. 

\vskip 5pt

(58) $\be_{63}=\tfrac {1} {30} (-10,-10,-1,5,5,11,-3,3)$. 

$Z_{\be_{63}}$ is spanned by $\coorde_i=e_{jkl}$ 
for the following $i,jkl$. 
\tiny
%
%
%
\begin{tabular}{|c|c|c|c|c|c|c|c|c|c|c|c|c|c|c|c|}
\hline
$i$ & $\underline{14}$ & $15$ & $27$ & $28$ \\ 
\hline
$jkl$ & $\underline{461}$ & $561$ 
& $362$ & $452$ \\ 
\hline
\end{tabular}

%
\normalsize

\vskip 5pt

Lemma \ref{lem:eliminate-standard} implies that 
we may assume that $x_{461}=0$. 

\vskip 5pt
 
(59) $\be_{64}=\tfrac {1} {78} (-14,-8,-8,4,4,22,-3,3)$. 

$Z_{\be_{64}}$ is spanned by $\coorde_i=e_{jkl}$ 
for the following $i,jkl$. 
\tiny
%
%
%
\begin{tabular}{|c|c|c|c|c|c|c|c|c|c|c|c|c|c|c|c|}
\hline
$i$ & $\underline{9}$ & $12$ & $20$ & $28$ \\
\hline
$jkl$ & $\underline{261}$ & $361$ & $162$ & $452$ \\ 
\hline
\end{tabular}

%
\normalsize

\vskip 5pt

Lemma \ref{lem:eliminate-standard} implies that 
we may assume that $x_{261}=0$. 

\vskip 5pt

(60) $\be_{65}=\tfrac {1} {78} (-5,-2,-2,1,1,7,-6,6)$.  

$Z_{\be_{65}}$ is spanned by $\coorde_i=e_{jkl}$ 
for the following $i,jkl$. 
\tiny
%
%
%
\begin{tabular}{|c|c|c|c|c|c|c|c|c|c|c|c|c|c|c|c|}
\hline
$i$ & $\underline{14}$ & $15$ & $18$ 
& $19$ & $21$ \\
\hline
$jkl$ & $\underline{461}$ & $561$ 
& $142$ & $152$ & $232$ \\ 
\hline
\end{tabular}

%
\normalsize

\vskip 5pt

We apply Lemma \ref{lem:eliminate-standard}
to $\lan e_{461},e_{561}\ran$ and may assume that 
$x_{461}=0$. 

\vskip 5pt

(61) $\be_{68}=\tfrac {1} {6} (-2,-2,0,0,0,4,-3,3)$.  

$Z_{\be_{68}}$ is spanned by $\coorde_i=e_{jkl}$ 
for the following $i,jkl$. 
\tiny
%
%
%
\begin{tabular}{|c|c|c|c|c|c|c|c|c|c|c|c|c|c|c|c|}
\hline
$i$ & $\underline{27}$ & $\underline{29}$ & $30$ \\
\hline
$jkl$ & $\underline{362}$ & $\underline{462}$ & $562$ \\ 
\hline
\end{tabular}

%
\normalsize

\vskip 5pt

Lemma \ref{lem:eliminate-standard} implies that 
we may assume that $x_{362},x_{462}=0$. 

\vskip 5pt

(62) $\be_{69}=\tfrac {1} {6} (-2,-2,-2,2,2,2,-3,3)$. 

$Z_{\be_{69}}$ is spanned by $\coorde_i=e_{jkl}$ 
for the following $i,jkl$. 
\tiny
%
%
%
\begin{tabular}{|c|c|c|c|c|c|c|c|c|c|c|c|c|c|c|c|}
\hline
$i$ & $\underline{28}$ & $\underline{29}$ & $30$ \\
\hline
$jkl$ & $\underline{452}$ & $\underline{462}$ & $562$ \\ 
\hline
\end{tabular}

%
\normalsize

\vskip 5pt

Lemma \ref{lem:alernating-matrix} implies that 
we may assume that $x_{452},x_{462}=0$.

\vskip 5pt

(63) $\be_{70}=\tfrac {1} {42} (-14,-2,-2,4,4,10,-21,21)$. 

$Z_{\be_{70}}$ is spanned by $\coorde_i=e_{jkl}$ 
for the following $i,jkl$. 
\tiny
%
%
%
\begin{tabular}{|c|c|c|c|c|c|c|c|c|c|c|c|c|c|c|c|}
\hline
$i$ & $\underline{24}$ & $27$ & $28$ \\
\hline
$jkl$ & $\underline{262}$ & $362$ & $452$ \\
\hline
\end{tabular}

%
\normalsize

\vskip 5pt

Lemma \ref{lem:eliminate-standard} implies that 
we may assume that $x_{262}=0$. 

\vskip 5pt
 
(64) $\be_{71}=\tfrac {1} {42} (-14,-14,-2,-2,4,28,-3,3)$.  

$Z_{\be_{71}}$ is spanned by $\coorde_i=e_{jkl}$ 
for the following $i,jkl$. 
\tiny
%
%
%
\begin{tabular}{|c|c|c|c|c|c|c|c|c|c|c|c|c|c|c|c|}
\hline
$i$ & $15$ & $\underline{27}$ & $29$ \\
\hline
$jkl$ & $561$ & $\underline{362}$ & $462$ \\ 
\hline
\end{tabular}

%
\normalsize

\vskip 5pt

Lemma \ref{lem:eliminate-standard} implies that 
we may assume that $x_{362}=0$. 

\vskip 5pt
 
(65) $\be_{72}=\tfrac {1} {42} (-14,-14,-14,10,16,16,-3,3)$.  

$Z_{\be_{72}}$ is spanned by $\coorde_i=e_{jkl}$ 
for the following $i,jkl$. 
\tiny
%
%
%
\begin{tabular}{|c|c|c|c|c|c|c|c|c|c|c|c|c|c|c|c|}
\hline
$i$ & $15$ & $\underline{28}$ & $29$ \\
\hline
$jkl$ & $561$ & $\underline{452}$ & $462$ \\ 
\hline
\end{tabular}

%
\normalsize

\vskip 5pt

Lemma \ref{lem:eliminate-standard} implies that 
we may assume that $x_{452}=0$. 

\vskip 5pt
 
(66) $\be_{73}=\tfrac {1} {66} (-22,-4,-4,8,8,14,-3,3)$.  

$Z_{\be_{73}}$ is spanned by $\coorde_i=e_{jkl}$ 
for the following $i,jkl$. 
\tiny
%
%
%
\begin{tabular}{|c|c|c|c|c|c|c|c|c|c|c|c|c|c|c|c|}
\hline
$i$ & $13$ & $\underline{24}$ & $27$ \\
\hline
$jkl$ & $451$ & $\underline{262}$ & $362$ \\ 
\hline
\end{tabular}

%
\normalsize

\vskip 5pt

Lemma \ref{lem:eliminate-standard} implies that 
we may assume that $x_{262}=0$. 

\vskip 5pt

(67) $\be_{77}=\tfrac {1} {6} (-2,-2,-2,1,1,4,-3,3)$. 

$Z_{\be_{77}}$ is spanned by $\coorde_i=e_{jkl}$ 
for the following $i,jkl$. 
\tiny
%
%
%
\begin{tabular}{|c|c|c|c|c|c|c|c|c|c|c|c|c|c|c|c|}
\hline
$i$ & $\underline{29}$ & $30$ \\
\hline
$jkl$ & $\underline{462}$ & $562$ \\
\hline
\end{tabular}

%
\normalsize

\vskip 5pt

Lemma \ref{lem:eliminate-standard} implies that 
we may assume that $x_{462}=0$. 

\vskip 5pt

(68) $\be_{79}=\tfrac {1} {3} (-1,-1,-1,-1,2,2,0,0)$. 

$Z_{\be_{79}}$ is spanned by $\coorde_i=e_{jkl}$ 
for the following $i,jkl$. 
\tiny
%
%
%
\begin{tabular}{|c|c|c|c|c|c|c|c|c|c|c|c|c|c|c|c|}
\hline
$i$ & $\underline{15}$ & $30$ \\
\hline
$jkl$ & $\underline{561}$ & $562$ \\ 
\hline
\end{tabular}

%
\normalsize

\vskip 5pt

Lemma \ref{lem:eliminate-standard} implies that 
we may assume that $x_{561}=0$.

\bibliographystyle{plain} 
\bibliography{ref4} 

\end{document}